\documentclass{amsart}
\usepackage[all]{xy}
\usepackage{latexsym}
\usepackage{amsmath}
\usepackage{amssymb,amscd}
\usepackage{setspace}
\usepackage{mathdots}
\usepackage[mathscr]{eucal}
\usepackage{bbm}
\usepackage{stmaryrd}
\usepackage{color}
\usepackage[colorlinks,linkcolor=blue]{hyperref}
\usepackage{hyperref}
\usepackage{enumitem}

\newtheorem{theorem}{Theorem}[section] 
\newtheorem{lemma}[theorem]{Lemma}     
 \newtheorem{corollary}[theorem]{Corollary}
\newtheorem{proposition}[theorem]{Proposition}
\newtheorem{remark}[theorem]{Remark}
\newtheorem{definition}[theorem]{Definition}

\numberwithin{equation}{section}

\newcommand{\unSym}{\underline{\Sym}} 
\newcommand{\tw}{\mathrm{tw}}

\newcommand{\fa}{\mathfrak{a}}
\newcommand{\p}{\mathfrak{p}}
\newcommand{\fm}{\mathfrak{m}}

\newcommand{\cF}{\mathcal{F}}
\newcommand{\cO}{\mathcal{O}}

\newcommand{\cM}{\mathcal{M}}

\newcommand{\cS}{\mathcal{S}}
\newcommand{\cI}{\mathcal{I}}
\newcommand{\F}{\mathbb F}
\newcommand{\EE}{\mathrm{E}}

\newcommand{\N}{\mathbb N}
\newcommand{\Q}{\mathbb Q}

\newcommand{\C}{\mathbb C}
\newcommand{\Z}{\mathbb Z}

\newcommand{\A}{\mathbb A}
\newcommand{\bP}{\mathbb{P}}

\newcommand{\rInj}{\mathrm{Inj}}
\newcommand{\rsoc}{\mathrm{soc}}
\newcommand{\rcosoc}{\mathrm{cosoc}}

\newcommand{\et}{\mathrm{\acute{e}t}}

\newcommand{\ra}{\rightarrow}
\newcommand{\lra}{\longrightarrow}

\newcommand{\wh}{\widehat}
\newcommand{\wt}{\widetilde}

\newcommand{\defn}{\overset{\rm{def}}{=}}

\newcommand{\GL}{\mathrm{GL}}
\newcommand{\SL}{\mathrm{SL}}
\newcommand{\JL}{\mathrm{JL}}

\newcommand{\bFp}{\overline{\F}_p}
\newcommand{\bQp}{\overline{\Q}_p}

\newcommand{\Sym}{\mathrm{Sym}}

\providecommand{\cInd}{\mathrm{c}\textrm{-}\mathrm{Ind}}

\newcommand{\rec}{\mathrm{rec}}

\newcommand{\brho}{\overline{\rho}}

\newcommand{\ide}{\mathbf{1}}

\newcommand{\plim}{\varprojlim}
\newcommand{\ilim}{\varinjlim}

\newcommand{\xto}[1][]{\xrightarrow{#1}}
\newcommand{\simto}{
\xto[\sim]} 

\providecommand{\ligne}{\ \textbf{---}\ }

\newcommand{\matr}[4]{\begin{pmatrix}{#1}&{#2}\\ {#3}&{#4}\end{pmatrix}}
\newcommand{\smatr}[4]{\bigl(\begin{smallmatrix} {#1}& {#2}\\ {#3}&{#4}\end{smallmatrix}\bigl)}

\def\onto{\twoheadrightarrow}

\def\OC{{\mathcal{O}}}

\def\AM{{\mathbb{A}}}
\def\TM{{\mathbb{T}}}
\def\Fov{{\overline{F}}}
\def\into{\hookrightarrow}

 \def\JL{\mathrm{JL}}

\def\wt{\widetilde}

\DeclareMathOperator{\Ann}{{\mathrm{Ann}}}
\DeclareMathOperator{\Art}{{\mathrm{Art}}}

\DeclareMathOperator{\Coker}{{\mathrm{Coker}}}

\DeclareMathOperator{\End}{{\mathrm{End}}}
\DeclareMathOperator{\Ext}{{\mathrm{Ext}}}

\DeclareMathOperator{\Frob}{{\mathrm{Frob}}}
\DeclareMathOperator{\Gal}{{\mathrm{Gal}}}
\DeclareMathOperator{\gr}{{\mathrm{gr}}}
\DeclareMathOperator{\Hom}{{\mathrm{Hom}}}

\DeclareMathOperator{\Ind}{{\mathrm{Ind}}}

\DeclareMathOperator{\JH}{{\mathrm{JH}}}

\DeclareMathOperator{\Ker}{{\mathrm{Ker}}}
\DeclareMathOperator{\loc}{{\mathrm{loc}}}

\DeclareMathOperator{\Mod}{\mathrm{Mod}}

\DeclareMathOperator{\Proj}{{\mathrm{Proj}}}

\DeclareMathOperator{\soc}{{\mathrm{soc}}}
\DeclareMathOperator{\Sp}{{\mathrm{Sp}}}
\DeclareMathOperator{\Spec}{{\mathrm{Spec}}}

\DeclareMathOperator{\Tor}{{\mathrm{Tor}}}

\def\a{\alpha}

\def\g{\gamma}
\def\G{\Gamma}
\def\d{\delta}

\def\e{\varepsilon}

\def\l{\lambda}

\def\o{\omega}

\def\s{\sigma}

\def\z{\zeta}

\def\To#1{\buildrel\hbox{\tiny{$#1$}}\over\longrightarrow}

\newcommand{\NEW}{\ide_{G_{\Q_p}}}

\newcommand{\quash}[1]{}

\setlist{itemsep=1mm}

\begin{document}

\title[]
{On some mod $p$ representations of quaternion algebra over $\Q_p$}
\author{Yongquan HU \and Haoran WANG }
\date{}

\maketitle

\begin{abstract}
 Let  $F$ be a totally real field in which $p$ is unramified and let $B$ be a quaternion algebra over $F$ which splits at at most one infinite place. Let $\overline{r}:\Gal(\overline{F}/F)\ra\GL_2(\overline{\mathbb{F}}_p)$ be a modular  Galois representation which satisfies the Taylor-Wiles hypotheses. Assume that for some fixed place $v|p$, $B$ ramifies at $v$ and $F_v$ is isomorphic to $ \mathbb{Q}_p$ and $\overline{r}$ is generic at $v$. 
We prove that the admissible smooth representations of the quaternion algebra over $\Q_p$ coming from mod $p$ cohomology of Shimura varieties associated to $B$ have Gelfand-Kirillov dimension $1$. As an application  we prove that the degree two Scholze's functor (which is defined in \cite{Scholze}) vanishes on generic supersingular representations of $\GL_2(\Q_p)$. We also prove some finer structure theorem about the image of Scholze's functor in the reducible case. 
 \end{abstract}

\setcounter{tocdepth}{1}
\tableofcontents

\section{Introduction}
\label{sec:introduction}

Let $p$ be a prime number. The mod $p$ (and $p$-adic) Langlands program has been emerged starting from the fundamental work  of Breuil \cite{Br03}. Up to present, the  correspondence in the case of $\GL_2(\Q_p)$ has been well-understood in various aspects, by  the work of \cite{Br03}, \cite{Co}, \cite{Em3}, and \cite{PaskunasIHES}.
Recently, there have been significant progress towards a  mod $p$ Langlands correspondence for $\GL_2(L)$, when $L$ is a finite unramified extension of $\Q_p$ (\cite{BHHMS1}, \cite{Hu-Wang}, \cite{BHHMS2}).
 However, a  mod $p$ Jacquet-Langlands correspondence is still largely unknown, even in the case of $\GL_2(\Q_p)$. 
   
Inspired by the local-global compatibility results  (\cite{Em3}, \cite{BDJ}), it is natural to search for  the correspondence in the cohomology of Shimura curves. 
 To explain this, let $F$ be a totally real extension of $\Q$ in which $p$ is unramified. Let $B$ be a quaternion algebra over $F$, which we assume to be split at only one infinite place in this introduction (in the text we will also treat the case $B$ is definite). If $U$ is a compact open subgroup of $(B\otimes_F\mathbb{A}_{F,f})^{\times}$, let $X_U$ be the associated smooth projective Shimura curve over $F$.  Let   $\overline{r}:\Gal(\overline{F}/F)\ra\GL_2(\bFp)$ be a continuous absolutely irreducible  representation.  Fix a place $v$ above $p$ and  a compact open subgroup   $U^v\subset (B\otimes_{F}\mathbb{A}_{F,f}^{\{v\}})^{\times}$, where $\mathbb{A}_{F,f}^{\{v\}}$ denotes the ring of finite ad\`eles of $F$ outside $v$.   We define
  \[\pi_v^B(\overline{r}):=\varinjlim_{U_v } \Hom_{\Gal(\overline{F}/F)}(\overline{r},H^1_{\et}(X_{U^vU_v}\times_F\overline{F},\bFp)\] 
  where $U_v$ runs over compact open subgroups of $B_v^{\times} :=(B\otimes_{F}F_v)^{\times}$.  In this way, we obtain an admissible smooth representation of $B_v^{\times}$. We assume that $B$ ramifies at $v$ from now on.
   
Assume that $\pi_v^B(\overline{r})$ is nonzero, i.e.~$\overline{r}$ is modular for $B$ and $U^v$; we also need to impose some extra assumptions on $\overline{r}$, see \S\ref{section-autom-forms} for details.     Then it is known that $\pi^B_v(\overline{r})$ is infinite dimensional (cf.~\cite[Cor.~3.5.4]{BreuilDiamond}, \cite[Thm.~1.4]{Scholze}).
On the other hand, since $B_v^{\times}$ is compact modulo its centre,  irreducible smooth mod $p$ representations of $B_v^{\times}$ (with a fixed central character)  are easy to classify. Actually, such a representation always has dimension $\leq 2$ and  there are only finitely many isomorphism classes. This implies that $\pi_v^B(\overline{r})$ is necessarily of \emph{infinite} length, and is built out by infinitely many pieces of a finite number of isomorphism classes of irreducible representations of $B_v^{\times}$ in a highly non-semisimple way.   A natural way to study such a representation is to look at its socle filtration. More conceptually, there is a standard invariant which measures the growth of the dimension of this socle filtration,  called \emph{Gelfand-Kirillov dimension} (cf. \S\ref{Sec::notation}).  

In this paper, we  study the Gelfand-Kirillov dimension of $\pi_v^B(\overline{r})$ in the case $F_v\cong \Q_p$. We make this assumption and assume $p\geq 5$ from now on; the reason for this restriction will be explained below after more notation is introduced.

Let $\brho:=\overline{r}_v(1).$ We make the following assumption on $\brho.$

\begin{enumerate}[label=(H\arabic*),ref=(H\arabic*)]
\item \label{H1} Assume that $\brho$ has one of the following forms:

\begin{itemize}
\item[$\bullet$] $\brho$ is absolutely irreducible and up to twist $\brho|_{I(\bQp/\Q_p)} \sim  \bigl(\begin{smallmatrix}
\omega_{2}^{r+1} & 0 \\ 0& \omega_2^{p(r+1)}
\end{smallmatrix} \bigr),$ with $2\leq r\leq p-3$, where $\omega_2$ is Serre's fundamental character of niveau 2;

\item[$\bullet$] $\brho$ is reducible nonsplit and up to twist $\brho|_{I(\bQp/\Q_p)} \sim \bigl(\begin{smallmatrix}
 \omega^{r+1} & * \\ 0& 1
\end{smallmatrix} \bigr)$,  with $0\leq r \leq p-3$, where $\omega$ is the mod $p$ cyclotomic character of $\Gal(\bQp/\Q_p)$.
\end{itemize}
\end{enumerate}
Below is our main result.  

\begin{theorem}\label{thm:intro-1}
Keep the above assumptions on $F, B$ and $\overline{r}$. Then $\pi_v^B(\overline{r})$ has Gelfand-Kirillov dimension    $1$. 
\end{theorem}

An analogue of Theorem \ref{thm:intro-1} was previously proved by Pa\v{s}k\=unas \cite{Paskunas-JL} when $\brho$ is reducible, using  Scholze's functor (introduced in \cite{Scholze}) and a result of Ludwig (\cite{Ludwig}). Combined with  some argument of \cite{Paskunas-JL}, Theorem \ref{thm:intro-1} implies some vanishing result on Scholze's functor, see Theorem \ref{thm:intro-2} below.

The proof of Theorem \ref{thm:intro-1} follows the method innovated in \cite{BHHMS1} (which treats the case  of $\GL_2$ over an unramified extension of $\Q_{p}$), but has several differences in technique. To explain this,   recall that one key step in \cite{BHHMS1}  is to  compare some  potentially crystalline deformation rings of $\brho$ of different (tame) types, and use it to gain information about the first $3$ steps of the socle filtration of certain $\bFp$-representations of $\GL_2$  with respect to the Iwahori subgroup.  In \emph{loc.~cit.}, the relevant  deformation rings are explicitly worked out by complicated computations, but unfortunately in doing this a  stronger genericity condition on $\brho$ is imposed, for example $12\leq r\leq p-15$ when $\brho$ is reducible. One may wonder, assuming this stronger genericity condition, if (the analogue of) Theorem \ref{thm:intro-1} remains true when $F_v$ is an  unramified extension  of $\Q_p$, namely  if $\pi_v^B(\overline{r})$ has Gelfand-Kirillov dimension equal to $[F_v:\Q_p]$. We believe this should be true and provable using the method of \cite{BHHMS1}. In fact,  we do give a criterion for controlling the Gelfand-Kirillov dimension  in this generality, see Corollary \ref{cor-gkdim-control} (which is an analogue of \cite[Cor.~5.3.5]{BHHMS1}). However, we caution that using only the deformation rings computed in \cite{BHHMS1} may not be enough to prove this statement,   because by the classical Jacquet-Langlands correspondence only those involving \emph{discrete series} inertial types are useful to obtain information about $\pi_v^B(\overline{r})$. Namely, to check the condition of Corollary  \ref{cor-gkdim-control}, one possibly needs to compute extra deformation rings (of discrete series inertial type), even when $F_v=\Q_p$. 
 
For the above reason and also with the  wish to weaken as much as possible the genericity condition in Theorem \ref{thm:intro-1}, we have chosen to restrict to the case $F_v \cong \Q_p$. The point is that  in this case there is an alternative construction   of Kisin's  potentially semistable deformation rings, due to Pa\v{s}k\={u}nas (\cite{Paskunas-BM}). This construction works only for two-dimensional representations of  $\Gal(\bQp/\Q_p)$ and in general does not  allow to determine the explicit form of these rings,  but  it fits perfectly our aim for the following two reasons.  \begin{itemize}
\item Firstly,  to carry out the strategy in \cite{BHHMS1},  we don't really need the explicit form of these deformation rings, but only certain congruence relations between them (cf.~\cite[Prop.~4.3.3]{BHHMS1}). In Pa\v{s}k\={u}nas' construction, these congruence relations can be proved    by  congruence relations between suitably chosen integral lattices inside the corresponding types. 
\item Secondly,  this construction closely relates  the structure of the deformation rings to the structure of $\pi(\brho)$,   the admissible smooth representation of $\GL_2(\Q_p)$ associated to $\brho$ by the mod $p$ local Langlands correspondence (see \S\ref{ss:LLC} for the precise definition). Thus, we may make use of the results of \cite{BL}, \cite{Br03}, \cite{Morra-ss}, \cite{Morra-atom} on $\pi(\brho)$ to study these  deformation rings; see Theorem \ref{thm-regular-HT02}  for such an example.   
\end{itemize}
Besides, in \cite{BHHMS1} they use  potentially crystalline deformation rings of Hodge-Tate weights  $(-1,2)$ (and of $(0,1)$), while we use deformation rings of Hodge-Tate weights  $(0,2)$. This also allows a further (minor) improvement on the genericity condition.

Theorem \ref{thm:intro-1} can be applied to study   Scholze's functors. Let $L$ be a finite extension of $\Q_p$ (not necessarily  unramified). Let $D$ be the central division algebra over $L$ of dimension $n^2$ and invariant $1/n,$ Scholze (\cite{Scholze}) has constructed  a cohomological covariant $\delta$-functor $\{\cS^i, i\geq0\}$ from the category of admissible smooth representations of $\GL_n(L)$ over $\bFp$ to admissible smooth representations of $ D^{\times}$ which carry a continuous and commuting action of $\Gal(\overline{L}/L)$. If $\pi$ is an admissible smooth representation of $\GL_n(L)$ over $\bFp$, then $\cS^i(\pi)$ is defined as the cohomolgy group $H^i_{\et}(\mathbb{P}_{\C_p}^{n-1},\mathcal{F}_{\pi})$, where $\mathcal{F}_{\pi}$ is a certain Weil-equivariant sheaf on the adic space $\mathbb{P}_{\C_p}^{n-1}$.  His construction is expected to realize both $p$-adic local Langlands   and Jacquet-Langlands correspondences. In general, these cohomology groups seem very difficult to compute, but Scholze has computed $\cS^0(\pi)$ and showed that $\cS^i(\pi)$ vanishes whenever $i>2(n-1)$.
Specializing to $n=2$, the case we are interested in, we have $\cS^i(-)=0$ for $i>2$. Later on, Ludwig proved  that $\cS^2(\pi)=0$  if either $\pi$ is principal series or special series of $\GL_2(\Q_p)$,  using the geometry of perfectoid modular curves (\cite{Ludwig}). Since it is easy to compute $\cS^2(\pi)$ if $\pi$ is one-dimensional, this leaves only the case of supersingular representations for $\cS^2.$ 

By Breuil's classification (\cite{Br03}), any supersingular representation of $\GL_2(\Q_p)$ with a central character is up to twist isomorphic to 
\[\big(\cInd_{\GL_2(\Z_p)\Q_p^{\times}}^{\GL_2(\Q_p)}\Sym^{r}\bFp^2\big)/T\]
where $0\leq r\leq p-1$ and $T$ is a certain Hecke operator (\cite{BL}).  As an  application of Theorem \ref{thm:intro-1}, we have the following result. 
\begin{theorem}\label{thm:intro-2}
Let $\pi$ be a supersingular representation of $\GL_2(\Q_p)$ as above and  assume $2\leq r\leq p-3$. Then $\cS^2(\pi)=0$.
\end{theorem}

Our proof of Theorem \ref{thm:intro-2} is inspired by Pa\v{s}k\=unas' work  \cite{Paskunas-JL}, where he has  used Ludwig's  vanishing result of $\cS^2$ to prove  Theorem \ref{thm:intro-1} in the case $\brho$ is reducible. We observe that his argument can actually go in reverse direction, namely the vanishing of $\cS^2$ on supersingular $\pi$ can be deduced from the  Gelfand-Kirillov dimension of $\cS^1(\pi)$ (see Proposition \ref{prop:equiv}).  Thus, Theorem \ref{thm:intro-2} follows from  Theorem \ref{thm:intro-1} and a local-global compatibility result \`a la Emerton (\cite{Em3}, \cite{Dosp-Lebras}).

Another reason for focusing on the case of $\GL_2(\Q_p)$ is that we can prove some finer results on the structure of $\cS^1(\pi(\brho))$. We put 
\[\JL(\brho)=\left\{\begin{array}{lll} \Hom_{G_{\Q_p}}\big(\chi\omega^{-1},\cS^1(\pi(\brho))\big) &\mathrm{if}\ \brho\sim \smatr{\chi}*0{\chi\omega} \medskip\\
\Hom_{G_{\Q_p}}\big(\brho\otimes\omega^{-1},\cS^1(\pi(\brho))\big) &\mathrm{otherwise}.
\end{array}\right.\]

\begin{theorem}\label{thm:intro-3}
Let $\brho$ be as in \ref{H1}.

(i) Assume $\brho\nsim \smatr{\chi}*0{\chi\omega}$ for any character $\chi$. Then $\cS^1(\pi(\brho))\cong (\brho\otimes\omega^{-1})\otimes \JL(\brho)$ as representations of $\Gal(\bQp/\Q_p)\times B_v^{\times}$.

(ii) Assume $\brho$ is reducible.  Denote by $\brho^{\rm ss}$ the semisimplification of $\brho$.

\begin{enumerate}
\item[(a)] Assume $\brho^{\rm ss}\nsim \chi\oplus \chi\omega$ for any $\chi$. Then   $\JL(\brho)$   depends only on $\brho^{\rm ss}$.

\item[(b)]  Let $\brho_1\sim \smatr{\omega}*01$ and $\brho_2\sim\smatr{1}*0{\omega}$  be nonsplit extensions.   Then there exists an admissible $\bFp$-representation $V$ of $B_v^{\times}$ such that 
\[0\ra \ide_{D^{\times}}\ra \JL(\brho_1)\ra V\ra0\]
\[0\ra V\ra \JL(\brho_2)\ra (\ide_{D^{\times}})^{\oplus 2}\ra 0.\]
\end{enumerate}
\end{theorem}

It may look  surprising that the representation $\JL(\brho)$  does not determine $\brho$, but only $\brho^{\rm ss}$, in case (a) of Theorem \ref{thm:intro-3}(ii); see Remark  \ref{rem:test} for an explanation. It would be  interesting to describe the precise structure of $\JL(\brho)$. We plan to come back to this question in future work. 

\medskip

We now give a brief overview of the contents of each section. In \S\ref{sec:Koh}, we study the structure of the $p$-adic group $B_v^{\times}$ and prove a criterion for controlling the Gelfand-Kirillov dimension of its representations (analogous to \cite[\S5]{BHHMS1}). In \S\ref{section-lattices-GL2} we study the structure of integral lattices in various locally algebraic types of $\GL_2(\Z_p)$. In \S\ref{sec:deform}, we use Pa\v{s}k\={u}nas' technique to study potentially crystalline  deformation rings of tame type and Hodge-Tate weights $(0,2)$. In \S\ref{section-autom-forms} and \S\ref{sec:GK}, we carry out the gluing process for $B_v^{\times}$-representations and prove our main result, Theorem \ref{thm:intro-1}. Finally, we  study Scholze's functors, and prove Theorem \ref{thm:intro-2} in \S\ref{section:application} and Theorem \ref{thm:intro-3} in \S\ref{sec:JL}.

\subsection{Notation}
\label{Sec::notation}

 \hfill
 
We fix a prime number $p\geq 5.$ Let $E\subset \bQp$ be a finite unramified extension of $\Q_p$, with ring of integers $\cO$ and residue field $\F .$ We will assume without further comment that $\F$ is sufficiently large.

If $F$ is a field, let $G_F := \Gal(\Fov/F)$ denote its absolute Galois group. Let $\e$ denote the $p$-adic cyclotomic character of $G_F$, and $\omega$ the mod $p$ cyclotomic character.

If $F$ is a $p$-adic field, $V$ is a de Rham $p$-adic representation of $G_F$ over $E,$ and $\kappa: F\into E,$ then we will write ${\rm HT}_{\kappa} (V)$ for the multiset of Hodge-Tate weights of $V$ with respect to $\kappa.$ By definition, ${\rm HT}_{\kappa} (V)$ consists of $- i$ with multiplicity $\dim_E(V \otimes_{\kappa, F} \wh{\Fov}(i))^{G_F}$, e.g. ${\rm HT}_{\kappa} (\e) = \{1\}$ at all embedding $\kappa.$

If $G$ is a $p$-adic analytic group, we denote by $\Mod_G^{\rm sm}(\cO)$ the category of smooth representations of $G$ on $\cO$-torsion modules. Let $\Mod_G^{\rm l.adm}(\cO)$ (resp. $\Mod_G^{\rm adm}(\cO)$) denote the full subcategory of locally admissible (resp. admissible) representations. If $\z: Z_G \to \cO^{\times}$ is a continuous character of the centre of $G,$ then we denote by $\Mod_{G,\z}^{\rm sm}(\cO)$ (resp. $\Mod_{G,\z}^{\rm l.adm}(\cO)$, resp. $\Mod_{G,\z}^{\rm adm}(\cO)$) the full subcategory of $\Mod_G^{\rm sm}(\cO)$ consisting of smooth (resp. locally admissible, resp. admissible) representations on which $Z_G$ acts by the character $\z.$

The Pontryagin duality $M \mapsto M^{\vee}: = \Hom_{\cO}^{\rm cont}(M , E/\cO)$ induces an anti-equivalence between the category of discrete $\cO$-modules and the category of pseudocompact $\cO$-modules. Under this duality  the category $\Mod_G^{\rm sm}(\cO)$ is anti-equivalent to the category of profinite augmented $G$-representations over $\cO$ which is denoted by $\Mod_G^{\rm pro}(\cO).$ Let $\frak{C}_{G}(\cO)$ (resp. $\frak{C}_{G,\z}(\cO)$) denote the full subcategory of $\Mod_G^{\rm pro}(\cO)$ which is anti-equivalent to $\Mod_{G}^{\rm l.adm}(\cO)$ (resp. $\Mod_{G,\z}^{\rm l.adm}(\cO)$) under the Pontryagin duality. Note that on an object in $\frak{C}_{G,\z}(\cO)$ the centre is acting by $\zeta^{-1}.$

Let $ (R,\fm)$ be a complete noetherian local commutative $\cO$-algebra with residue field $\F.$ We define the category $\Mod_{G}^{\rm sm}(R)$ of \emph{smooth} $R[G]$-modules, and the category $\Mod_{G}^{\rm l.adm}(R)$ of \emph{locally admissible} smooth $R[G]$-modules as in \cite[\S2]{PaskunasIHES}. Let $\frak{C}_{G}(R)$ be the dual category of $\Mod_{G}^{\rm l.adm}(R)$ under the Pontryagin duality. If $\z: Z_G \to \cO^{\times}$ is a continuous character of the centre of $G,$ we can similarly define $\Mod_{G,\z}^{\rm l.adm}(R)$ and its dual category $\frak{C}_{G,\z}(R).$

If $M$ is a torsion-free linear-topological $\cO$-module,  $M^{d}$ denotes its Schikhof dual $\Hom^{\rm cont}_{\cO} (M, \cO).$ The functor $M \mapsto M^d$ induces an anti-equivalence of categories between the category of pseudo-compact torsion-free linear-topological $\cO$-modules and the category of $\varpi$-adically complete and separated  torsion-free  $\cO$-modules.

If $R$ is a ring and $M$ is a left $R$-module, we denote by $\soc_R(M)$ (resp. $\mathrm{cosoc}_R(M)$) the socle (resp. cosocle) of $M.$ Inductively we define the socle (resp. cosocle)  filtration   of $M$.  If $M$ has finite length, we denote by $\JH(M)$  the set of Jordan-H\"older factors of $M$.

The \emph{grade} $j_{R}(M)$ of $M$ over $R$ is defined by
\[
j_{R}(M)=\mathrm{inf}\{i \in \N ~|~ \Ext^i_{R}(M,R)\neq0\}.
\]
Assume $R$ is noetherian. The ring $R$ is called \emph{Auslander-Gorenstein} if it has finite left and right injective dimension and the following Auslander condition holds: for any $R$-module $M$, any integer $m\geq0$ and any $R$-submodule $N$ of $\Ext^m_R(M,R)$, we have $j_{R}(N)\geq m$. An Auslander-Gorenstein ring is  called \emph{Auslander regular} if it has finite global dimension. If $R$ is an Auslander regular ring and $M$ is a finitely generated $R$-module, define the {\em dimension} 
\begin{equation*}
\d_R(M) : = {\rm gld}(R) - j_R(M),
\end{equation*}
 where ${\rm gld}(R)$ is the global dimension of $R.$

Let $G_0$ be a compact $p$-adic analytic group. The ring-theoretic properties of $ \cO [\![G_0]\!]$ are established by the fundamental works of Lazard \cite{Lazard} and Venjakob \cite{Ven}. In particular, if $G_0$ has no element of order $p$,  then $  \cO [\![G_0]\!]$ is an 
Auslander regular ring of dimension $1+\dim_{\Q_p} G_0$, where $\dim_{\Q_p} G_0$ is the dimension of $G_0$ as a $p$-adic analytic group. If $M$ is nonzero, we have
\[
0 \leq j_{  \cO [\![G_0]\!]} (M) \leq 1+\dim_{\Q_p} G_0,
\]
and $\d_{ \cO [\![G_0]\!]}( M) = 1+\dim_{\Q_p} G_0 - j_{ \cO [\![G_0]\!]}(M).$
If $G$ is a $p$-adic analytic group with a fixed open compact subgroup $G_0 \subseteq G$ and $M $ is a finitely generated $ \cO [\![G_0]\!]$-module equipped with a compatible $G$-action, we define $j_{G}(M)$ (resp. $\d_G(M)$) as $j_{\cO[\![G_0]\!]}(M)$ (resp. $\delta_{\cO[\![G_0]\!]}(G)$); this does not 
 depend on the choice of $G_0.$

If $\pi$ is an admissible smooth representation of $G$ over $\F,$ then $\pi^{\vee}$ is finitely generated over $\cO [\![G_0]\!].$ The {\em Gelfand-Kirillov} dimension  of $\pi$ is defined by (see \cite[Rem. 5.1.1]{BHHMS1})
\begin{equation*}
\dim_G(\pi) := \d_G(\pi^{\vee}) .
\end{equation*}

\subsection{Acknowledgements}
We thank  Yiwen Ding, Andrea Dotto and Vytautas Pa\v{s}k\=unas for several interesting  discussions during the preparation of the paper, and we thank Dingxin Zhang and Weizhe Zheng for answering our questions. We would like to thank Christophe Breuil and Florian Herzig for their comments on an earlier version. We are  grateful to Judith Ludwig for her help with the proof of Theorem \ref{thm-Scholze}(iii). We are also grateful to the anonymous referee for several helpful corrections and suggestions.  Y. H. has presented this work at Beijing International Center for Mathematical Research in November 2021, and he thanks Yiwen Ding for the invitation. 
It will be obvious to the reader that this work is largely inspired by the previous work  \cite{BHHMS1} and \cite{Paskunas-JL}.

H. W. is partially supported by National Key R\&D Program of China 2023YFA1009702 and National Natural Science Foundation of China Grants 12371011, 11971028. Y. H. is partially supported by  National Key R$\&$D Program of China 2020YFA0712600, CAS Project for Young Scientists in Basic Research, Grant No. YSBR-033; National Natural Science Foundation of China Grants 12288201 and 11971028; National Center for Mathematics and Interdisciplinary Sciences and Hua Loo-Keng Key Laboratory of Mathematics, Chinese Academy of Sciences.

\section{The $p$-adic Lie group $D^{\times}$}
\label{sec:Koh}

\subsection{Results of Kohlhaase}
We recall and extend some results of \cite{Ko}.

Let $L = \Q_{p^f}$ be the unramified extension of degree $f$ over $\Q_p.$ Let $D$ be the unique central division algebra of dimension $4$ over $L.$ For $a\in D$, define  $v_D(a) := v_p ({\rm Nrd}_D(a)),$ where $v_p$ is the $p$-adic valuation on $L$ normalized so that $v_p(p) = 1,$ and ${\rm Nrd}_D : D \to L$ is the reduced norm map; this gives   a non-archimedean valuation on $D$. Let $\OC_D : =\{a\in D~|~ v_D(a) \geq 0\}$ be the ring of integers  and $\frak{p}_D: = \{a\in D~|~ v_D(a) \geq 1\}$ the maximal ideal, which can be generated by a uniformizer $\varpi_D.$ The residue field $k_D : = \OC_D/\frak{p}_D$ is isomorphic to $\F_{q^2},$ where $q:= p^f.$ Let $L' $ be the unramified quadratic extension of $L$ in $\overline{\Q}_p.$ We denote by $\s : L' \to L'$ a lift of the Frobenius map $x \mapsto x^q$ on $\F_{q^2}.$ Let $L'\langle X \rangle$ denote the non-commutative polynomial ring in one variable over $L'$ satisfying the relation $X a   = \s(a) X,~ \forall a \in L'.$ Then the homomorphism $L'\langle X \rangle \to D,$ $X\mapsto \varpi_D$ induces an isomorphism of $L$-algebras 
\begin{equation}\label{eq:division-D}
L'\langle X \rangle / ( X^2 -  p) \cong D.
\end{equation}

Let $D^{\times}$ (resp. $\cO_D^{\times}$) denote the group of invertible elements of $D$ (resp. $\cO_D$) and 
\begin{equation}\label{eqn-U^i_D}
U^n_D: = 1+\varpi_D^n \OC_D,~n\geq 1
\end{equation} 
which are compact open normal (pro-$p$) subgroups of $D^{\times}.$  
We have \[D^{\times} = \OC^{\times}_D \rtimes \varpi_D^{\Z},\ \ \OC_D^{\times} / U^1_D \cong \F_{q^2}^{\times}. \]    Let $Z_D$ denote the centre of $D^{\times}$ which is isomorphic to $L^{\times}.$ Then   $Z_D\OC_D^{\times}$ is of index $2$ in $D^{\times}.$ Let
 $Z^1_D = Z_D \cap U^1_D.$

Assume $p \geq 5.$ Let  $\o : U^1_D \backslash \{1\} \to (0, \infty)$ be the map defined by $\o(g) : = \frac{1}{2} v_D (g - 1)$, and set $\o(1):=\infty$. As in \cite[Example 23.2]{Schneider}, one shows that $\o$ is a {\em $p$-valuation} on $U^1_D$ in the sense of Lazard \cite[III.2.1.2]{Lazard}. For any real number $\nu > 0,$ let
\[
(U^1_D)_{\nu} : = \big\{ g\in U^1_D~|~ \o (g)\geq \nu \big\},\ \   (U^1_D)_{\nu+} : = \big\{ g\in U^1_D~|~ \o (g)> \nu \big\}.
\]
We set 
\[
\gr U^1_D: = \bigoplus_{\nu>0} (U^1_D)_{\nu} / (U^1_D)_{\nu+}   .
\]
It is easy to see that $U_D^i=(U^1_D)_{\frac{i}{2}}$ and $U^{i+1}_D  = (U^1_D)_{\frac{i}{2} +},
$ so
we have
\[
\gr U^1_D = \bigoplus_{i\geq 1} U^i_D / U^{i+1}_D.
\]
We say a nonzero homogeneous element $t\in \gr U^1_D $ is of degree $i$ if $t \in U^i_D/ U^{i+1}_D.$

As explained in \cite[\S 25]{Schneider}, $\gr U^1_D$ is a graded Lie algebra over the polynomial ring $\F_p[\e]$ by setting
\[
[g U^{i+1}_D , g' U^{j+1}_D] : = g g' g^{-1} g'^{-1} U^{i+j+1}_D,~~ g \in U^{i}_D,~g' \in U^{j}_D,
\]
and
\[
\e (g U^{i+1}_D) : = g^p U^{i+3}_D,~ g \in U^i_D.
\]
Note that $U^i_D/ U^{i+1}_D \cong (\F_{q^2}, + )$ is an $\F_q$-vector space by setting 
\[
\lambda\cdot (1 + \varpi_D^i a) U^{i+1}_D := (1 + \varpi_D^i [\l]a) U^{i+1}_D,
\] 
where $[\l] \in \cO_L$ is the Teichm\"uller lift of $\l \in \F_q.$  One checks that the Lie bracket on $\gr U^1_D$ is $\F_q$-bilinear,  hence $\gr U^1_D$ becomes a graded Lie algebra over the polynomial ring $\F_q[\e].$

\begin{proposition}
The natural map $\F_q[\e] \otimes_{\F_q} (U^1_D/U^2_D \oplus U^2_D / U^3_D) \to \gr U^1_D$ is an isomorphism of $\F_q[\e]$-modules.
\end{proposition}
\begin{proof}
The proof of \cite[Lem.~3.12]{Ko} (when $L=\Q_p$) extends to the general case.
\end{proof}

 Let $\overline{\gr U^1_D} : = \gr U^1_D \otimes_{\F_q[\e]} \F_q$ where the map $\F_q[\e] \to \F_q$ sends $\e$ to $0$. We first determine the Lie algebra structure of $\overline{\gr U^1_D}$.   Fix $\xi \in \F_{q^2} \setminus \F_q$ and set \[\g_{1} : = 1+\varpi_D,\ \g_{2} : = 1+ \varpi_D [\xi],\ \g_{3} := \g_{1} \g_{2} \g_{1}^{-1} \g_{2}^{-1}, \    \g_{4} : = 1+p,\] where $[\xi] \in \cO_{L'}$ is the Teichm\"uller lift of $\xi.$  We have $\o(\g_{1}) = \o(\g_{2}) = 1/2$ and $\o(\g_{3} ) = \o(\g_{4}) = 1.$\footnote{One checks that $\g_{3} \equiv 1 + p([\xi] - [\xi]^q) \pmod{U^3_D}.$} Let $\overline{\g}_{1},\overline{\g}_{2}\in U^1_D/ U^2_D$ be the images of $\g_1$ and $\g_2$ and let $\overline{\g}_{3},\overline{\g}_{4}\in U^2_D/ U^3_D$ be the images of $\g_3$ and $\g_4.$ Then $\overline{\g}_{1},$ $\overline{\g}_{2},$ $\overline{\g}_{3}$, $\overline{\g}_{4}$ form an $\F_q$-basis of $U_D^1/U_D^2\oplus U_D^2/U_D^3$, hence also an $\F_q$-basis of  $\overline{\gr U^1_D}$. They  satisfy (in $\overline{\gr U^1_D}$, i.e.  after modulo $\e$)
\begin{equation}\label{eq:relation}
[\overline{\g}_{1},\overline{\g}_{2}] = \overline{\g}_{3},~~[\overline{\g}_{1},\overline{\g}_{3}]=[\overline{\g}_{2}, \overline{\g}_{3}] = [\overline{\g}_{4},\overline{\g}_{1}] = [\overline{\g}_{4},\overline{\g}_{2}] = [\overline{\g}_{4},\overline{\g}_{3}] = 0,
\end{equation}
see the discussion after \cite[Rem.~3.15]{Ko}.

Passing to the quotient group $U^1_D /Z^1_D,$ we can consider $\overline{\gr U^1_D/Z^1_D} : = \gr U^1_D/Z^1_D \otimes_{\F_q[\e]} \F_q,$ with the induced filtration on $ U^1_D/Z^1_D.$ Then $\overline{\gr U^1_D/Z^1_D} $ is isomorphic to $ \overline{\gr U^1_D}/ (\overline{\g}_{4} )$ as graded Lie algebras over $\F_q$, where  $ (\overline{\g}_{4}) : = \F_q \overline{\g}_4$ is the sub-Lie algebra of $\overline{\gr U^1_D}$ generated by $\overline{\g}_{4}.$

Let $\frak{g}_{\F_p} = \F_p e \oplus \F_p f \oplus \F_p h $ be the graded Lie algebra of dimension $3$ over $\F_p,$ with $e$  and $f$ in degree $1,$ $h$ in degree $2$ and satisfying the relations
\[
[e,f] = h,~ [h,e] = [h,f] =0.
\]
From \eqref{eq:relation} we easily deduce the following result.

\begin{corollary}\label{coro--lie-alg}

The graded Lie algebra $\overline{\gr U^1_D/Z_D^1}$  is isomorphic to $\frak{g}_{\F_q} : =\F_q \otimes_{\F_p} \frak{g}_{\F_p}.$
\end{corollary}

\begin{remark}
One can also deduce the structure of the Lie algebra $ \overline{\gr U^1_D/Z^1_D} \cong \overline{\gr U^1_D}/ (\overline{\g}_{4} )$ from the results of \cite[\S 5.3]{BHHMS1} by comparing with the pro-$p$-Iwahori subgroup of $\GL_2$ over $\cO_{L'}.$
\end{remark}
 
\subsection{The graded group algebra} \label{section-graded-algebra}

Let $\Z_p[\![ U^1_D ]\!]=\varprojlim_{i\geq 1}\Z_p[U_D^1/U_D^i]$ be the Iwasawa algebra of $U^1_D$ over $\Z_p.$ It is a pseudocompact local $\Z_p$-algebra. For $\nu \geq 0,$ let $J_{\nu}$ denote the smallest closed $\Z_p$-submodule of $\Z_p[\![ U^1_D]\!]$ which contains all elements of the form $p^{\ell} (h_1 - 1)\cdots (h_s - 1)$ with $\ell, s \geq 0,$ $h_1,\ldots, h_s \in U^1_D$ and \[\ell + \o(h_1) + \cdots +\o(h_s) \geq \nu.\] Let $J_{\nu+} : = \bigcup_{\nu' > \nu} J_{\nu'}$.    Let   
\[
\gr_J  \Z_p[\![ U^1_D ]\! ] := \bigoplus_{\nu\geq 0} J_{\nu}/J_{\nu+}.
\]
which is   an associative graded algebra  over $\gr \Z_p := \bigoplus_{i\geq 0} p^i \Z_p/p^{i+1} \Z_p.$ It naturally has a graded Lie algebra structure. 

The homomorphism of abelian groups $\mathcal{L}_{\nu} : \gr_{\nu} U^1_D \to J_{\nu} / J_{\nu +},$ $g (U^1_D)_{\nu + }
\mapsto (g-1) + J_{\nu +}$ extends to a homomorphism of graded $\F_p[\e]$-Lie algebras $\mathcal{L} : \gr U^1_D \to \gr \Z_p[\![ U^1_D ]\! ]$, where the  $\F_p[\varepsilon]$-algebra structure on $\gr\Z_p[\![U_D^1]\!]$ is given through  the isomorphism  $\F_p[\e] \simto \gr \Z_p,$ $\e \mapsto p+ p^2 \Z_p \in \gr^1 \Z_p.$ 
Let $U_{\F_p[\varepsilon]}(\gr U^1_D)$ be the universal enveloping algebra of $\gr U^1_D$ over $\F_p[\varepsilon].$ By the universal property of $U_{\F_p[\varepsilon]}(\gr U_D^1)$, we have a homomorphism of associative $\gr \Z_p$-algebras
\begin{equation} \label{eqn::1-1}
\wt{\mathcal{L}} :  U_{\F_p[\varepsilon]}(\gr U^1_D) \to \gr_J \Z_p[\![ U^1_D ]\!].
\end{equation}
By \cite[Thm.~28.3]{Schneider}, $\wt{\mathcal{L}}$ is an isomorphism.

In practice, we will consider the Iwasawa algebra associated to the quotient group $U_D^1/Z_D^1$.   Let $\Z_p[\![U_D^1/Z_D^1]\!]$ (resp. $\F_p[\![U_D^1/Z_D^1]\!]$) be the Iwasawa algebra of $U_D^1/Z_D^1$ over $\Z_p$ (over $\F_p$). We have $\F_p[\![U_D^1/Z_D^1]\!]=\Z_p[\![U_D^1/Z_D^1]\!]\otimes_{\Z_p}\F_p$. The filtration $\{J_{\nu},  \nu\geq 0\}$ induces a filtration on $\Z_p[\![U_D^1/Z_D^1]\!]$ and on $\F_p[\![U_D^1/Z_D^1]\!]$. On the other hand,  letting $\frak{m}_{D}$ denote the maximal ideal of $\F_p[\![ U^1_D/Z^1_D ]\!]$, we may consider the $\fm_{D}$-adic filtration on   $\F_p[\![ U^1_D/Z^1_D ]\!] $. The following result shows that these two filtrations coincide up to rescaling indices.

\begin{lemma}\label{lem:Ko-3.13}
Denote by $\overline{J}_{\nu}$ the image of $J_{\nu}$ in $\F_p[\![U_D^1/Z_D^1]\!]$. Then $\overline{J}_{i/2}=\fm_{D}^i$   for any $i\geq 0$. 
\end{lemma}
\begin{proof}
The proof of \cite[Lem.~3.13]{Ko} (when $L=\Q_p$) extends to the general case.    
\end{proof} 

One checks that $J_{\nu}\neq J_{\nu+}$ exactly when $\nu=\frac{i}{2}$ for some $i\geq 0$. Thus, by Lemma  \ref{lem:Ko-3.13} the graded algebra
\begin{equation}\label{eq:gr-Lambda}
\gr_{\fm_D} \F_p[\![ U^1_D /Z^1_D ]\! ] := \bigoplus_{i\geq 0} \frak{m}_{D}^i /  \frak{m}_{D}^{i+1}
\end{equation}
is identical to $\bigoplus_{\nu\geq 0}\overline{J}_{\nu}/\overline{J}_{\nu+}$.

\begin{proposition}\label{prop:PBW}
There is an isomorphism of graded $\F_p$-algebras
\[ \gr_{\fm_D}  \F_p[\![ U^1_D /Z^1_D ]\!] \cong U_{\F_p}(\frak{g}_{\F_{q}}).\]
\end{proposition}
\begin{proof}
By the above discussion, the result is a direct consequence of Corollary \ref{coro--lie-alg} via \eqref{eqn::1-1}.
\end{proof}

Let $\F$ be a finite extension of $\F_p$ such that $\F_q$ embeds into $\F.$ Let $\mathcal{J}$ denote the set of embeddings $\F_{q} \into \F$ and fix $\s_0\in \mathcal{J}$. We label the embeddings $\s_j = \s_0 \circ \varphi^j,$ so that $\mathcal{J}$ is identified with $\{0,\ldots, f-1\}.$ Let $\frak{g}_j: = \F \otimes_{\F_{q}, \s_j} \frak{g}_{\F_{q}}.$ We then have $\F\otimes_{\F_p}\frak{g}_{\F_{q}} = \bigoplus_{j=0}^{f-1} \frak{g}_j.$ Let $e_j,f_j,h_j \in \frak{g}_j$ denote $1\otimes e,$ $1\otimes f,$ $1\otimes h \in \F \otimes_{\F_{q}, \s_j} \frak{g}_{\F_{q}}.$  

We again denote by $\frak{m}_{D}$  the maximal ideal of $\F[\![ U^1_D/Z^1_D ]\!]=\F\otimes_{\F_p}\F_p[\![U_D^1/Z_D^1]\!]$.   Then Proposition \ref{prop:PBW} implies that
\begin{equation} \label{eqn-1-1-isom}
\gr_{\fm_D} \F[\![ U^1_D /Z^1_D ]\!] =\F\otimes_{\F_p} (\gr_{\fm_D}  \F_p[\![ U^1_D /Z^1_D ]\! ] ) \cong U_{\F}(\F\otimes_{\F_p}\frak{g}_{\F_q})\cong \bigotimes_{j=0}^{f-1} U_{\F} (\frak{g}_j).
\end{equation}
In particular, we have 
$\gr_{\fm_D}^1 \F[\![ U^1_D /Z^1_D ]\!]=\bigoplus_{j=0}^{f-1}(\F e_j\oplus\F f_j)$.

\begin{theorem}\label{thm:gradedring}
(i) The graded ring $\gr_{\fm_D} \F[\![ U^1_D /Z^1_D ]\!]$ is Auslander regular.

(ii) The sequence $(h_0,\ldots, h_{f-1})$ is a regular sequence of central elements of $\gr_{\fm_D}  \F[\![ U^1_D /Z^1_D ]\!].$ The quotient  $\gr_{\fm_D} \F[\![ U^1_D /Z^1_D ]\!]/(h_0,\ldots, h_{f-1})$ is commutative and is isomorphic to the polynomial ring $\F [e_j,f_j;~ 0\leq j\leq f-1].$
\end{theorem}
\begin{proof}
The proof is the same as that of \cite[Thm.~5.3.4]{BHHMS1}.
\end{proof}

Theorem \ref{thm:gradedring} is not enough for the application to Gelfand-Kirillov dimension, namely Corollary \ref{cor-gkdim-control} below. We shall find eigenbases of $\F\otimes_{\F_p}\frak{g}_{\F_q}$ for the $\F_{q^2}^{\times}$-action in the next subsection.

\subsection{Gelfand-Kirillov dimension}\label{subsection:GK}

We regard $\F_{q^2}^{\times} $ as a subgroup of $\cO_{L'}^{\times}$ via the Teichm\"uller lifting map, and then as a subgroup of $\cO_D^{\times}$ via the fixed embedding $L'\hookrightarrow D$. It normalizes $U_D^1$, thus  acts on $\overline{\gr U_D^1}$  and on $\frak{g}_{\F_{q}}$. 
In practice we need a basis of $\F\otimes_{\F_p}\frak{g}_{\F_q} $ consisting of eigenvectors for the action of $\F_{q^2}^{\times}$. Note that $e_j$ and $f_j$ are only eigenvectors for the action of $\F_q^{\times}$, but not for $\F_{q^2}^{\times}$.

Choose an embedding $\F_{q^2}\hookrightarrow \F$ which extends the fixed embedding $\sigma_0:\F_q\hookrightarrow \F$; we again denote it by $\sigma_0$ and let $\sigma_j=\sigma_0\circ\varphi^j$ for $0\leq j\leq 2f-1$. 

For $0\leq j\leq 2f-1$, define the following elements in $\F[\![U_D^1/Z_D^1]\!]$:
\[Y_j:=\sum_{\lambda\in\F_{q^2}^{\times}}\sigma_j(\lambda)^{-1}(1+\varpi_D[\lambda]),\]
where the term $1+\varpi_D[\lambda]$ is considered as an element in the \emph{group} $U_D^1/Z_D^1$. Since $\sum_{\lambda\in\F_{q^2}^{\times}}\sigma_{j}(\lambda)^{-1}=0$, we have $Y_j\in \fm_{D}$. If $\mu\in\F_{q^2}^{\times}$, then one  checks  that  \begin{equation}\label{eq:Yj-char}\mu\cdot Y_j:=[\mu]Y_j[\mu]^{-1}=\alpha_j(\mu)Y_j\end{equation}
where $\alpha_j:\cO_D^{\times}\ra\F^{\times}$ denotes the character defined by 
\begin{equation}\label{eq:def-alpha}\alpha_j(x):=\sigma_j(\overline{x})^{q-1}.\end{equation}
 Note that $\alpha_{j+f}=\alpha_j^q=\alpha_j^{-1}$.
 
For $0\leq j\leq 2f-1,$ let $y_j:= Y_j+ \frak{m}_D^2 \in \gr^1_{\fm_D}\F[\![U_D^1/Z_D^1]\!]$.

\begin{lemma}\label{lem:Schraen}
(i) The elements $\{Y_j, 0\leq j\leq 2f-1\}$ generate the ideal $\fm_{D}$. 

(ii) The elements $\{y_j, 0\leq j\leq 2f-1\}$ form a basis of $\gr^1_{\fm_D}\F[\![U_D^1/Z_D^1]\!]$. 
\end{lemma}
\begin{proof}
(i) It is equivalent to checking that the images of $Y_j$ in $\fm_{D}/\fm_{D}^2$ are linearly independent (over $\F$). This is proved by a standard technique; see the proof of  \cite[Prop.~2.13]{Schraen} for a similar argument.

(ii)  This is clear, because  $\gr^1_{\fm_D}\F[\![U_D^1/Z_D^1]\!]$ has dimension $2f$   (with a basis $\{e_j,f_j, 0\leq j\leq f-1\}$). 
\end{proof}

\begin{lemma}\label{lem:gh-1}
For $g\in U_D^i/(U_D^i\cap Z_D^1)$ and $h\in U_D^j/(U_D^j\cap Z_D^1)$, we have 
\[gh-1\equiv (g-1)+(h-1) \mod \fm_D^{i+j}.\]
\end{lemma} 
\begin{proof}
Using Lemma \ref{lem:Ko-3.13}, this is a consequence of the equality $(g-1)(h-1)=(gh-1)-(g-1)-(h-1).$  
\end{proof}

For $t\in \F_{q^2}^{\times}$, write \begin{equation}\label{eq:def-gt}g_{t}:=1+p [t]\in U_D^1/Z_D^1.\end{equation}
Note that $\omega(g_t)=1$, so  $g_t-1\in \fm_D^2$ by Lemma \ref{lem:Ko-3.13}. Let  $u_t$ denote the image of $g_t-1$ in $\gr^2_{\fm_D}\F[\![U_D^1/Z_D^1]\!]$. 

\begin{proposition}\label{prop:yiyj}
(i) We have $[y_i,y_j]=0$ for any pair $(i,j)$ with $i-j\neq f$ (in $\Z/2f\Z$).

(ii) Set $h_j':=[y_j,y_{f+j}]$ for $0\leq j\le f-1$. Then $\{h_j',~0\leq j\leq f-1\}$ are linearly independent in $\gr^2_{\fm_D}\F[\![U_D^1/Z_D^1]\!]$ and  they span the same subspace as $\{h_j,~0\leq j\leq f-1\}$.

\end{proposition} 

\begin{proof}
A direct computation shows
\[
Y_iY_{j}=\sum_{\lambda,\mu\in\F_{q^2}^{\times}}\sigma_i(\lambda)^{-1}\sigma_{j}(\mu)^{-1} (1+\varpi_D[\lambda]+\varpi_D[\mu]+p[\lambda^{q}\mu]).
\]
We may write (in $U_D^1$)
\[
1+\varpi_D[\lambda]+\varpi_D[\mu]+p[\lambda^q\mu]=(1+\varpi_D[\lambda]+\varpi_D[\mu])(1+p[\lambda^q\mu]+x)
\]
with $x\in \varpi_D^3\cO_D$ and note that $(1+p[\lambda^q\mu]+x)-1$ has the same image as $(1+p[\lambda^q\mu])-1$ in $\gr^2_{\fm_D}\F[\![U_D^1/Z_D^1]\!]$ by Lemma \ref{lem:gh-1}. Using Lemma \ref{lem:gh-1} again, we have
\[
(1+\varpi_D[\lambda]+\varpi_D[\mu]+p[\lambda^q\mu])-1\equiv (h_{\lambda,\mu}-1)+(g_{\lambda^q\mu}-1) \mod \fm_D^3.
\] 
where  $h_{\lambda,\mu}:=1+\varpi_D[\lambda]+\varpi_D[\mu]$ and $g_{\lambda^q\mu}$ is defined by \eqref{eq:def-gt}. 
Similarly, we have
\[
Y_jY_i\equiv\sum_{\lambda,\mu\in\F_{q^2}^{\times}}\sigma_i(\lambda)^{-1}\sigma_{j}(\mu)^{-1}\big((h_{\lambda,\mu}-1)+(g_{\lambda\mu^q}-1)\big)\mod \fm_D^3 
\]
and so 
\[
[Y_i,Y_j]\equiv \sum_{\lambda,\mu\in\F_{q^2}^{\times}}\sigma_i(\lambda)^{-1}\sigma_j(\mu)^{-1}\big((g_{\lambda^q\mu}-1)-(g_{\lambda\mu^q}-1)\big)\mod \fm_D^3.
\]
Taking the image in $\gr^2_{\fm_D}\F[\![U_D^1/Z_D^1]\!]$ and noting that $\sigma_i(\lambda)=\sigma_{i-f}(\lambda^q)$, we obtain 
\[
[y_i,y_{j}]=\sum_{\lambda,\mu\in\F_{q^2}^{\times}}\frac{\sigma_j(\lambda^q)}{\sigma_{i-f}(\lambda^q)}\sigma_{j}(\lambda^q\mu)^{-1}(u_{\lambda^q\mu}-u_{\lambda\mu^q}).
\]   
  The map 
\[
  \F_{q^2}^{\times}\times \F_{q^2}^{\times}\ra \F_{q^2}^{\times},\ (\lambda,\mu)\mapsto \lambda^q\mu
 \]
is surjective and  each fibre is bijective to $\F_{q^2}^{\times}$ (by projecting to the second component), thus 
\[
[y_i,y_{j}]=\sum_{t\in\F_{q^2}^{\times}}\Big(\sum_{\lambda\in\F_{q^2}^{\times}} \frac{\sigma_j(\lambda^q)}{\sigma_{i-f}(\lambda^q)}\Big)\cdot \sigma_j(t)^{-1}(u_{t}-u_{t^q}). 
\]
If $i-j\neq f$, then $\sum_{\lambda\in\F_{q^2}^{\times}}\frac{\sigma_j(\lambda^q)}{\sigma_{i-f}(\lambda^q)}=0$ and so $[y_i,y_{j}]=0$, proving (i). 
If $i-j=f$, then the last sum equals to $-1$, and so
\[
[y_{j+f},y_j]=-\sum_{t\in\F_{q^2}^{\times}}\sigma_j(t)^{-1}(u_t-u_{t^q}).
\]
To prove (ii), one could argue as in Lemma \ref{lem:Schraen}, but this needs to make explicit the $h_j$'s.  Nonetheless,  we can conclude by the following observation: since $y_i$ lies in $\bigoplus_{0\leq j\leq f-1}(\F e_j\oplus \F f_j)$,  $[y_i,y_j]$ lies in the subspace spanned by $h_j=[e_j,f_j]$ (recall $[e_i,e_j]=[e_i,f_j] = [f_i,f_j]=0$ whenever $i\neq j$), and vice versa by (i) and Lemma \ref{lem:Schraen}(ii). 
\end{proof}

To make the notation more transparent, we write $z_i:=y_{i+f}$ for $0\leq i\leq f-1$. Lemma \ref{lem:Schraen} and Proposition \ref{prop:yiyj} imply that the Lie algebra $\F\otimes_{\F_p}\frak{g}_{\F_q}$ has another basis over $\F$ given by $\{y_j,z_j,h_j'; 0\leq j\leq f-1\}$, with $y_j$ and  $z_j$ in degree $1$, $h_j'$ in degree $2$  and satisfying the relations
\[
h_j'=[y_j,z_j],\ \ [y_i,z_j]=0\ \mathrm{if}\ i\neq j,\ \ [y_i,y_j]=[z_i,z_j]=[y_i,h_j']=[z_i,h_j']=0.
\]   

 Let $I_D$ be the left ideal of  $\gr_{\fm_D}  \F[\![ U^1_D /Z^1_D ]\!]$ generated by the degree two elements $y_jz_j$ and $h'_j$ for all $0\leq j\leq f-1$. The ideal $I_D$ is in fact a two-sided ideal of  $\gr_{\fm_D} \F[\![ U^1_D /Z^1_D ]\!]$; it is also the left ideal generated by $(y_j z_j,  h_j; ~ 0\leq j \leq f-1)$ by Proposition \ref{prop:yiyj}(ii).

\begin{corollary}\label{cor:hj'}
(i) The sequence $(h_0',\cdots,h_{f-1}')$ is a regular sequence of central elements of $\gr_{\fm_D}\F[\![U_D^1/Z_D^1]\!]$. The quotient $\gr_{\fm_D}\F[\![U_D^1/Z_D^1]\!]/(h_0',\dots,h_{f-1}')$ is commutative and is isomorphic to the polynomial ring   $\F[y_j,z_j;0\leq j\leq f-1]$.

(ii) The quotient $\gr_{\fm_D} \F[\![ U^1_D /Z^1_D ]\!] / I_D$ is isomorphic to $\F [y_j,z_j; 0\leq j\leq f-1]/(y_jz_j; 0\leq j\leq f-1).$
\end{corollary}
\begin{proof}
The proof of \cite[Thm.~5.3.4]{BHHMS1} goes through by the above discussion.\end{proof}

Let $\chi : \cO_D^{\times} \to \F^{\times}$ be a smooth character. Let $\Proj_{\F[\![\cO_{D}^{\times} / Z^1_D ]\!]} \chi$ denote the projective envelope of $\chi$ in the category of $\F[\![\cO_{D}^{\times} / Z^1_D]\!]$-modules. For $n\geq 1,$ let
\begin{equation}\label{def--W_chi}
W_{\chi_,n}: = (\Proj_{\F[\![\cO_{D}^{\times} / Z^1_D]\!]} \chi) / \frak{m}_{D}^n.
\end{equation}
It is clear that $W_{\chi,n}\cong \chi\otimes W_{\ide,n}$ where $\ide$ denotes the trivial character. 
The module $W_{\chi,3}$ is of particular importance to us.

\begin{corollary}\label{cor:structure-W_chi,3}
The module $W_{\ide,3}$ has the following graded structure:
\[
\gr^0 W_{\ide,3} = \F,~ \gr^1W_{\ide,3} = \bigoplus_{i=0}^{f-1}\F\alpha_i \oplus \F\a_i^{-1},\]\[ \gr^2 W_{\ide,3} = \F^{2f}\oplus \bigoplus_{0\leq i\leq j\leq f-1}\F\alpha_i\alpha_j\oplus \bigoplus_{0\leq i\leq j\leq f-1}\F\alpha_{i}^{-1}\alpha_{j}^{-1}\oplus \bigoplus_{0\leq i\neq j\leq f-1} \F\alpha_i\alpha_j^{-1},
\]
where $\alpha_j:\cO_D^{\times}\ra\F^{\times}$ is the character defined in \eqref{eq:def-alpha}. 
\end{corollary}
\begin{proof}
It follows from Corollary \ref{cor:hj'} using \eqref{eq:Yj-char}; cf. \cite[(44)]{BHHMS1}. 
\end{proof}

We have the following criterion which allows to control the Gelfand-Kirillov dimension of an admissible smooth $\F$-representation of $\cO_D^{\times}/Z_D^1$. It is  an  analogue of \cite[Cor.~5.3.5]{BHHMS1}. Let $\overline{W}_{\chi,3}$ denote the quotient of $W_{\chi,3}$ by the sum of characters which occur in $\gr^2W_{\chi,3}$ and  non-isomorphic to $\chi$.  For example, if $L=\Q_p$ then $\dim_{\F}\overline{W}_{\chi,3}=5$ and has a socle filtration as follows (with $\alpha=\alpha_0$):
\[
(\chi\oplus\chi)\ligne (\chi\alpha\oplus\chi\alpha^{-1})\ligne \chi. 
\]

\begin{corollary}\label{cor-gkdim-control}
Let $\pi$ be an admissible smooth representation of $\cO_D^{\times} / Z_D^1$ over $\F.$ Assume for each character $\chi$ such that $\Hom_{\cO_D^{\times}}(\chi , \pi) \neq 0,$ the natural injection
\begin{equation}\label{cond-gkdim}
\Hom_{\cO_D^{\times}}(\chi , \pi) \into \Hom_{\cO_D^{\times}}( \overline{W}_{\chi,3} , \pi)
\end{equation}
is an isomorphism. Then $\dim_{\cO_{D}^{\times}} (\pi) \leq f,$ where $\dim_{\cO_{D}^{\times}}(\pi)$  is the Gelfand-Kirillov dimension of $\pi$ over $\cO_{D}^{\times }$.
\end{corollary}
\begin{proof}
The Pontryagin dual  $\pi^{\vee}$ is naturally a finitely generated module over $\F[\![U_D^1/Z_D^1]\!]$ as $\pi$ is admissible,   so  the graded module $\gr_{\fm_D}(\pi^{\vee})$  is finitely generated over $\gr_{\fm_D}\F[\![U_D^1/Z_D^1]\!]$. The condition (\ref{cond-gkdim}) implies that   $ \gr^0_{\frak{m}_D} (\pi^{\vee}) $ is killed by $y_jz_j$ and $h_j'$ (for $1\leq j\leq f-1$), hence also by $I_D$.  The  result then follows from Corollary \ref{cor:hj'}; see    \cite[Cor.~5.3.5]{BHHMS1} for details.
\end{proof}

\subsection{$\Ext^i$ groups when $L=\Q_p$}

We assume $L=\Q_p$ with $p\geq 5$. We write $\alpha=\alpha_0$.  

\begin{proposition}\label{prop-Ext1-U1}
Let $\psi,\chi:\cO_D^{\times}\ra\F^{\times}$ be two smooth characters. Then $\Ext^1_{\cO^{\times}_D/Z_D^1}(\psi,\chi)$ is nonzero if and only if $\psi=\chi\alpha$ or $\psi=\chi\alpha^{-1}$. Moreover,
\[\dim_{\F}\Ext^1_{\cO_D^{\times}/Z_D^1}(\chi\alpha,\chi)=\dim_{\F}\Ext^1_{\cO_D^{\times}/Z_D^1}(\chi\alpha^{-1},\chi)=1.\]
\end{proposition}
\begin{proof}
It is a consequence of Corollary \ref{cor:structure-W_chi,3}. 
 \end{proof}

\begin{proposition}\label{prop-Poincare}
Let $\tau_1,\tau_2$ be finite dimensional smooth representations of $\cO_D^{\times}/Z_D^1$. Then there is an isomorphism $\Ext^i_{\cO_D^{\times}/Z_D^1}(\tau_1,\tau_2)\cong\Ext^{3-i}_{\cO_D^{\times}/Z_D^1}(\tau_2,\tau_1)^{\vee}$ for $0\leq i\leq 3$.
\end{proposition}
\begin{proof}
Firstly,  we have isomorphisms   \[\Ext^i_{\cO^{\times}_D/Z_D^1}(\tau_1,\tau_2)\cong\Ext^i_{\cO_D^{\times}/Z_D^1}(\ide,\tau_1^{\vee}\otimes\tau_2)\cong H^i(\cO_D^{\times}/Z_D^1,\tau_1^{\vee}\otimes\tau_2)\cong H^i(U_D^1/Z_D^1,\tau_1^{\vee}\otimes\tau_2)^{\F_{p^2}^{\times}}.\]
Secondly, since $U_D^1/Z_D^1$ is a Poincar\'e group of dimension $3$ (cf. \cite[\S4.5]{Se97}),  Poincar\'e duality induces an isomorphism  \[H^i(U_D^1/Z_D^1,\tau)\cong H^{3-i}(U_D^1/Z_D^1,\tau^{\vee})^{\vee}\]
for $0\leq i\leq 3$ and any finite dimensional representation $\tau$. The result easily follows.
\end{proof}

\section{Lattices in some locally algebraic representations of $\GL_2(\Z_p)$}\label{section-lattices-GL2}

Let $K: = \GL_2(\Z_p),$ $\G : = \GL_2(\F_p),$ and $K_1: = \Ker(K \onto \G)$. Let $I$ (resp. $I_1$) denote the upper Iwahori (resp. pro-$p$ Iwahori) subgroup of $K.$ Let $Z$ denote the centre of $G,$ $Z_1 : = Z\cap K_1.$ Let \[H : = \Big\{\matr{[a]}00{[d]}
,~ a,d\in \F_p^{\times}\Big\}.\]   Let $\a: H \to \F^{\times}$ be the character of $H$ sending $\bigl(\begin{smallmatrix}
[a] & 0 \\ 0& [d]
\end{smallmatrix} \bigr)$ to $a d^{-1}.$  By abuse of notation we also denote   the image of $H$ in $\Gamma$ by the same letter. If $\chi$ is a character of $H$, we denote by $\chi^s$ the character sending $h$ to $\chi(shs)$ 
where $s:=\smatr0110$. We regard a character of $H$ as a character of $I$ via the quotient map $I\twoheadrightarrow H$; note that any smooth $\F$-valued  character of $I$  arises in this way.

For $m,n\in\N$, we denote  \[\s_{m,n} := \Sym^m \F^2 \otimes {\det}^n\] which are naturally representations of $\G$ over $\F$. We also regard them as representations of $K$   via the natural projection $K\onto \G.$ Up to isomorphism the set $\{\s_{m,n},~ 0\leq m \leq p-1,~0\leq n \leq p-2\}$ forms a complete list of {\em irreducible} representations of $\G$ (and of $K$) over $\F.$ 

We choose the standard basis of $\sigma_{m,n}$ to be $\{X^iY^{m-i}; 0\leq i\leq m\}$, with the action of $\Gamma$  given by 
\[ \matr{a}bcd X^iY^{m-i}=(aX+cY)^i(bX+dY)^{m-i}.\] 
It is well-known that $\sigma_{m,n}^{I_1}$ is $1$-dimensional (spanned by $X^m$), on which $H$ acts via the character sending $\bigl(\begin{smallmatrix}
{[a]} & 0 \\ 0& {[d]}
\end{smallmatrix} \bigr)$ to $a^{m+n} d^n$, which we denote by   $\chi_{m,n}$. Similarly, the space of coinvariants $(\sigma_{m,n})_{I_1}$ is  $1$-dimensional on which $H$ acts via $\chi_{m,n}^s$. 
 
Recall $E := W(\F)[1/p],$ where $\cO:=W(\F)$ is the ring of Witt vectors in $\F.$ If $V$ is a finite dimensional representation of $K$ over $E$, then $V^{\circ}$ will denote a $K$-stable $\cO$-lattice in $V$ and $\overline{V^{\circ}}$ its reduction   modulo $p$. We will write  $\overline{V}^{\rm ss}$ for the semi-simplification of $\overline{V^{\circ}}.$ Following \cite{EGS}, we say $V$ is {\em residually multiplicity free} if any of the Jordan-H\"older factors of $\overline{V}^{\rm ss}$ occurs with multiplicity one. In this section,  a lattice \emph{always} means a $K$-stable $\cO$-lattice. 

\subsection{Preliminaries} 

Denote by $U(\Z_p)$ (resp. $B(\Z_p)$)   the (upper) unipotent (resp.  Borel) subgroup of $K$. Note that $H$ normalizes $U(\Z_p)$. 

\begin{proposition}\label{prop:U-inv}
Let $W$ be a finite dimensional $\F$-representation of $B(\Z_p)$, of dimension $\geq 2$. 

(i) Assume that $W^{U(\Z_p)}$ is $1$-dimensional and isomorphic to $\chi$ as an $H$-representation. Then \[(\Sym^{1}\F^2\otimes W)^{U(\Z_p)}\cong \chi\chi_{1,0}\oplus \chi\chi_{1,0}^s.\]

(ii) Assume that $W_{U(\Z_p)}$ is $1$-dimensional and isomorphic to $\chi$ as an $H$-representation. Then \[(\Sym^1\F^2\otimes W)_{U(\Z_p)}\cong\chi\chi_{1,0}\oplus \chi\chi_{1,0}^s.\]

\end{proposition}

\begin{proof}
(i) Let $W_0:=W^{U(\Z_p)}\cong \chi$. We first prove that  $(W/W_0)^{U(\Z_p)}$ is $1$-dimensional and isomorphic to $\chi\alpha^{-1}$ as an $H$-representation. Since $\dim_{\F}W\geq 2$ by assumption, $W/W_0$ is nonzero, hence $(W/W_0)^{U(\Z_p)}$ is also nonzero because $U(\Z_p)$ is a pro-$p$-group. On the other hand,   
we have an $H$-equivariant injection
\[0\ra (W/W_0)^{U(\Z_p)}\ra H^1(U(\Z_p),W_0)\]
which is actually an isomorphism because $H^1(U(\Z_p),\chi)\cong \chi\alpha^{-1}$ is $1$-dimensional (see e.g.~\cite[Lem.~5.5]{Pa10}).   This proves the claim.

Any element  $w\in \Sym^1\F^2\otimes W$ can be written as 
$Y\otimes w_0+X\otimes w_1$
for (unique) $w_0,w_1\in W$. Let $g=\smatr{1}{t}01\in U(\Z_p)$. Then 
\[ 
gw=(\bar{t}X+Y)\otimes gw_0+X\otimes gw_1 
=Y\otimes gw_0+X\otimes(\bar{t}\cdot gw_0+gw_1).
\]
Hence $w$ is fixed by $U(\Z_p)$ if and only if 
\[\left\{\begin{array}{ll}
gw_0=w_0\\
gw_1=w_1-\bar{t}gw_0.
\end{array}\right.\]
We have two cases:
\begin{itemize}
\item[(a)] If $w_0=0$, then the above condition becomes $gw_1=w_1$, i.e. $w_1\in W_0$.
\item[(b)] If $w_0\neq0$, then $w_0\in W_0$ and $w_1\in (W/W_0)^{U(\Z_p)}$. Moreover, $(W/W_0)^{U(\Z_p)}$ is $1$-dimensional as seen above and  the condition $gw_1=w_1-\bar{t}w_0$ determines uniquely $w_1$ (whenever $w_0\neq0$ is fixed). \end{itemize}
The result easily follows.

(ii) It follows from (i) via the  fact that   $(W_{U(\Z_p)})^{\vee}\cong (W^{\vee})^{U(\Z_p)}$ (and similarly for $\Sym^1\F^2\otimes W$).  
\end{proof}

\begin{corollary}\label{cor:U-inv}
Let $V=\Ind_{I}^{K}\chi$ for some smooth character $\chi:I\ra \F^{\times}$. Then \[(\Sym^1\F^2\otimes V)^{U(\Z_p)}\cong \chi^s\chi_{1,0}\oplus\chi^s\chi_{1,0}^s\oplus \chi\chi_{1,0},\]
\[(\Sym^1\F^2\otimes V)_{U(\Z_p)}\cong \chi^s\chi_{1,0}\oplus\chi^s\chi_{1,0}^s\oplus \chi\chi_{1,0}^s.\]
\end{corollary}
\begin{proof}
Mackey's decomposition theorem  gives an isomorphism $V|_{I}\cong \chi\oplus V'$, where $V':= \Ind_{HK_1}^{I}\chi^s$. It is easy to see that $V'$ has dimension $p$, and $V'^{U(\Z_p)}\cong V'_{U(\Z_p)}\cong\chi^s$. Thus Proposition \ref{prop:U-inv} applies to $\Sym^{1}\F^2\otimes V'$. The results then follow by noting that $(\Sym^1\F^2\otimes \chi)^{U(\Z_p)}\cong \F X\otimes \chi$ and $(\Sym^1\F^2\otimes \chi)_{U(\Z_p)}\cong \F Y\otimes \chi$.
\end{proof}

Consider the following situation: $V_1, V_2$ are two  irreducible  locally algebraic representations of $K$, and $L_i\subset V_i$ is a lattice for $i=1,2$. Assume that we are given an $\F[K]$-module $W$, together with $K$-equivariant morphisms $r_i:L_i\ra W$.  Let $L$ be the fibered product of $r_1$ and $r_2$, namely    
\begin{equation}\label{eq:glue-L}0\ra L\ra L_1\oplus L_2\To{r_1-r_2} W.\end{equation}
Then $L$ is a lattice in $V_1\oplus V_2$. 
We also call $L$ the gluing lattice in $L_1$ and $L_2$ along $W$. Remark that, if either $r_1$ or $r_2$ is surjective, then so is $r_1-r_2$. 

\begin{lemma}\label{lem:glue-L/p}
Assume that $r_1$ is surjective.
\begin{enumerate}
\item[(i)] There exists a short exact sequence
\[0\ra \Ker(r_1)/p\Ker(r_1)\ra L/pL\ra L_2/pL_2\ra0.\]

\item[(ii)] Let $r_L$ denote the composite morphism $L\ra L/pL\ra  L_2/pL_2$, where the second map is as in (i). Then $\Ker(r_L)=\Ker(r_1)+pL$ and
\[\Ker(r_L)/p\Ker(r_L)\cong \Ker(r_1)/p\Ker(r_1)\oplus pL_2/p^2L_2.\] 
\end{enumerate}
\end{lemma}

\begin{proof}
(i) We have the following commutative diagram
\[\xymatrix{0\ar[r] & \Ker(r_1)\ar[r]\ar@{^(->}[d]&L_1\ar^{r_1}[r]\ar^{\mathrm{id}\oplus 0}@{^(->}[d]&W\ar[r]\ar@{=}[d]&0\\
0\ar[r]&L\ar[r]&L_1\oplus L_2\ar^{\ \ r_1-r_2}[r]&W\ar[r]&0.}\]
By the snake lemma, it  induces  a short exact sequence 
$0\ra \Ker(r_1)\ra L\ra L_2\ra0$. We obtain the result by taking  mod $p$ reduction (as $L_2$ is $\cO$-flat).

(ii) It is clear from (i) that $\Ker(r_L)=\Ker(r_1)+pL$, so we have a short exact sequence
\[0\ra \Ker(r_1)\cap pL\ra \Ker(r_1)\oplus pL\ra \Ker(r_L)\ra0.\]
Taking mod $p$ reduction and noting that $\Ker(r_1)\cap pL=p\Ker(r_1)$ by (i), we obtain an exact sequence 
\[0\ra p\Ker(r_1)/p^2\Ker(r_1)\ra \Ker(r_1)/p\Ker(r_1)\oplus pL/p^2L\ra \Ker(r_L)/p\Ker(r_L)\ra0. \] 
But the map $p\Ker(r_1)/p^2\Ker(r_1)\ra \Ker(r_1)/p\Ker(r_1)$ is identically zero, so the result follows from (i).
\end{proof}

\begin{lemma}\label{lem:cosoc-glue}
Assume that both $r_1$ and $r_2$ are surjective. Assume moreover   
\begin{enumerate}
\item[(a)]  $\rcosoc(L_1)=\rcosoc(W)$;  

\item[(b)] $\rcosoc(\Ker(r_1))$ and $\rcosoc(\Ker(r_2))$ do not admit common Jordan-H\"older factors. 
\end{enumerate} Then
$\rcosoc(L)\cong \rcosoc(L_2)$.
\end{lemma}
\begin{proof}
We need to show that the natural map
\[\Hom_{\cO[K]}(L_2,\sigma)\ra\Hom_{\cO[K]}(L,\sigma)\]
is an isomorphism for any Serre weight $\sigma$. 
By applying $\Hom_{\cO[K]}(-,\sigma)$ to \eqref{eq:glue-L} we obtain a long exact sequence
  \begin{multline*}0\ra \Hom(W,\sigma)\ra\Hom(L_1,\sigma)\oplus \Hom(L_2,\sigma)\ra \Hom(L,\sigma)\\ \ra \Ext^1(W,\sigma)\ra \Ext^1(L_1,\sigma)\oplus \Ext^1(L_2,\sigma).  \end{multline*}
By (a), the surjection $r_1:L_1\twoheadrightarrow W$ induces an isomorphism  $\Hom(W,\sigma)\simto \Hom(L_1,\sigma)$. To conclude we need to show that the morphism 
\[\Ext^1(W,\sigma)\ra \Ext^1(L_1,\sigma)\oplus \Ext^1(L_2,\sigma)\]
is injective. For this it is enough to prove that either $\Hom(\Ker(r_1),\sigma)$ or $\Hom(\Ker(r_2),\sigma)$ vanishes, which is a consequence of  (b).
\end{proof}

Finally, we record a result which will be used later on. 

\begin{proposition}
Let $V$ be an irreducible smooth representation of $K$ over $E$. Then the $K$-representation $ \Sym^1 E^2 \otimes V $ is again irreducible.
\end{proposition}

\begin{proof}
This is \cite[Prop.~3.4]{Prasad}.
\end{proof}

\subsection{Lattices in tame types} \label{sec::lattices-in-tame-types}

We consider the following representations of $\G$ over $E,$ and  view them as smooth representations of $K$ via the projection $K \onto \G.$ 
\begin{enumerate}
\item[$\bullet$] Let $\chi_1,\chi_2: \F_{p}^{\times} \to E^{\times}$ be two characters. Let $I(\chi_1,\chi_2)$ denote the principal series representation $\Ind_{B(\F_p)}^{\Gamma}\chi_1\otimes \chi_2$, where $B(\F_p)$ is the (upper) Borel subgroup of $\G$.   It is well-known that $I(\chi_1,\chi_2)$ is irreducible if $\chi_1 \neq \chi_2.$ If $\chi_1 =\chi_2 =\chi,$ then 
\[
I(\chi, \chi) \cong (\chi \circ\det) \oplus \  ({\rm sp} \otimes\chi\circ\det),
\]
where ${\rm sp}$ denotes the  Steinberg representation. 
\item[$\bullet$]Let $\psi :\F_{p^2}^{\times} \to E^{\times}$ be a character which doesn't factor through the norm map $ \F_{p^2}^{\times} \to \F_{p}^{\times}.$ This is equivalent to requiring $\psi \neq \psi^p.$ There is an irreducible $(p-1)$-dimensional representation $\Theta(\psi)$ characterized by the isomorphism  $\Theta(\psi)\otimes {\rm sp} \cong \Ind_{\F_{p^2}^{\times}}^{\GL_2(\F_p)} \psi,$ where $\F_{p^2}^{\times} \into \GL_2(\F_p)$ is a fixed group embedding. For two such characters $\psi,\psi',$ $\Theta(\psi) \cong \Theta(\psi')$ if and only if $\psi'\in \{\psi, \psi^p\}.$ 
\end{enumerate}
The Jordan-H\"older factors of the reduction mod $p$ of any lattice in the above representations are determined in \cite{Diamond}. We recall the results in the next proposition.

Let $x: \F_p \hookrightarrow \F$ denote the natural embedding and $[x]: \F_p \to \cO$ be the Teichm\"uller lift of $x$ which will be viewed as a multiplicative character of $\F_p^{\times}.$ Let $\xi: \F_{p^2} \hookrightarrow \F$ be an embedding extending $x.$ Let $\xi':= \xi^p$ and $\zeta: = \xi \xi'.$ Let $[\xi] : \F_{p^2} \to \cO$ be the Teichm\"uller lift of $\xi$ which will be viewed as a multiplicative character of $\F_{p^2}^{\times}.$ We have $[x]^{p-1} = \ide$ and $[\xi]^{p+1} = [x].$

\begin{proposition}\label{prop:Diamond}
(i) Let   $0\leq a \leq p-1$ and $0\leq b\leq p-2.$ Then
\[
\overline{I([x]^b , [x]^{b+a})}^{\rm ss} \cong \s_{a,b} \oplus \s_{p-1-a, a+b}.
\]

 (ii) Let $\psi:\F_{p^2}^{\times}  \to E^{\times}$ with $\psi\neq \psi^p.$ Write $\psi = [\xi]^{a+1 + (p+1)b}$ with $0\leq a \leq p-1$ and $0\leq b\leq p-2.$ Then
\[
 \overline{\Theta(\psi)}^{\rm ss} \cong \s_{a-1, b+1} \oplus  \s_{p-2-a, a+b+1},
\]
with the convention that $\s_{-1, b} = 0.$

 (iii) The representations $I([x]^b , [x]^{b+a})$ and $ \Theta(\psi) $ are residually multiplicity free.
\end{proposition}
\begin{proof}
(i) follows from \cite[Prop.~1.1]{Diamond}; (ii) follows from \cite[Prop.~1.3]{Diamond}. (iii) follows directly from (i) and (ii).
\end{proof}

We recall Lemma 4.1.1 of \cite{EGS} on the lattices of finite dimensional irreducible residually multiplicity free $E$-representations of $K$.

\begin{proposition}[\cite{EGS}]\label{prop-lattice-EGS}
Let $V$ be a finite dimensional irreducible representation of $K$ over $E$ which is residually multiplicity free. Let $\s$ be a Jordan-H\"older factor of $\overline{V}^{\rm ss}.$ Then there is up to homothety a unique lattice $V^{\circ}_{\s}$ in $V$ such that the socle of $\overline{V^{\circ}_{\s}}$ is $\s.$ Similarly, there is up to homothety a unique lattice $V^{\circ,\s}$ in $V$ such that the cosocle of $\overline{V^{\circ,\s}}$ is $\s.$
\end{proposition}

\subsection{Lattices in $\unSym^1E^2 \otimes  \Theta(\psi)$}
\label{subsection:lattice}

\hfill

Let  $\mathrm{pr}:\Q_p^{\times}\ra 1+p\Z_p$ denotes the projection sending $p$ to $1$.  
As $p>2$, we can define the square root on $1+p\Z_p$ by the usual binomial formula. Define
\begin{equation}\label{eq:def-unSym}
\unSym^1E^2:=\Sym^{1}E^2\otimes (\mathrm{pr}\circ\det)^{-1/2}.
\end{equation}
The reason to introduce the twist is to make the central character of $\unSym^1E^2$ to be trivial on $Z_1$. Note that the mod $p$ reduction of $\unSym^1\cO^2:=\Sym^{1}\cO^2\otimes (\mathrm{pr}\circ\det)^{-1/2}$  still gives $\Sym^1\F^2$.

Let $\psi: \F_{p^2}^{\times} \to E^{\times}$ be a character with $\psi \neq \psi^p.$ Write $\psi = [\xi]^{a+1 + (p+1)b}$ with $0\leq a \leq p-1$ and $0\leq b\leq p-2.$ By Proposition \ref{prop:Diamond},  $\overline{\Theta(\psi)}^{\rm ss}$ is multiplicity free and has two (resp. one) Jordan-H\"older factors if $1 \leq a \leq p-2$ (resp. if $a\in\{0,p-1\}$).

Assume first $1\leq a\leq p-2$. By Propositions \ref{prop:Diamond} and \ref{prop-lattice-EGS}, there are two lattices $T,~T'$ in $\Theta(\psi)$ such that 
\begin{equation}\label{eq:T/pT}
0\ra \s_{p-2-a, a+b+1}  \ra T/pT\ra \s_{a-1, b+1}\ra 0
\end{equation}
\begin{equation}\label{eq:T'/pT'}
0\ra  \s_{a-1, b+1} \ra  T'/pT'\ra  \s_{p-2-a, a+b+1}\ra 0
\end{equation}
where both extensions are nonsplit. Note that $T/pT$ and $T'/pT'$ are $\Gamma$-representations as $\Theta(\psi)$ itself is. Moreover, if we fix $T$ and normalize $T'$ (by a scalar) so that $T'\subset T$ and $T'\nsubseteq pT$, then by \cite[Prop.~5.2.3(1)]{EGS} we have
\begin{equation}\label{eq:T-T'}
pT\subset T'\subset T.\end{equation}

\begin{lemma}\label{lemma-I1-coinv-T/pT}
 (i) We have $(T/pT)^{I_1}\cong \chi_{p-2-a,a+b+1}$ and $(T'/pT')^{I_1}\cong \chi_{a-1,b+1}$. 

(ii) We have $(T / pT)_{I_1} = \chi^s_{ a -1, b +1}$ and $(T'/pT')_{I_1}\cong\chi_{p-2-a,a+b+1}^{s}$. 
\end{lemma}
\begin{proof}
(i) We only give the proof in the case of $T/pT$. 
Using \eqref{eq:T/pT} we obtain an exact sequence
\[
0\ra (\s_{p-2-a,a+b+1})^{I_1} \to (T/pT)^{I_1} \to (\s_{a-1 ,b+1})^{I_1}.
\]
Assume for a contradiction that $(T / pT)^{I_1}$ is $2$-dimensional.  Then  we would obtain \[(T/pT)^{I_1}\cong \chi_{p-2-a,a+b+1}\oplus \chi_{a-1,b+1},\] 
and consequently an $I$-equivariant injection $\chi_{a-1,b+1}\hookrightarrow T/pT.$ By Frobenius reciprocity, we would get a nonzero $K$-equivariant map
\[
 \Ind_{I}^{K} \chi_{a-1,b+1}\ra T/pT.
\]
 By comparing the Jordan-H\"older factors, this map can not be injective and must have image isomorphic to $\sigma_{a-1,b+1}$ (see \cite[Lem.~2.3]{BP}).  This gives a contradiction because the sequence \eqref{eq:T/pT} is nonsplit.

(ii) It is proved in a similar way as (i). Alternatively, it can be deduced from (i) by taking dual.
\end{proof}

Recall that $E$ is unramified over $\Q_p.$ Consider $\unSym^1\cO^2:=\cO Y\oplus \cO X$, the standard lattice in $\unSym^1E^2$ 
and set
\begin{equation}\label{eq:def-L}
L:=  \unSym^1\cO^2\otimes_{\cO} T\end{equation}
\begin{equation}\label{eq:def-L'}
L' := \unSym^1\cO^2\otimes_{\cO} T'.\end{equation}
Then we have\footnote{To remind of the distinguished role of $\Sym^{1}\F^2$, here and below  we write $\Sym^{1}\F^2$ instead of $\sigma_{1,0}$.}
\begin{align*}
L / p L & \cong \Sym^1\F^2\otimes T/pT\\
L' / p L' & \cong \Sym^1\F^2\otimes T'/p T',
\end{align*}
and \eqref{eq:T-T'} implies $p L \subset L' \subset L$. 
\begin{lemma}\label{lem:K1-trivial}
$K_1$ acts trivially on $L/pL$ and $L'/pL'.$
\end{lemma}
\begin{proof}
This is because $K_1$ acts trivially on both $\Sym^1\F^2$ and $\Theta(\psi)$.
\end{proof}

\begin{lemma}\label{lem:L-I1}
(i) We have $(L/pL)^{I_1}\cong\chi_{p-1-a,a+b+1}\oplus\chi_{p-3-a,a+b+2}$ and $(L'/pL')^{I_1}\cong \chi_{a,b+1}\oplus\chi_{a-2,b+2}$.

 (ii) We have $(L / p L)_{I_1} \cong \chi^s_{ a, b +1} \oplus \chi^s_{ a - 2, b +2} $ and $(L'/pL')_{I_1}\cong \chi_{p-1-a,a+b+1}^s\oplus \chi_{p-3-a.a+b+2}^s$. 
\end{lemma}
\begin{proof}
By Lemma \ref{lem:K1-trivial}, we have $(L/pL)^{I_1}=(L/pL)^{U(\Z_p)}$ and $(L/pL)_{I_1}=(L/pL)_{U(\Z_p)}$, so the results follow from Proposition \ref{prop:U-inv} and  Lemma \ref{lemma-I1-coinv-T/pT}.
\end{proof}

\begin{proposition}\label{prop-reduction-L}
Assume $1\leq a \leq p-2.$

(i)   $L/pL$ is  multiplicity free and has a two-step socle (and cosocle) filtration
\begin{equation}\label{eq:L=twostep}
(\s_{p-3-a, a+b+2} \oplus\s_{p-1-a, a+b+1} )  \ligne (\s_{a, b+1} \oplus \s_{a-2, b+2})
\end{equation}
(with the convention $\s_{-1,b+1} = \s_{-1,b+2} =0$). Moreover, the following nonsplit extensions 
\begin{align*}
E_1 & = (\s_{p-3-a, a+b +2} \ligne \s_{a, b+1})\\
 E_2 & = (\s_{p-1-a, a+b+1} \ligne \s_{a-2, b+2})\\
 E_3 &= (\s_{p-1-a, a+b+1} \ligne \s_{a, b+1})
\end{align*} 
occur in  $L/pL$ as subquotients,  with the exception that $E_1$ (resp. $E_2$) doesn't exist if $a=p-2$ (resp. $a=1$). 

(ii)  $L'/pL'$ is multiplicity free and has a two-step socle (and cosocle) filtration
\[
 (\s_{a, b+1} \oplus \s_{a-2, b+2}) \ligne (\s_{p-3-a, a+b+2} \oplus\s_{p-1-a, a+b+1} ).
\]
(with the convention $\s_{-1,b+1} = \s_{-1,b+2} =0$). Moreover, the following nonsplit extensions 
\begin{align*}
E'_1 &= (\s_{a, b+1} \ligne  \s_{p-3-a, a+b +2} )\\
E'_2 & = ( \s_{a-2, b+2} \ligne \s_{p-1-a, a+b+1})\\
E'_3 & = (\s_{a, b+1}  \ligne \s_{p-1-a, a+b+1} )
\end{align*} occur in $L'/pL'$ as subquotients, with the exception that $E'_1$ (resp. $E'_2$) doesn't exist if $a=p-2$ (resp. $a=1$). 
\end{proposition}

\begin{proof} 
It suffices to prove (i).  Recall the following facts (see \cite[Lem.~3.8]{BP})
\begin{align*}
\Sym^1\F^2\otimes \s_{a-1, b+1} &\cong \s_{a, b+1} \oplus \s_{a-2, b+2} \\
\Sym^1\F^2\otimes\s_{p-2-a,a+b+1} &\cong \sigma_{p-1-a,a+b+1}\oplus\s_{p-3-a,a+b+2},
\end{align*}
with the convention $\s_{-1,b+1} = \s_{-1,b+2}= 0$. Using \eqref{eq:T/pT} this implies that  $L/pL\cong \Sym^1\F^2\otimes T/pT$ has a two-step filtration as claimed in \eqref{eq:L=twostep}, and is multiplicity free. By Lemma \ref{lem:L-I1},  the filtration gives exactly the socle (and cosocle) filtration. This also completes the proof if $a\in\{1,p-2\}$.  

Assume $2\leq a \leq p-3$ in the rest of the proof. For a Serre weight $\sigma$, denote by ${\rm Inj}_{\Gamma}(\s)$ the injective  envelope of $\s$ in the category of $\F[\Gamma]$-modules; we remark that $\rInj_{\G}(\s)$ is also projective. Let $W_1$ denote the image of the unique (up to scalar) nonzero map ${\rm Inj}_{\G}(\s_{a-2,b+2}) \to L/ p L$. Since  $\s_{p-3-a, a+b+2}$ is not a Jordan-H\"older factor of ${\rm Inj}_{\G}(\s_{a-2,b+2})$ (see \cite[Lem.~3.2]{BP}), $W_1$ does not admit $\sigma_{p-3-a,a+b+2}$ as a subquotient. Since $\rcosoc(W_1)\cong \sigma_{a-2,b+2}$ by construction, we deduce from \eqref{eq:L=twostep} that \[W_1\cong  (\s_{p - 1-a, a+b+1} \ligne \s_{a-2, b+2}),\]
i.e. the nonsplit extension $E_2$ occurs in $L/pL$. Consequently, the cokernel of the inclusion $W_1\into L/ p L$, denoted by $W_2$, has $\{\sigma_{a,b+1},\s_{p-3-a,a+b+2}\}$ as the set of Jordan-H\"older  factors, hence is isomorphic to the nonsplit extension $E_1=(\s_{p- 3-a, a+b+2} \ligne \s_{a , b+1})$ because $\s_{p-3-a,a+b+2}$ does not occur in the cosocle of $W_2$ by \eqref{eq:L=twostep}. 

We are left to show that $L / pL$ is a nonsplit extension of $W_2$ by $W_1$ (this  implies that $E_3$ occurs in $L/pL$). 
Assume for a contradiction that $L / pL \cong W_1 \oplus W_2.$ Let $V$ denote the principal series $\Ind_{I}^{K}\chi_{a+1,b}^s$ which is isomorphic to the  (unique) nonsplit extension $(\s_{a+1,b}\ligne \s_{p-2-a,a+b+1})$.  By \cite[\S3]{BP}, there exists a short exact sequence
\[
0 \to T / pT \to {\rm Inj}_{\G} (\s_{p - 2 - a, a+b +1}) \to V \to 0
\]
which induces a short exact sequence
\[
0 \to L/pL \to \Sym^1 \F^2 \otimes {\rm Inj}_{\G} (\s_{p - 2 - a, a+b +1}) \to \Sym^1 \F^2 \otimes 
V \to 0.
\]
By Lemma \ref{lem--injotimesSym} below, if $2 \leq a \leq p-4,$ then
\[
\Sym^1 \F^2 \otimes {\rm Inj}_{\G} (\s_{p - 2 - a, a+b +1}) = {\rm Inj}_{\G} (\s_{p -1-a, a+b+1}) \oplus {\rm Inj}_{\G} (\s_{p-3-a,a+b+2}).
\]
Comparing the socles, it is clear that $W_2\cap {\rm Inj}_{\G} (\s_{p -1-a, a+b+1}) = 0$,   thus $W_2\hookrightarrow 
 {\rm Inj}_{\G} (\s_{p-3-a,a+b+2})$.  Moreover, we have 
\[L/pL \cap \rInj_{\G}\sigma_{p-3-a,a+b+2}=W_2\]
which induces a (nonzero) morphism
\[
{\rm Inj}_{\G} (\s_{p-3-a,a+b+2}) / W_2  \ra\Sym^1 \F^2 \otimes V.
\]
However, by \cite[\S3]{BP} we have \[{\rm Inj}_{\G} (\s_{p-3-a,a+b+2}) / W_2 \cong \Ind_{B(\F_p)}^{\Gamma}\chi_{p-3-a,a+b+2},\]
so by Frobenius reciprocity we obtain a nonzero $I$-equivariant morphism 
\[\chi_{p-3-a,a+b+2}\ra \Sym^1\F^2\otimes V.\] But this contradicts Corollary \ref{cor:U-inv},  by which $(\Sym^{1}\F^2\otimes V)^{I_1}\cong \chi_{a+2,b}\oplus \chi_{a,b+1}\oplus \chi_{p-1-a,a+b+1}$.

The case $a = p-3$ is a little subtle.  By Lemma \ref{lem--injotimesSym} below we have \[\Sym^1 \F^2 \otimes {\rm Inj}_{\G} (\s_{1, b -1}) = {\rm Inj}_{\G} (\s_{2, b- 1}) \oplus {\rm Inj}_{\G} (\s_{0,b}) \oplus \s_{p-1,b}.\]
Comparing the socles, one checks that $W_2 $ embeds into ${\rm Inj}_{\G} (\s_{0,b})$ and actually
\[L/pL\cap \rInj_{\Gamma}(\s_{0,b})=W_2.\] Hence, we obtain a nonzero morphism from ${\rm Inj}_{\G} (\s_{0,b})/ W_2 \cong \s_{0,b} $ to $\Sym^1\F^2 \otimes V$. On the other hand, $\sigma_{p-1,b}$ also occurs in $\Sym^1\F^2\otimes V$, and in fact is a direct summand because $\sigma_{p-1,b}$ is an injective $\F[\Gamma]$-module. Thus  there exists an embedding 
\[\s_{0,b}\oplus\s_{p-1,b}\hookrightarrow \Sym^1\F^2\otimes V.\]
However, by Corollary \ref{cor:U-inv} we have $(\Sym^{1}\F^2\otimes V)^{I_1}\cong \chi_{p-1,b}\oplus \chi_{p-3,b+1}\oplus \chi_{2,b-1}$, in which the character $\chi_{p-1,b}$ ($=\chi_{0,b}$) occurs only once.   This gives a contradiction  and finishes the proof.
\end{proof}

Recall the following facts (see \cite[\S3]{BP}): ${\rm Inj}_{\G} (\s_{a , b})$ is of dimension $2p$ if $1\leq a \leq p-2$; ${\rm Inj}_{\G}(\s_{p-1,b}) \cong \s_{p-1 , b}$ is of dimension $p$;  ${\rm Inj}_{\G} (\s_{0,b}) \cong (\s_{0,b} \ligne \s_{p-3, b+1} \ligne \s_{0,b})$ is of dimension $p.$

\begin{lemma}\label{lem--injotimesSym}
(i) If $a = 1,$ then $\Sym^1 \F^2 \otimes {\rm Inj}_{\G} (\s_{1, b}) \cong {\rm Inj}_{\G} (\s_{2, b})\oplus {\rm Inj}_{\G}(\s_{0, b+1}) \oplus \s_{p-1, b+1}.$

(ii) If $a = p-2,$ then  $\Sym^1 \F^2 \otimes {\rm Inj}_{\G} (\s_{p-2, b}) \cong\s_{p-1, b}\oplus \s_{p-1, b} \oplus {\rm Inj}_{\G}(\s_{p-3, b+1}).$

(iii) If $0 \leq a \leq p - 1 $ and $a\notin \{1,p-2\},$ then $\Sym^1 \F^2 \otimes {\rm Inj}_{\G} (\s_{a, b}) = {\rm Inj}_{\G} (\s_{a+1, b})\oplus {\rm Inj}_{\G}(\s_{a - 1, b+1})$ with the convention ${\rm Inj}_{\G}(\s_{- 1, b+1}) = {\rm Inj}_{\G}(\s_{p, b}) =0.$
\end{lemma}

\begin{proof}
Using the fact that $\Sym^{1}\F^2\otimes {\rm Inj}_{\G}\sigma$ is an injective object in the category of $\F[\Gamma]$-modules, the results can be easily deduced from  \cite[Lem.~3.8]{BP}.
\end{proof}
Finally, we treat the case $a\in \{0,p-1\}$.

\begin{proposition}\label{prop-reduction-L-2}
Assume $a\in \{0,p-1\}.$ There are two lattices (unique up to homothety) $L, L'$ of $\unSym^1 E^2\otimes\Theta(\psi)$ such that $pL \subset L' \subset L$ and 
\begin{align*}
L/pL & \cong \s_{p-1, b+1} \oplus \s_{p-3, b+2},\\
L'/pL' & \cong (\s_{p-1, b+1} \ligne \s_{p-3, b+2}).
\end{align*}
The lattice $pL$ is then identified with the kernel of the natural projection $L' \onto \s_{p-3,b+2}.$ Moreover, $(L'/pL')_{K_1} \cong \s_{p-3, b+2}.$
\end{proposition}
\begin{proof}
Let $T$ be any lattice in $\Theta(\psi) $ and  $L: = \unSym^1 \cO^2 \otimes T.$  Then $T/pT \cong \s_{p-2, b+1}$ by Proposition \ref{prop:Diamond}, and consequently $L/pL \cong \s_{p-1, b+1} \oplus \s_{p-3,b+2}$ by \cite[Lem.~3.8]{BP}. 
 Let $L'$ be the kernel of the composition $L \to L/pL \overset{p_1}{\twoheadrightarrow} \s_{p-1,b+1}.$ Then $pL \subsetneq L'\subsetneq L$. Moreover, we have a short exact sequence  
 \begin{equation}\label{ses:a=0,p-1}
0\to pL/pL' \to L'/pL' \to \s_{p-3,b+2} \to 0.
\end{equation}
We claim that  (\ref{ses:a=0,p-1}) induces an isomorphism $(L'/pL')_{K_1} \simto \s_{p-3,b+2};$ this will imply that $L'/pL'$ is a nonsplit extension of $\sigma_{p-3,b+2}$ by $\sigma_{p-1,b+1}$. 

The proof of Proposition \ref{prop:U-inv} shows that there exist $\overline{w}_0,\overline{w}_1\in  T/pT$ such that $X\otimes \overline{w}_0$ and $Y\otimes\overline{w}_0+X\otimes \overline{w}_1$ span $(L/pL)^{I_1}$. 
Comparing the $H$-action, we must have 
\[(\s_{p-1,b+1})^{I_1}=\F (X\otimes \overline{w}_0),\ \ (\s_{p-3,b+2})^{I_1}=\F (Y\otimes \overline{w}_0+X\otimes \overline{w}_1).\]
 Let $w_0,w_1 \in T$ be a lift of $\overline{w}_0,\overline{w}_1$ respectively.  From the definition of $L'$ we see that $Y\otimes w_0+X\otimes w_1\in L'$.  As $K_1$ acts trivially on $T,$ we have
\[
\big(\smatr{1}p01 - 1\big) (Y\otimes w_0 + X \otimes w_1) = (pX )\otimes w_0  \in pL.
\]  
Since $(pX )\otimes \overline{w}_0 $ generates $pL/pL',$ the claim follows.

The uniqueness of $L'$ (up to homothety) follows from Proposition \ref{prop-lattice-EGS}. Since $pL$ is identified with the kernel of the natural projection $L' \onto \s_{p-3,b+2},$ the uniqueness of $L$ follows.
\end{proof}

\subsubsection{Sublattices in $L$}

In this subsection we specify some sublattices in $L$ in the case $1 \leq a \leq p-2.$
Recall that  $\s_{-1, b+1} = \s_{-1, b+2} = 0$ by our convention.

Let $L_1 : = \Ker( L \onto L/p L \onto \s_{a-2, b+2} ).$ It is clear that $pL\subset L_1\subset L$. 

\begin{proposition} \label{prop--sublattice-1}
The following nonsplit  extensions 
\begin{align*}
&(\s_{p-3-a, a+b +2} \ligne \s_{a, b+1})\\
&(\s_{p-1-a, a+b+1}  \ligne \s_{a, b+1})\\
&(\s_{a-2, b+2}  \ligne \s_{p-1-a, a+b+1}) 
\end{align*} 
occur in $L_1/ p L_1$ as subquotients. Consequently, $L_1/pL_1$ has a cosocle filtration
\[
\s_{a-2, b+2} \ligne  ( \s_{p-1-a, a+b+1}  \oplus \s_{p-3-a, a+b+2} ) \ligne \s_{a, b+1}
\]
and $L_1$ is the unique (up to homothety) lattice  in $\unSym^1 E^2 \otimes \Theta(\psi)$ whose reduction has cosocle $\s_{a, b+1}$.
Moreover, we have
\[
(L_1/ p L_1)_{K_1} = ( \s_{p-1-a, a+b+1}  \oplus \s_{p-3-a, a+b+2} ) \ligne \s_{a, b+1}.
\]
\end{proposition}
\begin{proof}
By construction,  we have $pL\subset L_1$  and $L_1 / pL = \Ker (L/pL \to \s_{a-2, b+2})$. By Proposition \ref{prop-reduction-L}(i), the cosocle filtration of $L_1/pL$ is 
\[(\sigma_{p-1-a,a+b+1}\oplus\sigma_{p-3-a,a+b+2}) \ligne\sigma_{a,b+1},\]
thus the nonsplit extensions  $(\s_{p-3-a, a+b +2} \ligne \s_{a, b+1})$ and $(\s_{p-1-a, a+b+1} \ligne \s_{a, b+1})$ occur in $L_1/pL$,  hence also in $L_1/pL_1$.  

We need to show that the nonsplit extension $(\s_{a-2, b+2} \ligne \s_{p-1-a, a+b+1})$ also occurs in $L_1/pL_1$. For this we note that $pL_1 \subset L' \subset L_1$, where $L'$ is  defined in \eqref{eq:def-L'}. Consequently, $L' / pL_1 $ is a subrepresentation of $L_1 / pL_1$, and it is easy to see that \[\JH(L'/pL_1)=\{\s_{a-2,b+2},\s_{p-1-a,a+b+1},\s_{p-3-a,a+b+2}\}.\]  As a quotient of $L' / pL',$ $L' / pL_1 $ admits the nonsplit extension $(\s_{a-2, b+2} \ligne \s_{p-1-a, a+b+1})$ as a subquotient, see  Proposition \ref{prop-reduction-L}(ii). The structure of $(L_1 / pL_1)_{K_1} $ and other statements easily follow. \end{proof}

Let $L_1':=\Ker(L'\onto L'/pL'\onto \sigma_{p-3-a,a+b+2}),$ where $L'$ is defined in \eqref{eq:def-L'}. Then $pL'\subset L_1'\subset L'$. Alternatively, $L_1'$ is characterized by the following exact sequence 
\begin{equation}\label{eq:L1'}0\ra pL\ra L_1'\ra \sigma_{p-1-a,a+b+1}\ra0. \end{equation}
In a similar way to Proposition \ref{prop--sublattice-1}, we have the  following result. 
   
\begin{proposition} \label{prop--sublattice3}
 The following  nonsplit extensions 
 \begin{align*} 
 &(\s_{p-3-a, a+b +2} \ligne \s_{a, b+1}) \\
 &( \s_{a, b+1} \ligne \s_{p-1-a, a+b+1} ) \\
 & (\s_{a-2, b+2} \ligne \s_{p-1-a, a+b+1})
 \end{align*}
  occur in $L'_1/ p L'_1.$ Consequently, $L'_1/pL'_1$ has a cosocle filtration
\[
\s_{p-3-a, a+b+2}  \ligne  ( \s_{a-2, b+2}\oplus  \s_{a, b+1} ) \ligne \s_{p-1-a, a+b+1},
\]
and $L_1'$ is the unique (up to homothety) lattice  in $\unSym^1 E^2 \otimes \Theta(\psi)$ whose reduction has cosocle $ \s_{p-1-a, a+b+1}$.
Moreover, we have
\[
(L'_1/ p L'_1)_{K_1} =  ( \s_{a-2, b+2}\oplus  \s_{a, b+1} ) \ligne\s_{p-1-a, a+b+1}.
\]
\end{proposition}
 \medskip

Let $L_2 : = \Ker( L \onto L/p L \onto (\s_{p-3-a, a+b+2} \ligne \s_{a, b+1}) ).$ Then $pL\subset L_2\subset L$ and there is a short exact sequence 
\begin{equation}\label{eq:L2}0\ra pL\ra L_2\ra (\s_{p-1-a,a+b+1}\ligne\s_{a-2,b+2})\ra0.\end{equation}

\begin{proposition}\label{prop--sublattice2}
Assume $2\leq a \leq p-2.$  Then the following nonsplit extensions 
\begin{align*}
&(\s_{p-3-a, a+b +2} \ligne \s_{a, b+1})\\
&(\s_{a,b+1} \ligne \s_{p-1-a, a+b+1} )\\
&(\s_{p-1-a, a+b+1} \ligne \s_{a-2, b+2})
\end{align*} 
occur in $L_2/ p L_2.$ Consequently, $L_2/pL_2$ has a cosocle filtration
\[
\s_{p-3-a, a+b +2} \ligne \s_{a,b+1} \ligne \s_{p-1-a, a+b+1}  \ligne \s_{a-2, b+2},
\]
and $L_2$ is the unique (up to homothety) lattice in $\unSym^1 E^2 \otimes \Theta(\psi)$ whose reduction has cosocle $\s_{a-2, b+2}$.
Moreover, we have
\[
(L_2 / pL_2)_{K_1} = ( \s_{p-1-a, a+b+1}  \ligne \s_{a-2, b+2}).
\] 
\end{proposition}
\begin{proof}
Since $L_2 / pL_2$ is an extension of $L_2 / pL$ by $pL / pL_2,$ it suffices to show that the nonsplit extension $(\s_{a, b+1} \ligne \s_{p-1-a, a+b+1})$ occurs in $L_2 / pL_2.$ 

It follows from \eqref{eq:L1'} and \eqref{eq:L2} that  $pL\subset L'_1\subset L_2$ and there is a short exact sequence  
\[0\ra L_1'\ra L_2\ra \s_{a-2,b+2}\ra0.\]
This implies that 
\[L_1'/pL_2\cong (L_1'/pL_1')/(\s_{a-2,b+2}),\]
and the nonsplit extension $(\s_{a, b+1} \ligne \s_{p-1-a, a+b+1})$   occurs in $L'_1 / pL_2$ by  Proposition \ref{prop--sublattice3}. Since $L_1'/pL_2$ embeds in   $L_2/pL_2$, this nonsplit extension also occurs in $L_2/pL_2$.
\end{proof}

The sublattices $L_1$ and $L_2$ of $L$ satisfy the following property.

\begin{proposition}\label{prop-control-K1-coinv}
Assume $3\leq a\leq p-2$. Fix $i\in \{1,2\}.$ Then for every $x\in p L,$ there exist $r\in \N,$ $k_1,\ldots, k_r \in K_1,$ $y_1,\ldots, y_r \in L_i,$ such that
\[
x= (k_1 -1) y_1 + \cdots + (k_r - 1) y_r.
\]
\end{proposition}

The proof of Proposition \ref{prop-control-K1-coinv} requires a technique introduced in  \cite[\S 7]{BHHMS1}, so we first recall some notation. Let $\frak{sl}_{2,\F_p}$ be the Lie algebra consisting of trace zero $2\times 2$ matrices with coefficients in $\F_p.$ It is a $3$-dimensional vector space over $\F_p$ with a basis 
\[
e = \left(
                   \begin{array}{cc}
                     0 & 1 \\
                     0 &  0\\
                   \end{array}
                 \right),~ f = \left(
                   \begin{array}{cc}
                     0 & 0 \\
                     1 &  0\\
                   \end{array}
                 \right),~ h = \left(
                   \begin{array}{cc}
                     1 & 0 \\
                     0 &  -1\\
                   \end{array}
                 \right)
                 \]
subject to the Lie bracket relations
\[
[e,f] = h, ~ [h,e] = 2e ,~ [h,f] = -2f.
\]
Let $(V,\rho)$ be a continuous finite dimensional representation of $K/Z_1$ over $E.$ Assume that $V^{\circ}$ is a $K$-stable $\cO$-lattice in $V$ such that $K_1$ acts trivially on $V^{\circ} / p V^{\circ}.$ \cite[\S 7.1]{BHHMS1} defines a Lie algebra action of $\frak{sl}_{2,\F_p}$ on $V^{\circ}/pV^{\circ}$, which induces an $\F$-linear map
\[
\beta_{V^{\circ}} : \frak{sl}_{2,\F_p} \otimes_{\F_p} (V^{\circ}/pV^{\circ} )\to V^{\circ}/pV^{\circ}
\]
sending $x \otimes v$ to $p^{-1} (\rho (\exp (p \wt{x})) \wt{v} - \wt{v}) \pmod{p V^{\circ}},$ where $\wt{x} \in \frak{sl}_{2,\Z_p}$ is a trace zero $2\times 2$ matrix with coefficients in $\Z_p$ lifting $x\in \frak{sl}_{2,\F_p},$ and $\wt{v} \in V^{\circ}$ is a lift of $v \in V^{\circ}/pV^{\circ}.$ The definition does not depend on the choice of the lifts. Moreover, letting $K$ act on $\frak{sl}_{2,\F_p}$ by conjugation $k\cdot x:=\overline{k}x\overline{k}^{-1}$ (for $k\in K$ and $k\mapsto \overline{k}\in \GL_2(\F_p)$), $\beta_{V^{\circ}}$ is $K$-equivariant. Indeed,
\begin{align*}\beta_{V^{\circ}}(k(x\otimes v))=\beta_{V^{\circ}}(\overline{k}x\overline{k}^{-1}\otimes kv)&=p^{-1}\big(\rho(\exp(pk\wt{x}k^{-1}))k\wt{v}-k\wt{v}\big) \pmod{p V^{\circ}}\\
&=p^{-1}\big(\rho(k\exp(p\wt{x}))\wt{v}-k\wt{v}\big)\pmod{p V^{\circ}}\\
&=k\beta_{V^{\circ}}(x\otimes v). 
\end{align*}

In the special case $V=\unSym^1E^2$  and $V^{\circ}=\unSym^1\cO^2 $, we easily check that  
the map $\beta_{\unSym^1\cO^2}:\frak{sl}_{2,\F_p}\otimes\Sym^1\F^2\ra \Sym^2\F^2$ is given by
\begin{equation}\label{eq::beta_Sym}\beta_{\unSym^1\cO^2}(\alpha\otimes \ell(X,Y))=\ell(a_{11}X+a_{21}Y,a_{12}X-a_{11}Y)\end{equation}
for $\a = \bigl(\begin{smallmatrix}
a_{11} & a_{12} \\ a_{21}& -a_{11}
\end{smallmatrix} \bigr) \in \frak{sl}_{2,\F_p}$ and $\ell(X,Y) \in \Sym^1 \F^2$.  

\begin{proof}[Proof of Proposition \ref{prop-control-K1-coinv}]
We only give the proof  for $L_2$ (which will be used in the proof of Proposition \ref{prop:R-cosoc}). 
Take $V=\unSym^1E^2\otimes\Theta(\psi)$ and $V^{\circ}=L=\unSym^1\cO^2\otimes T$ in the above discussion. Since $K_1$ acts trivially on $T$,  the map $\beta_{L}:\frak{sl}_{2,\F_p}\otimes \Sym^1\F^2\otimes T/pT\ra \Sym^1\F^2\otimes T/pT$ is given by (cf. \cite[Rem.~7.1.3]{BHHMS1})
\begin{equation}\label{eq::beta_L}
\beta_{L}=\beta_{\unSym^1\cO^2}\otimes \mathrm{Id}_{T/pT}.\end{equation}
  
Let $W_1=(\s_{p - 1-a, a+b+1} \ligne \s_{a-2, b+2})$ be the subrepresentation of $L/ pL$ defined in the proof of Proposition \ref{prop-reduction-L}. Then $L_2$ is exactly the preimage of $W_1$ in $L$ under the surjection $L \onto L/pL.$ Taking $W= W_1$ (and  $V^\circ = L$) in \cite[Lem.~7.1.4]{BHHMS1}, we obtain a commutative diagram in which both rows are exact:
$$\xymatrix{\frak{sl}_{2, \F_p} \otimes W_1 \ar@{.>}[dr]^{\varphi}  \ar[r]^{\ \ \beta_{L}|_{\frak{sl}_{2,\F_p} \otimes W_1}} \ar@{^(->}[d] & L/pL \ar@{=}[d] \ar[r]^{p\ \ \ }& (L_2 / p^2 L)_{K_1} \ar@{^(->}[d] \ar[r] &  W_1  \ar@{^(->}[d] \ar[r] & 0 \\
\frak{sl}_{2,\F_p} \otimes L /pL  \ar[r]^{\ \ \ \beta_{L}}  & L/pL  \ar[r]^{p\ \ \ }& (L / p^2 L)_{K_1} \ar[r] & L / pL  \ar[r] &0. 
}
$$
To prove the proposition, it suffices to check that the dotted map $\varphi,$ which is the composite
\begin{equation}\label{eqn--surj}
\varphi: \frak{sl}_{2, \F_p} \otimes W_1 \into \frak{sl}_{2,\F_p} \otimes L / pL \To{\beta_{L}} L/ pL,
\end{equation}
is surjective, since this implies that the images of $pL$ and $p^2L$ in $(L_2)_{K_1}$ coincide. Recall that $T/pT$ fits into a short exact sequence 
\[
0\ra \sigma_{p-2-a,a+b+1}\ra T/pT\ra \sigma_{a-1,b+1}\ra0
\]
and $W_1\cap (\Sym^1\F^2\otimes \sigma_{p-2-a,a+b+1})=\sigma_{p-1-a,a+b+1}$ (see the proof of Proposition \ref{prop-reduction-L}). Using \eqref{eq::beta_L}, we see that $\varphi$ sends $\frak{sl}_{2,\F_p}\otimes \sigma_{p-1-a,a+b+1}$ to $\Sym^1\F^2\otimes \sigma_{p-2-a,a+b+1}$, so that $\varphi$ induces a $K$-equivariant map
\[\overline{\varphi}: \frak{sl}_{2,\F_p}\otimes \sigma_{a-2,b+2}\hookrightarrow \frak{sl}_{2,\F_p}\otimes \Sym^1\F^2\otimes\sigma_{a-1,b+1}\ra \Sym^1\F^2\otimes\sigma_{a-1,b+1}\] 
where the second map is given by $\beta_{\unSym^1\cO^2}\otimes\mathrm{Id}_{\sigma_{a-1,b+1}} $. By Proposition \ref{prop-reduction-L}(i), $\mathrm{cosoc}(L/pL)=\Sym^1\F^2\otimes\sigma_{a-1,b+1}$. Hence, to prove $\varphi$ is surjective it suffices to prove $\overline{\varphi}$ is surjective.
 
We prove that $\overline{\varphi}$ is surjective by a direct computation (analogous to \cite[Lem.~7.2.1]{BHHMS1}). Fix a nonzero element $v\in (\sigma_{a-1,b+1})^{U(\Z_p)}$. Then $ (\sigma_{a-1,b+1})^{U(\Z_p)} = \F v$ and the group $H = \big\{\smatr{[a]}00{[d]},~a,d\in \F_p^{\times} \big\}$ acts on $v$ by $\chi_{a-1,b+1}.$ Recall that there is a natural action of  $\frak{sl}_{2,\F_p}$ on $\sigma_{a-1,b+1}$ and the set $\{v,f(v),\ldots, f^{a-1}(v)\}$ forms an  $\F$-basis of $\s_{a-1,b+1},$  see \cite[\S 4.2.1]{Paskunas-MSMF}.  By construction we have $e(v)=0$ and $h(v)=(a-1)v$. 
Consider the following nonzero elements of $\Sym^1\F^2\otimes\sigma_{a-1,b+1}$: \[l_1 : = X \otimes v,\ \ \ l_2: = Y\otimes v - \frac{1}{a-1} X \otimes f(v).\]  It is easy to see that $H$ acts on $l_1$ (resp. $l_2$) by $\chi_{ a, b +1}$ (resp. $\chi_{ a - 2, b +2}$).  Moreover,  using the fact $ef(v)=fe(v)+h(v)=(a-1)v$, one  checks that $e(l_1)=e(l_2)=0$, so that $l_1$ and $l_2$ are fixed by $U(\Z_p)$. Since $\Sym^1\F^2\otimes \sigma_{a-1,b+1}\cong \sigma_{a,b+1}\oplus\sigma_{a-2,b+2}$, we deduce that $l_1\in(\sigma_{a,b+1})^{U(\Z_p)}$ and $l_2\in (\sigma_{a-2,b+2})^{U(\Z_p)}$  under this decomposition. 
In particular, $\Sym^1\F^2\otimes\sigma_{a-1,b+1}$ is generated by $l_1$ and $l_2$ as a $K$-representation. 
 
Since $\overline{\varphi}$ is $K$-equivariant, to finish the proof it suffices to prove that both $l_1$ and $l_2$ lie in the image of $\overline{\varphi}.$   
We let  (recall $a\geq 3$)
\[
w_1 := e \otimes l_2,\quad w_2 : = -h\otimes l_2 - \frac{2}{a-2} e\otimes f(l_2)
\]
be elements of $ \frak{sl}_{2,\F_p} \otimes \sigma_{a-2,b+2}$ and claim that $\overline{\varphi}(w_i) = l_i,$ $i=1,2.$ Indeed, as $\beta_{\unSym^1 \cO^2}(e\otimes  X)=0$ and $ \beta_{\unSym^1 \cO^2}(e\otimes  Y)=X$ by \eqref{eq::beta_Sym}, 
\begin{align*}
\overline{\varphi} (w_1) & = \beta_L(e \otimes l_2) = \beta_L\big(e\otimes  (Y\otimes v - \frac{1}{a-1} X \otimes f(v))\big) \\
& = \beta_{\unSym^1 \cO^2} (e\otimes Y) \otimes v - \frac{1}{a-1} \beta_{\unSym^1 \cO^2} (e\otimes X) \otimes f(v) = X\otimes v = l_1.
\end{align*}
Similarly, $
\overline{\varphi}(w_2) = \beta_L\big( -h\otimes l_2 - \frac{2}{a-2} e\otimes f(l_2)\big).$ We have
\begin{align*}
f (l_2) &= f \big(Y\otimes v - \frac{1}{a-1} X \otimes f(v)\big)\\ &= Y \otimes f(v)  - \frac{1}{a-1} (f(X) \otimes f(v) + X \otimes f^2(v))\\
&= \frac{a-2}{a-1}Y\otimes f(v) - \frac{1}{a-1}X \otimes f^2 (v),
\end{align*}
hence
\begin{align*}
\beta_L( e\otimes f (l_2)) & = \frac{a-2}{a-1}\beta_{\unSym^1 \cO^2} (e\otimes Y) \otimes f(v) -  \frac{1}{a-1}  \beta_{\unSym^1 \cO^2}(e\otimes  X) \otimes f^2 (v) \\
& =  \frac{a-2}{a-1} X\otimes f(v) .
\end{align*}
As $h(X)=X$ and $h(Y)=-Y$ by \eqref{eq::beta_Sym}, we obtain
\begin{align*}
\overline{\varphi} (w_2) & =- \beta_{L} \big(h \otimes (Y\otimes v - \frac{1}{a-1}X \otimes f(v))\big) - \frac{2}{a-2}\cdot \frac{a-2}{a-1} X\otimes f(v)  \\
& = Y \otimes v + \frac{1}{a-1}X \otimes f(v) - \frac{2}{a-1} X\otimes f(v) = l_2. 
\end{align*}
This proves the claim and finishes the proof of the proposition.
\end{proof}

\subsection{Gluing Lattices}\label{subsection:gluing}

\hfill

Assume $1 \leq a \leq p-3.$ Consider the following three characters of $\F_{p^2}^{\times}:$
\begin{equation}\label{equation-choice-psi}
\psi_1 = [\xi]^{a+2 + (p+1)b},~ \psi_{2} = [\xi]^{a+3 + (p+1)(b-1)},~ \psi_{3} = [\xi]^{a+1 + (p+1)b}.
\end{equation}
In this subsection, we construct a lattice $\widetilde{R}$ in \[\Theta(\psi_1) \oplus \left(\unSym^1 E^2 \otimes\Theta(\psi_2)\right) \oplus \left(\unSym^1 E^2 \otimes\Theta(\psi_3)\right)\] such that $\widetilde{R} / p\widetilde{R}$ is killed by $\frak{m}_{K_1}^2$ and $\rcosoc (\widetilde{R} / p\widetilde{R}) = \s_{a,b+1}.$ We divide the construction into two cases: $1\leq a\leq p-4$ and $a=p-3.$

\subsubsection{The case  $1\leq a\leq p-4$}\label{subsubsection:p-4}
Denote by $W$ the nonsplit $\G$-extension $( \s_{p- 3 - a, a + b + 2} \ligne \s_{a, b+1}).$ 
\begin{enumerate}
\item[(1)] Let $R_1$ be the unique (up to homothety) lattice in $\Theta(\psi_1)$ such that $\rcosoc(R_1/ p R_1) = \s_{a,b+1}.$ Then $R_1/ p R_1\cong W.$ Let $r_1$ denote the composite
\[
r_1: R_1 \onto R_1 / pR_1 \cong W.
\]
\item[(2)] By Proposition \ref{prop:Diamond} and \cite[Lem.~3.8]{BP}, we have 
\[\JH\left(\overline{\unSym^1E^2\otimes\Theta(\psi_2)}^{\rm ss}\right)=\big\{\s_{a,b+1},\s_{p-3-a,a+b+2},\s_{a+2,b},\s_{p-5-a,a+b+3}\big\}\]
with the convention $\s_{-1,b}=0$. Let $R_2\subset \unSym^1E^2\otimes\Theta(\psi_2)$ be the unique (up to homothety) lattice such that $\rcosoc(R_2/ p R_2) = \s_{a,b+1}.$  The structure of $R_2 / p R_2$ is given by Proposition \ref{prop--sublattice2}, i.e. 
\begin{equation}\label{eq:R2/p}R_2/pR_2\cong(\s_{p-5-a,a+b+3}\ligne \s_{a+2,b}\ligne\s_{p-3-a,a+b+2}\ligne\s_{a,b+1}).\end{equation}
 \item[(3)]  By Proposition \ref{prop:Diamond} and \cite[Lem.~3.8]{BP}, we have \[\JH\left(\overline{\unSym^1E^2\otimes\Theta(\psi_3)}^{\rm ss}\right)=\big\{\sigma_{a,b+1},\s_{p-3-a,a+b+1},\s_{a-2,b+2},\s_{p-1-a,a+b+1}\big\}\]
 with the convention $\s_{-1,b+2}=0$. Let $R_3\subset \unSym^1E^2\otimes\Theta(\psi_3)$ be the unique (up to homothety) lattice such that $\rcosoc(R_3)\cong\s_{a,b+1}$. By Proposition \ref{prop--sublattice-1}, $R_3/pR_3$ has a cosocle filtration
\begin{equation}\label{eq:R3/p}
\s_{a-2, b+2} \ligne  ( \s_{p-1-a, a+b+1}  \oplus \s_{p-3-a, a+b+2} ) \ligne \s_{a, b+1}.\end{equation}
\end{enumerate}

Note that there exists a surjection  $R_2\twoheadrightarrow W$ which we denote by $r_2$; let $R_2':=\Ker(r_2)$. The structure of $R_2'/pR_2'$ is determined in Proposition \ref{prop-reduction-L}(i). Precisely, it has a two-step socle and cosocle filtration 
\begin{equation}\label{eq:R2'/p}(\s_{p-3-a,a+b+2}\oplus \s_{p-5-a,a+b+3})\ligne (\s_{a+2,b}\oplus\s_{a,b+1})\end{equation}
 and all possible extensions do occur. 

Similarly, there exists a surjection $R_3\twoheadrightarrow W$ which we denote by $r_3$;  let $R_3':=\Ker(r_3)$.  The structure of $R_3'/pR_3'$  is also determined in Proposition \ref{prop-reduction-L}(i). Precisely, it 
has a cosocle filtration
\begin{equation}\label{eq:R3'/p}
\s_{p-3-a, a+b+2} \ligne  ( \s_{a, b+1}  \oplus \s_{a-2, b+2} ) \ligne \s_{p-1-a, a+b+1}\end{equation}
 and all possible extensions do occur.

\subsubsection{Glue $R_1$ and $R_2$ ($1\leq a\leq p-4$)}

Let $R$ be the lattice in $\Theta(\psi_1) \oplus (\unSym^1 E^2 \otimes\Theta(\psi_2))$ obtained by gluing $R_1$ and $R_2$ along $W$, i.e. $R$ is given by the short exact sequence 
 \begin{equation} \label{equation-ses3}
0 \to R \to R_1 \oplus R_2 \To{r_1 - r_2} W \to 0.
\end{equation}

Let $r_{R}$ denote the composition $R \onto R/pR \onto R_1/pR_1 \cong W.$

\begin{lemma}\label{lem--first-glue}
(i) $r_R$ induces a short exact sequence
\begin{equation} \label{equation-ses1}
0 \to R_2' / p R_2' \ra R /p R \to W \to 0 .
\end{equation}
In particular, $R/ pR$ is killed by $\frak{m}_{K_1}^2.$

(ii) $\Ker(r_R)=R_2'+pR$ and
\[\Ker(r_R)/p\Ker(r_R)\cong  R_2'/pR_2'\oplus W.\]
\end{lemma}
\begin{proof}
This is  a special case of Lemma \ref{lem:glue-L/p} applied to $L_1=R_2$ and $L_2=R_1$.
\end{proof}

\begin{proposition}\label{prop:R-cosoc}
The short exact sequence (\ref{equation-ses1}) induces an isomorphism $(R/pR)_{K_1} \cong W.$ In particular $\rcosoc (R/ p R) = \s_{a,b+1}.$
\end{proposition}
\begin{proof}
By Lemma \ref{lem--first-glue}(i), it suffices to show that for any $x \in R_2',$ there exist $r\in \N,$ $k_1,\ldots, k_r \in K_1,$ $v_1,\ldots, v_r \in R$ such that $ x= (k_1 -1) v_1 + \cdots + (k_r - 1) v_r.$ By Proposition \ref{prop-control-K1-coinv} (applied to $L=p^{-1}R_2'$ and $L_2=R_2$), there exist $r\in \N,$ $k_1,\ldots, k_r \in K_1,$ $y_1,\ldots, y_r \in R_2$  such that
 \[
  x =  (k_1 -1) y_1 + \cdots + (k_r - 1) y_r. 
 \]
For  $1\leq i \leq r$, choose $z_i \in R_1$ such that $r_1(z_i) = r_2 (y_i)$ and let $v_i = (z_i ,y_i)\in R.$ Since $K_1$ acts trivially on $R_1,$ we have
 \[
 (0, x) = \big(\sum_{i=1}^r (k_i -1) z_i, \sum_{i=1}^r (k_i -1) y_i \big) = \sum_{i=1}^r (k_i - 1)  v_i,
 \]
giving the result.
\end{proof}

\subsubsection{Glue $R$ and $R'_3$ ($1\leq a\leq p-4$)}

We define $\widetilde{R}$ to be the lattice in $\Theta(\psi_1)\oplus \big(\unSym^1E^2\otimes\Theta(\psi_2)\big)\oplus \big(\unSym^1E^2\otimes \Theta(\psi_3)\big)$ obtained by gluing  $R$ and $R_3$ along $W$, i.e. 
\begin{equation}\label{eq:tildeR}0\ra \wt{R}\ra R\oplus R_3\To{r_R-r_3} W\ra0.\end{equation}

\begin{proposition}\label{prop:tildeR}
(i) There exists a short exact sequence
\[ 
0 \to \Ker(r_R) / p \Ker(r_R) \ra \wt{R} /p \wt{R} \to R_3/pR_3 \to 0 .
\]

(ii) $\rcosoc (\wt{R}/ p \wt{R}) = \s_{a,b+1}.$
\end{proposition}

\begin{proof}
(i) This is  a special case of Lemma \ref{lem:glue-L/p}.

(ii) This is a special case of Lemma \ref{lem:cosoc-glue}, with $L_1=R$ and $L_2=R_3$.  Firstly, Condition (a) in \emph{loc. cit.} holds by Proposition \ref{prop:R-cosoc}. Secondly, we have \[\rcosoc(\Ker(r_R))=\rcosoc(W)\oplus \rcosoc(R'_2)=\sigma_{a,b+1}\oplus \sigma_{a,b+1}\oplus \sigma_{a+2,b}\]
by \eqref{eq:R2'/p} and Lemma \ref{lem--first-glue}(ii), and  \[\rcosoc(\Ker(r_3))\cong\sigma_{p-1-a,a+b+1}\] 
by \eqref{eq:R3'/p}, hence Condition (b) in Lemma \ref{lem:cosoc-glue} also holds. \end{proof}

\begin{proposition}\label{prop-lattice-red-wtR}
Let $V$ denote the quotient of $R_3/pR_3$ by $\sigma_{a-2,b+2}$ via \eqref{eq:R3/p}. 
Then there exists a short exact sequence
\begin{equation}\label{equation-ses2}
0 \to   R_2'/ p R_2'\oplus W \oplus \s_{a-2, b+2} \to \wt{R} / p\wt{R} \to  V \to 0.
\end{equation}
In particular, $\wt{R} / p\wt{R}$ is killed by $\frak{m}_{K_1}^2.$
\end{proposition}
\begin{proof}
By definition we have   
\[0\ra \sigma_{a-2,b+2}\ra R_3/pR_3\ra V\ra0.\]
 Note that $\s_{a-2, b+2}$ has no nontrivial extensions with any Jordan-H\"older factor of $W$ and of $R_2 / pR_2$, using \cite[Cor.~5.6]{BP} and  Lemma \ref{lem:BP5.6} below. The result easily follows by Proposition \ref{prop:tildeR}.
\end{proof}

\begin{lemma}\label{lem:BP5.6}
Assume $2\leq a\leq p-2$. Then $\Ext^1_{K}(\sigma_{a-2,b+2}, \sigma_{a,b+1})=0.$ 
\end{lemma}
\begin{proof}
We have a short exact sequence $0\ra \sigma_{p+1-a,a+b}\ra \Ind_I^{K}\chi_{a-2,b+2}\ra \sigma_{a-2,b+2}\ra0$. Since 
$\Ext^1_{K}(\sigma_{p+1-a,a+b},\sigma_{a,b+1})=0$ by \cite[Cor.~5.6]{BP}, we are reduced to proving $$\Ext^1_K(\Ind_I^K\chi_{a-2,b+2},\sigma_{a,b+1})=0,$$ equivalently $\Ext^1_{I}(\chi_{a-2,b+2},\sigma_{a,b+1})=0$ by Frobenius reciprocity.

Consider an $I$-extension $0\ra \sigma_{a,b+1}|_I\ra \mathcal{E}\ra \chi_{a-2,b+2}\ra0$.  We first prove that it splits as $U(\Z_p)$-representation. Since $\sigma_{a,b+1}$ is a cyclic $\F[\![U(\Z_p)]\!]$-module, we have $H^1(U(\Z_p),\sigma_{a,b+1})\cong H^1(U(\Z_p),\chi_{a,b+1}^s)$, where $\chi_{a,b+1}^s$ is identified with the $U(\Z_p)$-cosocle of $\sigma_{a,b+1}$. As seen in the proof of Proposition \ref{prop:U-inv}, we get \[H^1(U(\Z_p),\sigma_{a,b+1})\cong\chi_{a,b+1}^s\alpha^{-1}.\]
As $2\leq a\leq p-2$, this implies $\chi_{a-2,b+2}\neq \chi_{a,b+1}^s\alpha^{-1}$, and so $\Ext^1_{U(\Z_p)}(\chi_{a-2,b+2},\sigma_{a,b+1})=0$. 

As a consequence, we may choose $v\in \mathcal{E}$ which is fixed by $U(\Z_p)$ and on which $H$ acts via $\chi_{a-2,b+2}$. Next, as in the proof of \cite[Prop.~7.2]{Pa10}, we show that $v$ is actually fixed by $I_1$, showing that $\mathcal{E}$ splits as $I$-representation. This finishes the proof.
\end{proof}

We  obtain the following corollary.

\begin{corollary}\label{cor-wtL-structure} 
$\wt{R} / p \wt{R}$ has cosocle $\s_{a,b+1}$ and is a quotient of $(\Proj_{ \cO[\![K/Z_1]\!]}\s_{a,b+1}) / \frak{m}_{K_1}^2.$
\end{corollary}

\subsubsection{The case $a = p-3$}
In this case we need to slightly modify the above construction. We only sketch the construction and leave the detail to the reader.
\begin{enumerate}
\item[(1)] Let $R_1$ be the unique (up to homothety) lattice in $\Theta(\psi_1)$ such that $\rcosoc(R_1/ p R_1) = \s_{p-3,b+1}.$ Then \[R_1/ p R_1\cong (\s_{0,b} \ligne \s_{p-3,b+1}) = : W.\]
Let $r_1$  denote the projection 
$ R_1  \onto \s_{p-3,b+1}.$
 
\item[(2)] By Proposition \ref{prop:Diamond} and \cite[Lem.~3.8]{BP}, we have 
\[\JH\left(\overline{\unSym^1 \cO^2 \otimes \Theta(\psi_2)}^{\rm ss}\right)=\{\sigma_{p-1,b},\sigma_{p-3,b+1}\}.\]
Let $R_2$ be  the unique lattice in $\unSym^1 E^2 \otimes \Theta(\psi_2)$ such that $\rcosoc (R_2/pR_2) = \s_{p-3, b+1}.$ Then \[R_2/pR_2\cong(\s_{p-1,b}\ligne\s_{p-3,b+1}).\]
Let $r_2$ denote the projection $R_2\onto \s_{p-3,b+1}$ and $R_2':=\Ker(r_2)$. Proposition \ref{prop-reduction-L-2} implies that 
\[R_2'/pR_2'\cong \s_{p-3,b+1}\oplus\s_{p-1,b}.\] 

\item[(3)] Let $R_3$ and $R'_3$ be the lattices in $\unSym^1 E^2 \otimes \Theta(\psi_3)$ constructed as in the case $1\leq a\leq p-4.$ Namely, $R_3$ has cosocle $\s_{p-3,b+1}$, and $R_3':=\Ker(r_3)$ where $r_3$ denotes the projection $R_3\onto W$. 
\end{enumerate}

\medskip

We first glue $R_1 $ and $R_2$ along $\s_{p-3,b+1}$, namely 
\[
0\ra R\ra  R_1 \oplus R_2 \To{r_1-r_2} \s_{p-3,b+1}\ra0. 
\]
Then by Lemma \ref{lem:glue-L/p}(i) there is a short exact sequence
\[0\ra R_2'/pR_2'\ra R/pR\To{r_R} W\ra0.\]
Moreover, as in the proof of Proposition \ref{prop:R-cosoc} one can show that $r_R$ induces an isomorphism  $(R/pR)_{K_1} \cong  W.$ In particular, $\rcosoc(R/pR)\cong \s_{p-3,b+1}$.

The gluing of $R$ and $R_3$ is exactly as in the case $1\leq a \leq p-4.$ Let $\wt{R}$ be defined by
\[0\ra \wt{R}\ra R\oplus R_3\To{r_R-r_3} W\ra0.\]
One can follow the proof of Proposition \ref{prop-lattice-red-wtR} and Corollary \ref{cor-wtL-structure}, to show the following result.
\begin{proposition}\label{prop:tildeR-p-3}
(i) $\wt{R} / p \wt{R}$ has cosocle $\s_{p-3,b+1}$.

(ii) Let $V$ denote the quotient of $R_3/pR_3$ by $\sigma_{p-5,b+2}$. Then there is a short exact sequence 
\begin{equation}\label{equation-ses2'}
0 \to W \oplus R'_2/ p R'_2 \oplus \s_{p-5, b+2} \to \wt{R} / p\wt{R} \to  V \to 0.
\end{equation}
As a consequence,  $\wt{R} / p \wt{R}$ is a quotient of $(\Proj_{ \cO[\![K/Z_1]\!]}\s_{p-3,b+1}) / \frak{m}_{K_1}^2.$
\end{proposition}

\section{Galois deformation rings}\label{sec:deform}

Assume  $p\geq 5.$ Let $\brho: G_{\Q_p}\to \GL_2(\F)$ be a two dimensional continuous representation of $G_{\Q_p} = \Gal(\overline{\Q}_p / \Q_p).$
In this section, we study the congruence relation of Galois deformation rings of different (tame) types. Our method does not allow to determine the precise structure of the Galois deformation rings, but is enough for application in \S\ref{section:GK-dim}.

\subsection{Preliminaries}

We collect some results on the set of Serre weights associated to $\brho$ and some results of  Pa\v{s}k\=unas and of Morra. We prove in \S\ref{subsubsection:regular} a criterion for some Galois deformation rings to be regular.

\subsubsection{Serre weights } \label{section-GL2-Serrewt}

Let $\omega$ (resp. $\omega_2$) be the mod $p$ cyclotomic character (resp. Serre's fundamental character of niveau 2) of $G_{\Q_p}.$ Up to isomorphism, $\brho$ has one of the following forms:

\begin{enumerate}[label=(Case \arabic*),ref=(Case \arabic*)]
\item\label{Case1} $\brho$ is absolutely irreducible and $\brho|_{I_p} \sim  \bigl(\begin{smallmatrix}
\omega_{2}^{r+1} & 0 \\ 0& \omega_2^{p(r+1)}
\end{smallmatrix} \bigr) \otimes  \omega^{s+1},$ $0\leq r\leq p-1,$ $0\leq s \leq p-2.$

\item\label{Case2} $\brho \sim \bigl(\begin{smallmatrix}
{\rm unr}_1 \omega^{r+1} & * \\ 0& {\rm unr}_2
\end{smallmatrix} \bigr) \otimes \omega^{s+1}$ is reducible nonsplit, where ${\rm unr}_1,$ ${\rm unr}_2$ are unramified characters, and $0\leq r \leq p-2,$ $0\leq s \leq p-2.$

\item\label{Case3}  $\brho \sim \bigl(\begin{smallmatrix}
{\rm unr}_1 \omega^{r+1} & 0 \\ 0& {\rm unr}_2
\end{smallmatrix} \bigr) \otimes \omega^{s+1}$ is reducible split, where ${\rm unr}_1,$ ${\rm unr}_2$ are unramified characters, and $0 \leq r\leq p-2,$ $0\leq s \leq p-2.$
\end{enumerate}

Let $W(\brho)$ be the set of Serre weights associated to $\brho$ in \cite{BDJ}. We have the following explicit description of $W(\brho)$. 

\begin{theorem} $\mathrm{(}$\cite[Thm.~3.17]{BDJ}$\mathrm{)}$\label{thm:Serrewt-BDJ}
\begin{enumerate}
\item[(i)] Assume $\brho$ is in \ref{Case1}. Then  $W(\brho)= \{ \s_{r,s+1},~ \s_{p-1-r, r+s+1} \} .$
\medskip

\item[(ii)] Assume $\brho$ is in \ref{Case2}.
\begin{enumerate}
\item[(ii-a)] If $r \neq 0,$ then  $W(\brho)= \{\s_{r, s+1}\}.$

\item[(ii-b)] If $r=0,$ ${\rm unr}_1 = {\rm unr}_2$ and $\brho $ is tr\`es ramifi\'e, then  $W(\brho) =\{ \s_{p-1,s+1}\}.$

\item[(ii-c)] For other $\brho ,$ $W(\brho)= \{ \s_{0,s+1},~ \s_{p-1,s+1}  \}.$
\end{enumerate}\medskip
\item[(iii)] Assume $\brho$ is in \ref{Case3}.
\begin{enumerate}
\item[(iii-a)] If $1\leq r\leq p-4,$ then $W(\brho)=\{ \s_{r,s+1},~ \s_{p-3-r, r+s+2}\}.$

\item[(iii-b)] If $r = 0,$ then  $W(\brho) =\{ \s_{0,s+1},~ \s_{p-3, s+2},~ \s_{p-1, s+1} \}.$

\item[(iii-c)] If $r = p-3,$ then $W(\brho)=\{ \s_{0, s},~ \s_{p-3,s+1} , ~\s_{p-1, s} \}.$

\item[(iii-d)] If $r = p-2$, then  $W(\brho)=\{ \s_{p-2, s+1}\}.$
\end{enumerate}
\end{enumerate}
\end{theorem}

\subsection{Mod $p$ representations of $\GL_2(\Q_p)$}\label{ss:LLC}  Assume that $\brho$ satisfies $\End_{G_{\Q_p}}(\brho)\cong \F$. We associate to $\brho$ an admissible smooth $\F$-representation  $\pi(\brho)$ of $G:=\GL_2(\Q_p)$ as follows. 

\begin{enumerate}
\item[\ref{Case1}] If $\brho$ is absolutely irreducible, then $\pi(\brho)$ is the irreducible supersingular representation of $G$ associated to $\brho$ by the mod $p$ local Langlands correspondence defined in \cite{Br03}.

 \medskip

\item[\ref{Case2}]  If $\brho  \sim \bigl(\begin{smallmatrix}
\chi_1 & * \\ 0& \chi_2
\end{smallmatrix} \bigr)$ with $\chi_1 \chi_2^{-1} \neq \o^{\pm 1}, \ide,$ then there is an exact nonsplit sequence
\[
0 \to \Ind_{B(\Q_p)}^{G} \chi_2\otimes \chi_1 \o^{-1} \to \pi(\brho) \to \Ind_{B(\Q_p)}^{G} \chi_1\otimes \chi_2 \o^{-1} \to 0.
\]
 If $\brho  \sim \bigl(\begin{smallmatrix}
\chi & * \\ 0& \chi\o
\end{smallmatrix} \bigr),$ then there is an exact nonsplit sequence
\[
0\to \Ind_{B(\Q_p)}^{G} \chi\omega \otimes \chi\o^{-1} \to \pi(\brho) \to \tau_1 \otimes \chi\circ \det \to 0,
\]
where ${\rm Sp}$ is the Steinberg representation of $G$ and $\tau_1$ is a nonsplit extension $0 \to {\rm Sp} \to \tau_1 \to \ide_G^{\oplus 2} \to 0$ with $\soc_G(\tau_1)=\Sp$.\\
If $\brho\sim\smatr{\chi\omega}*0{\chi}$, then $\pi(\brho)$ is the representation  defined in \cite[Lem.~6.7]{Paskunas-BM} (denoted by $\beta$ there). Its precise structure will be recalled in \S\ref{ss:S-nongeneric}. 
\end{enumerate}

Remark that the representation $\pi(\brho)$ is just the representation corresponding to $\brho$ in the  mod $p$ local Langlands correspondence for $\GL_2(\Q_p)$, except in the case  $\brho\sim \smatr{\chi}*0{\chi\omega}$, $\pi(\brho)$ has one extra copy of $\chi\circ \det$ than the usual form.

The following theorem is a consequence of results of Morra (\cite{Morra-ss} \cite{Morra-atom}).

\begin{theorem}\label{thm-morra}
Assume $\brho$ is either in \ref{Case1} of \S\ref{section-GL2-Serrewt} with $r\notin \{1,p-2\}$ or  $\brho$ is in \ref{Case2} of \S\ref{section-GL2-Serrewt} with $1 \leq r\leq p-3$.\footnote{The case where $r=0$ may also be considered, see the footnote of \cite[Thm.~1.1]{Morra-atom}. But as our method requires  to exclude this case in \S \ref{section-def-ring}, we choose to ignore it here.} Then for any $\s \in W(\brho),$ $\s$ occurs in $\pi(\brho) [\frak{m}_{K_1}^2]$ with multiplicity one.
\end{theorem}
\begin{proof}
If $\brho$ is absolutely irreducible, then $\pi(\brho)$ is the representation $\pi(\brho)$ in \cite{Morra-ss} whose $K$-socle filtration is given by \cite[Thm.~1.1]{Morra-ss}. If $\brho$ is reducible nonsplit and $\brho \nsim  \bigl(\begin{smallmatrix}
\ide  & * \\ 0& \omega
\end{smallmatrix} \bigr) \otimes \chi,$ then $\pi(\brho) $ equals to the representation $A_{r,\l}$ (in \cite[Thm.~1.1]{Morra-atom}) for some $\l\in\F^{\times}.$ If $\brho \sim  \bigl(\begin{smallmatrix}
\ide  & * \\ 0& \omega
\end{smallmatrix} \bigr) \otimes \chi,$ then $\pi(\brho)$ has an extra copy of $\chi\circ\det$ than  the representation $A_{r,\l}$. However, in this case  $(\chi\circ\det)|_K$ is not a Serre weight of $\brho$. Thus, for any $\sigma\in W(\brho)$ the multiplicity of $\s$ in $\pi(\brho) [\frak{m}_{K_1}^2]$ is equal to the multiplicity of $\s$ in $A_{r,\l} [\frak{m}_{K_1}^2].$
The $K$-socle filtration of $A_{r,\l}$ is given by \cite[Thm.~1.1]{Morra-atom} and \cite[Thm.~1.2]{Morra-ss}, from which the result follows.
\end{proof}

\subsubsection{Results of Pa\v{s}k\=unas}\label{section-Morra}

 Recall that $\brho$ is called {\em generic} in the sense of \cite{Paskunas-BM} if either $\brho$ is absolutely irreducible or $\brho  \sim \bigl(\begin{smallmatrix}
\chi_1 & * \\ 0& \chi_2
\end{smallmatrix} \bigr)$ is reducible nonsplit with $\chi_1 \chi_2^{-1} \neq \o, \ide.$ We assume $\brho$ is generic, so  in particular $\End_{G_{\Q_p}} (\brho) = \F.$ Let $\eta: G_{\Q_p} \to \cO^{\times}$ be a character such that $\eta \pmod{\varpi} = \det \brho.$ Let $R_{\brho}^{\eta}$ denote the universal deformation ring of $\brho$ with determinant $\eta$ and let $\rho^{\rm univ}$ denote the universal object over $R_{\brho}^{\eta}.$ 

Let $\psi = \eta \e^{-1}.$ According to \cite[\S6.1]{Paskunas-BM}, there exists  $N \in \frak{C}_{G,\psi}(\cO)$ with a faithful continuous action of $R^{\eta}_{\brho}$ which commutes with the action of $G$ such that:

(N0)  $ \F\wh{\otimes}_{R_{\brho}^{\eta}} N$ is of finite length in $\frak{C}_{G,\psi}(\cO)$ and is finitely generated over $\cO[\![ K ]\!];$

(N1)  $\Hom_{\SL_2(\Q_p)}(\ide_G,N^{\vee})=0$;

(N2) $\End_{\frak{C}_{G,\psi}(\cO)} (N) \cong R_{\brho}^{\eta}$ and $\check{\mathbf{V}}(N)$ is isomorphic to $\rho^{\rm univ}$ as $R_{\brho}^{\eta}[\![ G_{\Q_p} ]\!]$-modules, where $\check{\mathbf{V}}$ is the modified Colmez functor in \cite[\S3]{Paskunas-BM};

(N3) $N$ is projective in $\frak{C}_{G,\psi}(\cO)$, and there exists $x\in R^{\eta}_{\brho}$ such that $N/xN$ is isomorphic to a projective envelope of $\oplus_{\sigma\in W(\brho)}\sigma^{\vee}$ in $\Mod^{\rm pro}_{K ,\psi}(\cO)$.

\begin{remark}
Under our assumption on $\brho$, $N$ is just a projective envelope of $\F\wh{\otimes}_{R_{\brho}^{\eta}} N$ in $\frak{C}_{G,\psi}(\cO)$. Hence  (N3) follows from \cite[Thm.~5.2]{Paskunas-BM}.
\end{remark}

\begin{proposition}\label{prop:kappa}
Assume $\brho$ is generic. Then there is an isomorphism $\F\wh{\otimes}_{R_{\brho}^{\eta}}N\cong\pi(\brho)^{\vee}$.
\end{proposition}
\begin{proof}
See the proof of \cite[Prop.~6.1]{Paskunas-BM}.\footnote{We remark that in \cite[Prop.~6.1]{Paskunas-BM}, the characters $\chi_1,\chi_2$ should be swapped.}
\end{proof}

If $\Theta$ (resp. $\sigma$) is a finite free $\cO$-module (resp. $\F$-module) equipped with a continuous action of $ K$, we define
\[
M(\Theta) := \Hom^{\rm cont}_{\cO [\![ K ]\!]} (N , \Theta^{d})^{d}\quad (\mathrm{resp}.\ M(\s) := \Hom^{\rm cont}_{\cO [\![ K ]\!]} (N , \s^{\vee})^{\vee}).
\]
Then $M(\Theta)$ (resp. $M(\s)$) is a finitely generated $R_{\brho}^{\eta}$-module by (N0).

 Let  $\mathbf{w} = (a, b)$ be a pair of integers with $a<b$ and $\tau: I_{\Q_p} \to \GL_2(E)$ be an inertial type, where $I_{\Q_p}$ is the inertia subgroup of $G_{\Q_p}$. Let \[\sigma(\mathbf{w},\tau):= \Sym^{b-a-1} E^2 \otimes {\det}^a\otimes \s(\tau) \] 
 \[ 
\s^{\rm cr}(\mathbf{w},\tau):=\Sym^{b-a-1} E^2 \otimes {\det}^a\otimes \sigma^{\rm cr}(\tau),
\]
 where $\sigma(\tau)$ and $\sigma^{\rm cr}(\tau)$ are finite-dimensional representations of $K$ over $E$ associated to $\tau$ by the inertial local Langlands correspondence \cite{Henniart} (see \S\ref{section-inertial-type} for details). Let $ R^{\eta }_{\brho}( \mathbf{w},\tau)$ (resp. $ R^{\eta ,{\rm cr}}_{\brho}(\mathbf{w}, \tau)$) denote the reduced $p$-torsion-free quotient of $R_{\brho}^{\eta}$ which parametrizes potentially semistable (resp. potentially crystalline) deformations of $\brho$ of Hodge-Tate weights $\mathbf{w}$ and type $\tau$. It is well-known that these rings are nonzero only if $\eta\varepsilon^{-(a+b)}|_{I_{\Q_p}}\sim \det\tau$, in which case they have Krull dimension $2$. This requires in particular that $\eta$ is locally algebraic.

Recall the following theorem of Pa\v{s}k\=unas.

\begin{theorem} \label{prop::support}
Let $\mathbf{w}, \tau$ be as above. Let $\Theta$ be any $K$-stable $\cO$-lattice in $\sigma(\mathbf{w},\tau)$ (resp. $\sigma^{\rm cr}(\mathbf{w},\tau)$).  Then $R_{\brho}^{\eta} /\Ann_{R^{\eta}_{\brho}}( M (\Theta))$ is equal to $R_{\brho}^{\eta }(\mathbf{w},\tau)$ (resp. $R_{\brho}^{\eta,{\rm cr}}(\mathbf{w},  \tau )$).
\end{theorem}
\begin{proof}
See \cite[Cor.~6.5]{Paskunas-BM}.
\end{proof}

Let $\delta:G_{\Q_p}\ra \cO^{\times}$ be a continuous character that is trivial modulo $p$, not necessarily locally algebraic. Twisting by $\delta$ induces a natural isomorphism of $\cO$-algebras 
\begin{equation}\label{eq:tw-delta}
\mathrm{tw}_{\delta}:R_{\brho}^{\eta\delta^2}\simto R_{\brho}^{\eta}.\end{equation} 
By a similar discussion as in \cite[\S6.1]{CEGGPS2}, we have the following variant of Theorem \ref{prop::support}.
\begin{corollary}\label{cor::support}
Assume $\eta\delta^2\varepsilon^{-(a+b)}|_{I_{\Q_p}}\sim \det\tau$. Let $\Theta$ be any $K$-stable $\cO$-lattice in $\sigma(\mathbf{w},\tau)\otimes\delta^{-1}\circ\det$ (resp. $\sigma^{\rm cr}(\mathbf{w},\tau)\otimes\delta^{-1}\circ\det$). Then $R_{\brho}^{\eta}/\mathrm{Ann}_{R_{\brho}^{\eta}}(M(\Theta))$ is equal to $\mathrm{tw}_{\delta}(R_{\brho}^{\eta\delta^2}(\mathbf{w},\tau))$ (resp. $\mathrm{tw}_{\delta}(R_{\brho}^{\eta\delta^2, \rm cr}(\mathbf{w},\tau))$). As a consequence, we have isomorphisms of $\cO$-algebras
\begin{equation}\label{eq:cor:twist-isom}\tw_{\delta}: R_{\brho}^{\eta\delta^2}(\mathbf{w},\tau)\simto R_{\brho}^{\eta}/\mathrm{Ann}_{R_{\brho}^{\eta}}(M(\Theta))\end{equation}
(resp. for $R_{\brho}^{\eta\delta^2,\rm cr}(\mathbf{w},\tau)$).
\end{corollary}

 If $\s$ is a finite dimensional $\F[K]$-module, by Proposition \ref{prop:kappa} we have
\begin{equation}\label{eqn-M-kappa-functor}
\F\wh{\otimes}_{R_{\brho}^{\eta}} M(\s)  = \Hom_{K} (\s , \pi(\brho)).
\end{equation}
It follows from (N3) and Nakayama's lemma that  $M(\sigma)\neq0$ if and only if $\sigma$ admits at least one Jordan-H\"older factor lying in  $W(\brho)$.

\subsubsection{A criterion for regularity}\label{subsubsection:regular}
\begin{lemma}\label{lem:regular}
Let $\sigma\in \Mod_{K}^{\rm sm}(\F)$ be of finite length. Assume that,  taking into account multiplicities, $\JH(\sigma)$ contains exactly one element in $W(\brho)$. Then $M(\sigma)$ is a cyclic $R_{\brho}^{\eta}$-module and isomorphic to $\F[\![x]\!]$ where $x\in R_{\brho}^{\eta}$ is as in (N3).
\end{lemma}
\begin{proof}
See (the end of) the proof of \cite[Thm.~6.6]{Paskunas-BM}.
\end{proof}

Recall that $\cO$ is unramified over $\Z_p$.

\begin{proposition}\label{prop:regular}
Let $\mathbf{w}, \tau$ be as above. Assume that there exist two $K$-stable $\cO$-lattices $\Theta_1, \Theta_2$ in $\sigma(\mathbf{w},\tau)$ (resp. $\sigma^{\rm cr}(\mathbf{w},\tau)$) such that the following conditions hold:
\begin{enumerate}
\item[(a)]  $p\Theta_1\subset\Theta_2\subset\Theta_1$ and $ \dim_{\F}\Hom_{K}(\Theta_i/p\Theta_i,\pi(\brho))=1$ for $i=1,2$; \medskip
\item[(b)] taking into account multiplicities, $\JH(\Theta_1/\Theta_2)$ contains exactly one element in $W(\brho)$.
\end{enumerate}
Then $R^{\eta}_{\brho}(\mathbf{w},\tau)$ (resp. $R_{\brho}^{\eta,\rm cr}(\mathbf{w},\tau)$) is a regular local ring.
\end{proposition}
\begin{proof}
We only treat the case for $R^{\eta}_{\brho}(\mathbf{w},\tau)$. By Nakayama's lemma and \eqref{eqn-M-kappa-functor}, Condition (a) implies that $M(\Theta_1)$ and $M(\Theta_2)$ are both cyclic modules over $R_{\brho}^{\eta}$. Hence $M(\Theta_1)$ and $M(\Theta_2)$ are isomorphic to $R^{\eta}_{\brho}(\mathbf{w},\tau)$ by Theorem \ref{prop::support}.

The exact sequence $0\ra \Theta_2\ra \Theta_1\ra \Theta_1/\Theta_2\ra0$ induces a sequence of $R_{\brho}^{\eta}$-modules
\[0\ra M(\Theta_2)\To{f} M(\Theta_1)\ra M(\Theta_1/\Theta_2)\ra0\]
which is again exact by (N3). Since both $M(\Theta_1) $ and $M(\Theta_2)$ are isomorphic to $R_{\brho}^{\eta}(\mathbf{w},\tau)$, the morphism $f$ is equal to the multiplication by some element $y\in R_{\brho}^{\eta}(\mathbf{w},\tau)$. On the other hand, by Lemma \ref{lem:regular} Condition (b)  implies that $M(\Theta_1/\Theta_2)$ is isomorphic to $\F[\![x]\!]$.  This means that  $R_{\brho}^{\eta}(\mathbf{w},\tau)/(y)$ is a regular local ring of Krull dimension $1$.   Since $R_{\brho}^{\eta}(\mathbf{w},\tau)$ has Krull dimension $2$, it is also regular.
\end{proof}

\subsection{Potentially crystalline deformation rings of tame supercuspidal inertial types }\label{section-def-ring}
In this subsection, we assume $\brho$ is of one of the following forms:
\begin{enumerate}[label=(C\arabic*),ref=(C\arabic*)]
\item\label{C1} $\brho$ is in \ref{Case1} of \S\ref{section-GL2-Serrewt} with $2 \leq r\leq p-3;$

\item\label{C2} $\brho$ is in \ref{Case2} of \S\ref{section-GL2-Serrewt} with $1 \leq r\leq p-3.$
\end{enumerate}
In particular, $\brho$ is generic (see \S\ref{section-Morra}). We study the properties of deformation rings of tame supercuspidal inertial types and Hodge-Tate weights $(0,1)$ and $(0,2)$ in the cases \ref{C1} and \ref{C2} separately. The main result is Theorem \ref{thm-regular-HT02}.   

Recall that given a pair of integers $(a,b)$ with $1\leq a \leq p-3,$ we can associate
\begin{enumerate}
\item[$\bullet$] characters $\psi_i,$ $1 \leq i \leq 3,$ introduced in (\ref{equation-choice-psi});
\item[$\bullet$] tame supercuspidal inertial types $\tau_i=\psi_i\oplus \psi_i^p$ satisfying $\s(\tau_i) = \Theta(\psi_i)$ (cf. Lemma \ref{lem::tame-cuspidal-type});
\item[$\bullet$] lattices $R_1, R_2, R_3, R , \wt{R}$ introduced in \S\ref{subsection:gluing} satisfying ${\rm cosoc}(\mathcal{R} / p \mathcal{R}) = \s_{a,b+1}$ for any $\mathcal{R} \in \{R_1, R_2, R_3, R , \wt{R}\}$.
\end{enumerate}We choose $(a,b)$ as follows:
\begin{enumerate}
\item[$\bullet$] in the case \ref{C1}, let $(a, b) \in \{ (r,s), (p-1-r, r+s)\};$

\item[$\bullet$]  in the case \ref{C2}, let $(a,b) = (r,s).$
\end{enumerate}
Then  $\s_{a,b+1} $ lies in $W(\brho)$ by Theorem \ref{thm:Serrewt-BDJ}. For $\mathcal{R} \in \{R_1, R_2, R_3, R , \wt{R}\},$ we denote by
\begin{equation}\label{eqn::I_R}
I_{\mathcal{R}}: = \Ann_{R_{\brho}^{\eta}} (M(\mathcal{R}))
\end{equation}
the annihilator of $M(\mathcal{R})$ in $R_{\brho}^{\eta}.$ \begin{proposition}\label{prop--cyclic-lattices}
$M(\mathcal{R})$ is a (nonzero) cyclic $R_{\brho}^{\eta}$-module for $\mathcal{R} \in \{R_1, R_2, R_3, R , \wt{R}\}.$ As a consequence $M(\mathcal{R}) \cong R_{\brho}^{\eta} / I_{\mathcal{R}}.$
\end{proposition}
\begin{proof}
By Nakayama's lemma, it suffices to show $ M(\mathcal{R})/\frak{m} $ is of dimension $1$ over $\F,$ where $\frak{m}$ denotes the maximal ideal of $R_{\brho}^{\eta}.$  Since $\s_{a,b+1}\in W(\brho)$ is a quotient of $\mathcal{R}/p\mathcal{R} ,$ we always have $\dim_{\F} M(\mathcal{R}/p\mathcal{R})/ \fm \geq 1$ by (N3) of \S\ref{section-Morra}.

To show  $\dim_{\F} M(\mathcal{R}/p\mathcal{R})/ \fm \leq 1,$ we note that $\mathcal{R}/p\mathcal{R}$ is a quotient of $(\Proj_{\F[\![K / Z_1]\!]} \s_{a, b+1}) /\frak{m}_{K_1}^2 $ by Lemma \ref{lem--first-glue}, Corollary \ref{cor-wtL-structure} and Proposition \ref{prop:tildeR-p-3}. Hence by (N3) of \S\ref{section-Morra}
\[
\dim_{\F} M (\mathcal{R}/ p\mathcal{R} )/ \fm \leq \dim_{\F} M \big( (\Proj_{\F[\![K / Z_1]\!]} \s_{a, b+1})/\frak{m}_{K_1}^2 \big)/ \fm .
\]
If $\brho$ satisfies either \ref{C1} or \ref{C2}, then by (\ref{eqn-M-kappa-functor}) and Theorem \ref{thm-morra}, we have
\[
\dim_{\F} M \big( (\Proj_{\F[\![K / Z_1]\!]} \s_{a, b+1})/\frak{m}_{K_1}^2 \big) / \frak{m}  =  \dim_{\F} \Hom_{K} \big((\Proj_{\F[\![K / Z_1]\!]} \s_{a, b+1}) /\frak{m}_{K_1}^2, \pi(\brho) \big) = 1.
\]
Hence $\dim_{\F} M (\mathcal{R}/ p\mathcal{R} )/\frak{m} = 1.$
\end{proof}

\begin{remark}\label{rem:cyclic-L}
For $i\in\{2,3\}$, we have constructed $K$-stable $\cO$-lattices $L, L'$ in $\unSym^1E^2\otimes \Theta(\psi_i)$ in Proposition \ref{prop-reduction-L}. The cosocle of $L/pL$ (resp. $L'/pL'$) need not be irreducible, but $M(L)$ (resp. $M(L')$) is still cyclic over $R_{\brho}^{\eta}$. 

Indeed, if  $\JH\left(\overline{\unSym^1E^2\otimes\Theta(\psi_i)}^{\rm ss}\right)\cap W(\brho)$ consists of one element then the claim is obvious. Otherwise, $\brho$ satisfies \ref{C1} and $W(\brho)$ consists of two elements, say $W(\brho)=\{\sigma_1,\sigma_2\} \subset \JH\left(\overline{\unSym^{1}E^2\otimes\Theta(\psi_i)}^{\rm ss}\right)$.  By Proposition \ref{prop-reduction-L}, one of the nonsplit extensions, $E=(\sigma_1\ligne\sigma_2)$ or $E'=(\sigma_2\ligne\sigma_1)$,  occurs in $L/pL$ (resp. $L'/pL'$). As in the proof of Proposition \ref{prop--cyclic-lattices}, $M(E)$ and $M(E')$ are cyclic over $R_{\brho}^{\eta}$, from which the  claim follows as $M(\sigma)=0$ for $\sigma\notin W(\brho)$.
\end{remark}

\begin{corollary}\label{cor-ideals-relation}
We have
\begin{enumerate}
\item[(i)] $I_{R_1} + I_{R_2} = (p, I_{R_1})$ and $I_R = I_{R_1} \cap I_{R_2}.$

\item[(ii)] $I_{R} + I_{R_3} = (p, I_{R_1})$ and $I_{\wt{R}} = I_{R_1}\cap I_{R_2} \cap I_{R_3}.$
\end{enumerate}
\end{corollary}

\begin{proof}
Recall the following lemma from \cite[Lem.~8.11]{Hu-Wang}.

\begin{lemma}\label{lemma-cyclic-CA}
Let $(A,\fm_A)$ be a commutative noetherian local ring with $k=A/\fm_A$. Let   $\cI_0, \cI_1,\cI_2$ be ideals of $A$ such that $\cI_1,\cI_2\subset \cI_0\subset \fm_A$. Consider the natural surjective homomorphism
$A/\cI_1\oplus A/\cI_2\twoheadrightarrow A/\cI_0$. Then $\Ker(A/\cI_1\oplus A/\cI_2\twoheadrightarrow A/\cI_0)$ is a cyclic $A$-module if and only if $\cI_1+\cI_2=\cI_0.$
\end{lemma}

By (N3), the  sequence (\ref{equation-ses3}) induces a short exact sequence
\[
0 \to M(R) \to M(R_1) \oplus M(R_2) \to M(R_1 / pR_1) \to 0.
\]
Since $M(R)$ is cyclic over $R_{\brho}^{\eta}$ by Proposition \ref{prop--cyclic-lattices}, we deduce (i) using Lemma \ref{lemma-cyclic-CA} and the fact that\footnote{In general, if $A$ is a commutative ring, $I$ an ideal of $A$ and  $M$ a finite $A$-module, then $\mathrm{Ann}_A(M)+I\subseteq \mathrm{Ann}_A(M/IM)$ and their radicals coincide. In our situation, $(p,I_{\Theta_1})$ is a prime ideal, so we have the claimed equality.}
\[\mathrm{Ann}_{R_{\brho}^{\eta}}(M(R_1/pR_1))=\mathrm{Ann}_{R_{\brho}^{\eta}}(M(R_1)/p)=(p,I_{R_1}).\]  Similarly we obtain (ii) by using the short exact sequence (\ref{eq:tildeR}).
\end{proof}
 
Let $\delta:G_{\Q_p}\ra\cO^{\times}$ denote the character, via the local class field theory,  sending $x\in \Q_p^{\times}\mapsto \mathrm{pr}(x)^{1/2}\in 1+p\Z_p$. By Theorem \ref{prop::support} and Corollary \ref{cor::support}, we have
\begin{equation}\label{eq:isom-twist}
R_{\brho}^{\eta} / I_{R_1} = R_{\brho}^{\eta }((0,1), \tau_1), ~ R_{\brho}^{\eta} / I_{R_2} \overset{\tw_{\delta}^{-1}}{\cong} R_{\brho}^{\eta\delta^2}((0,2),\tau_2), ~R_{\brho}^{\eta} / I_{R_3} \overset{\tw_{\delta}^{-1}}{\cong} R_{\brho}^{\eta\delta^2 }((0,2),\tau_3).
\end{equation}
\begin{proposition}\label{thm-regular-HT01}
The ring  $R_{\brho}^{\eta }((0,1), \tau_1)$  is a regular local ring.
\end{proposition}
\begin{proof}
Recall from \S\ref{subsection:lattice} that there exist two $K$-stable $\cO$-lattices $T,T'\subset \sigma(\tau_1)$ such that $pT\subset T'\subset T$ and $T/T'\cong \sigma_{a,b+1}$ and $\rcosoc(T/pT)\cong \sigma_{a,b+1}$. Here, if $\JH\left(\overline{\sigma(\tau_1)}^{\rm ss}\right)\cap W(\brho)$ consists of only one element, then we take $T'=pT$. In any case, the cosocle of $T'/pT'$ is  irreducible. Using Theorem \ref{thm-morra}, it is easy to check that \[\dim_{\F}\Hom_K(T/pT,\pi(\brho))=\dim_{\F}\Hom_K(T'/pT',\pi(\brho))=1.\]
  The result then follows from Proposition \ref{prop:regular}.
\end{proof}

\begin{remark}
If $\brho$ is generic in the sense of \cite[Def.~11.7]{BP}, Proposition \ref{thm-regular-HT01} is a direct consequence of  \cite[Thm.~7.2.1]{EGS}.
\end{remark}

\begin{theorem}\label{thm-regular-HT02}
The rings $R_{\brho}^{\eta\delta^2}((0,2),\tau_2)$ and $R_{\brho}^{\eta\delta^2}((0,2),\tau_3)$  are  regular local rings.
\end{theorem}
\begin{proof}
Assume $\brho$ is in the case \ref{C1}.  For $R_{\brho}^{\eta\delta^2  }((0,2),\tau_2),$  it is equivalent to proving that $R^{\eta}_{\brho}/I_{R_2}$ is a regular local ring by \eqref{eq:isom-twist}. Note that $\s_{a,b+1}$ is the unique Serre weight in the intersection $\JH\left(\overline{\unSym^1E^2 \otimes\s(\tau_2)}^{\rm ss}\right) \cap W(\brho).$ The assertion follows from Proposition \ref{prop:regular}, by choosing any $K$-stable $\cO$-lattice $\Theta_1$ in $\unSym^1E^2\otimes\sigma(\tau_2)$, and taking $\Theta_2=p\Theta_1$ in Proposition \ref{prop:regular}.

We now consider $R_{\brho}^{\eta\delta^2}((0,2),\tau_3)$, equivalently $R_{\brho}^{\eta}/I_{R_3}$ via \eqref{eq:isom-twist}. By Proposition \ref{prop-reduction-L}, there are $K$-stable $\cO$-lattices $L,L'$ of $\unSym^1 E^2\otimes\Theta(\psi_3)$ such that $pL \subset L' \subset L$ and
\[
L / L' = \s_{a, b+1} \oplus \s_{a-2, b+2}.
\]
Note that $\s_{a-2,b+2}\notin W(\brho)$. Using Remark \ref{rem:cyclic-L}, the result follows from Proposition \ref{prop:regular}.

Assume $\brho$ is in the case \ref{C2}. Then $\brho$ has only one Serre weight $\s_{a,b+1}$, and we conclude as in the first paragraph.
\end{proof}
 
\subsection{Endomorphism rings and faithfulness}\label{subsection:Endo}

In this subsection, we assume $\brho$ is reducible nonsplit and isomorphic to $\smatr{1}{*}0{\omega}$. Let $N\in \mathfrak{C}_{G/Z_G}(\cO)$ be as in \S\ref{section-Morra}.  In this case $N$ is isomorphic to a projective envelope of $(\Ind_{B(\Q_p)}^{G}\omega\otimes\omega^{-1})^{\vee}$ in $\mathfrak{C}_{G/Z_G}(\cO)$. 

 Let $ (A,\fm_A)$ be a pseudo-compact \emph{flat} local $\cO$-algebra with residue field $\F$. Set  $R:=A\widehat{\otimes}_{\cO}R_{\brho}^{\eta}$ and \[M:=R\widehat{\otimes}_{R_{\brho}^{\eta}}N \cong A\widehat{\otimes}_{\cO}N.\]
Then $M\in \frak{C}_{G/Z_G}(R)$. In fact, as in \cite[Lem.~4.9]{CEGGPS2} we show that $M^{\vee}$ is admissible in $\Mod_G^{\rm sm}(R)$, and so $M\in \mathfrak{C}_{G/Z_G}(\cO)$.

\begin{lemma}\label{lem:M=proj}
$M$ is a projective object in $\mathfrak{C}_{G/Z_G}(\cO)$.
\end{lemma}
\begin{proof}
By assumption $A$ is $\cO$-flat. Since pseudo-compact flat $\cO$-modules are projective (see e.g. \cite[Prop.~3.1]{Brumer}), $A$ is a projective $\cO$-module.
By definition of $M$  we have
\begin{equation}\label{eq:adjoint}\Hom_{\frak{C}_{G/Z_G}(\cO)}(M,-)\cong \Hom_{\frak{C}_{G/Z_G}(\cO)}(A\widehat{\otimes}_{\cO}N,-)\cong\Hom_{\cO}^{\rm cont}(A,\Hom_{\frak{C}_{G/Z_G}(\cO)}(N,-))\end{equation}
from which  the result  follows.
\end{proof}

\begin{lemma}\label{lem:Ext1-M}
We have $\Hom_{\frak{C}_{G/Z_G}(R)}(M,\ide_G^{\vee})=0$ and $\Ext^1_{\frak{C}_{G/Z_G}(R)}(M,\ide_G^{\vee})=0$.
\end{lemma}

\begin{proof}
The first assertion follows from \eqref{eq:adjoint} because $\Hom_{\frak{C}_{G/Z_G}(\cO)}(N,\ide_G^{\vee})=0$, see (N1) in \S\ref{section-Morra}. For the second, we work on the dual side and show $\Ext^1_{R[G]}(\ide_G,M^{\vee})=0$. By Lemma \ref{lem:M=proj}, $M$ is a projective object in $\frak{C}_{G/Z_G}(\cO)$, so dually $M^{\vee}$ is an injective object in $\Mod_{G/Z_G}^{\rm l.adm}(\cO)$. Consider an extension
\[0\ra M^{\vee}\ra \mathcal{E}\ra \ide_G\ra0\]
in $\Mod_{G/Z_G}^{\rm l.adm}(R)$. It must split in $\Mod_{G/Z_G}^{\rm l.adm}(\cO)$, so we may find $v\in \mathcal{E}$ such that $\langle \cO[G].v\rangle \cong \ide_G$. It suffices to show that $R$ acts on $v$ via the quotient $R\twoheadrightarrow R/\fm_R\cong\F$. This is clear, since if $x\in \fm_{R}$, then $x\cdot v\in M^{\vee}$ and, if it were nonzero, then it would  generate a subrepresentation of $M^{\vee}$ isomorphic to $\ide_G$, which is not possible by  the first assertion.
\end{proof}

\begin{proposition}\label{prop:Hom-R}
For any compact $R$-module $\mathrm{m}$, the natural map $v\mapsto (m\mapsto(v\widehat{\otimes} m))$ (where  $v\in \mathrm{m}$ and $m\in M$) induces an isomorphism
 \[\mathrm{m}\simto\Hom_{\frak{C}_{G/Z_G}(R)}(M,\mathrm{m}\widehat{\otimes}_{R}M).\]
\end{proposition}

\begin{remark}
Note that $M$ is \emph{not}  projective  in $\frak{C}_{G/Z_G}(R)$ so that we can not apply \cite[Lem.~2.9]{PaskunasIHES}.
\end{remark}

\begin{proof}
The proof is similar to \cite[Prop.~3.12]{HP}. As in \emph{loc. cit.}, we may assume that $\mathrm{m}$ is of finite length. In particular, the completed tensor product $\mathrm{m}\widehat{\otimes}_{R}M$ coincides with the usual one.

We proceed by induction on the length of $\mathrm{m}.$ Note that since $R$ is a local ring, any $R$-module of length $1$ is isomorphic to $R/\fm_R\cong\F$.  If $\mathrm{m}\cong \F$, we need to show that $\Hom_{\frak{C}_{G/Z_G}(R)}(M,\F\otimes_{R}M)\cong\F$. But any morphism $M\ra \F\otimes_{R}M$ in $\frak{C}_{G/Z_G}(R)$ factors through
\[M\twoheadrightarrow \F\otimes_{R} M\ra \F\otimes_{R}M, \]
so the assertion is reduced to
\[\End_{\frak{C}_{G/Z_G}(\cO)}(\F\otimes_{R}M)=\End_{\frak{C}_{G/Z_G}(\cO)}(\F\otimes_{R_{\brho}^{\eta}}N)\cong\F\]
which is a direct consequence of Proposition \ref{prop:kappa}. If the length of $\mathrm{m}$ is $\geq 2$, let $\mathrm{m}_1\subset \mathrm{m}$ be a proper $R$-submodule such that $\mathrm{m}_2:=\mathrm{m}/\mathrm{m}_1$ has length $1$, i.e. $\mathrm{m}_2\cong \F$.  We then obtain a long exact sequence
\begin{equation}\label{eq:seq-Tor1}\Tor_1^{R}(\mathrm{m}_2,M)\ra \mathrm{m}_1\otimes_{R}M\ra\mathrm{m}\otimes_{R}M\ra \mathrm{m}_2\otimes_{R}M\ra0.\end{equation}
Since $\mathrm{m}_2\cong\F$ and $R$ is flat over $R_{\brho}^{\eta}$ by construction, we have \[\Tor_1^{R}(\mathrm{m}_2,M)= \Tor_1^{R}(\F,R\widehat{\otimes}_{R_{\brho}^{\eta}}N)\cong \Tor_1^{R_{\brho}^{\eta}}(\F,N)\cong (\ide_{G}^{\vee})^{\oplus 2},\]
where the last isomorphism follows from \cite[Prop.~3.30]{HuJEMS}. By applying $\Hom_{\frak{C}_{G/Z_G}(R)}(M,-)$ to \eqref{eq:seq-Tor1} and using Lemma \ref{lem:Ext1-M}, we obtain the following  short exact sequence
\[0\ra \Hom_{\frak{C}_{G/Z_G}(R)}(M,\mathrm{m}_1\otimes_{R}M)\ra \Hom_{\frak{C}_{G/Z_G}(R)}(M,\mathrm{m}\otimes_{R}M)\ra \Hom_{\frak{C}_{G/Z_G}(R)}(M,\mathrm{m}_2\otimes_{R}M).\]
By inductive hypothesis, we have $\mathrm{m}_i\simto \Hom_{\frak{C}_{G/Z_G}(R)}(M,\mathrm{m}_i\otimes_{R}M)$ for $i\in\{1,2\}$, hence the result  using the snake lemma.
\end{proof}

\begin{corollary}\label{cor:R'-faithful}
We have $\End_{\frak{C}_{G/Z_G}(R)}(M)\cong R$. In particular, $R$ acts faithfully on $M$.
\end{corollary}

\begin{proposition}\label{prop:faithful}
Let $x_1,\dots,x_g\in R$ be  an $M$-regular sequence. Then $(x_1,\dots,x_g)$ is also $R$-regular and \[\End_{\frak{C}_{G/Z_G}(R)}\big(M/(x_1,\dots,x_i) M\big)\cong R/(x_1,\dots,x_i)R\] for any $1\leq i\leq g$.
\end{proposition}
\begin{proof}
The proof is analogous to \cite[Prop.~5.11]{HuJEMS}.

Since $R$ acts faithfully on $M$ by Corollary \ref{cor:R'-faithful} and $x_1$ is $M$-regular, $x_1$ is also $R$-regular.
By Proposition \ref{prop:Hom-R}, we have
\[R/x_1R\simto \Hom_{\frak{C}_{G/Z_G}(R)}(M,M/x_1M)=\End_{\frak{C}_{G/Z_G}(R)}(M/x_1M).\]
This in turn shows that $R/x_1R$ acts faithfully on $M/x_1M$, hence $x_2$ is $R/x_1R$-regular because it is $M/x_1M$-regular by assumption.
We may thus continue the above argument to conclude.
\end{proof}

\begin{remark}\label{rem:faithful}
Proposition \ref{prop:faithful} could be used to prove a big ``$R=\mathbb{T}$'' theorem, see the proof of Proposition \ref{prop:S1-rho2} below.  Such a result is proved in  \cite[Thm.~B(3)]{Gee-Newton} when $R_{\brho}^{\eta}$ is formally smooth, by first proving that a suitable patched module $M_{\infty}$ is faithfully flat over the patched ring $R_{\infty}$ and then passing to the quotient. However, when $\brho \sim \smatr{1}{*}0{\omega} \otimes \chi$,  $R_{\brho}^{\eta}$ is not formally smooth and the patched module $M_{\infty}$ is \emph{not} flat over $R_{\infty}$, so the argument in \cite{Gee-Newton} does not apply.  Besides, this case is also excluded in \cite[Thm.~1.3]{Em3}, so Proposition \ref{prop:faithful} may be of independent interest.
\end{remark}

\section{Automorphic forms and big patched modules} \label{section-autom-forms}

 Let $F$ be a totally real extension of $\Q$ in which $p$ is unramified, and let $\OC_F$ be its ring of integers. Let $\Sigma_p$ denote the set of places of $F$ dividing $p$ and let $\Sigma_{\infty}$ denote the set of infinite places of $F.$ For any place $v$ of $F,$ let $F_v$ denote the completion of $F$ at $v$ with ring of integers $\OC_{F_v},$ uniformizer $\varpi_v$ and residue field $k_{F_v}.$ Let $q_v$ denote the cardinality of $k_{F_v}.$ Let $\AM_{F,f}$ denote the ring of finite ad\`eles of $F.$ If $S$ is a finite set of finite places of $F,$ let $\AM^S_{F,f}$ denote the ring of finite ad\`eles outside $S.$ Recall $G_F=\Gal(\Fov/F)$  and $G_{F_v}=\Gal(\Fov_v/F_v)$. By fixing an embedding $\Fov\into \Fov_v$, $G_{F_v}$ is identified with the decomposition group at $v.$ We let $\Frob_v\in G_{F_v}$ denote a (lift of the) geometric Frobenius element, and let $\Art_{F_v}$ denote the local Artin map, normalized so that it sends  $\varpi_v$ to $\Frob_v$. The global Artin map is denoted by $\Art_F$ which is compatible with the local Artin map. We denote by $\rec_{v}$  the local Langlands correspondence normalized as in the introduction of \cite{Harris-Taylor}, so that if $\pi$ is a smooth irreducible $\overline{\Q}_p$-representation of $\GL_2(F_{v}),$ then $\rec_{v}(\pi)$ is a Weil-Deligne representation of the Weil group $W_{F_v}$ defined over $\overline{\Q}_p.$ Recall that $\F$ is a sufficiently large finite extension of $\F_p,$ $\cO = W(\F)$ and $E = \cO[1/p].$ We prepare the global setup we need in this section.

\subsection{Tame types and the inertial local Langlands correspondence}\label{section-inertial-type}

Let $I_{F_{v}}$ be the inertia subgroup of $G_{F_{v}}.$ An {\em inertial type} at $v$ is a two-dimensional representation $\tau: I_{F_{v}} \to \GL_2(\overline{\Q}_p)$ with open kernel which extends to a representation of $G_{F_{v}}.$ We say $\tau$ is a {\em discrete series} inertial type if it is either scalar, or extends to an irreducible representation of $G_{F_{v}}.$ In the latter case, we call $\tau$  {\em supercuspidal}. We say $\tau$ is {\em tame} if it is trivial on the wild inertia subgroup. Under Henniart's inertial local Langlands correspondence \cite{Henniart} (cf. also \cite{Kisin-FM}), there is a unique finite dimensional irreducible representation $\s(\tau)$ (resp. $\s^{\rm cr}(\tau)$) of $\GL_2(\cO_{F_{v}})$ over $\overline{\Q}_p$-vector spaces, called {\em types}, satisfying if $\pi$ is an infinite dimensional smooth irreducible representation of $\GL_2(F_{v})$, then $\Hom_{\GL_2(\cO_{F_{v}})}(\s(\tau),\pi) \neq 0$ (resp. $\Hom_{\GL_2(\cO_{F_{v}})}(\s^{\rm cr}(\tau),\pi) \neq 0$) if and only if $\rec_{v}(\pi)|_{I_{F_{v}}} \cong \tau$ (resp. $\rec_{v}(\pi)|_{I_{F_{v}}} \cong \tau$ and the monodromy operator $N$ on $\rec_{v}(\pi)$ is zero), in which case the space $\Hom_{\GL_2(\cO_{F_{v}})}(\s(\tau),\pi) $ (resp. $\Hom_{\GL_2(\cO_{F_{v}})}(\s^{\rm cr}(\tau),\pi) $) is one-dimensional. We always have $\s(\tau) = \s^{\rm cr}(\tau)$ except when $\tau = \chi\oplus \chi$, in which case $\s (\chi \oplus \chi) = {\rm sp}\otimes \chi\circ\det$ (here ${\rm sp}$ denotes the Steinberg representation of $\GL_2(k_{F_v})$ over $\overline{\Q}_p$) and $\s^{\rm cr}(\chi \oplus \chi) =  \chi\circ\det.$ Let $\psi : \F_{q_v^2}^{\times} \to \overline{\Q}_p^{\times}$ be a character such that $\psi\neq \psi^{q_v}.$ Let $\Theta(\psi)$ be the irreducible cuspidal $\overline{\Q}_p$-representation of $\GL_2(k_{F_v})$ associated to $\psi$ as in \cite{Diamond}. 
A {\em tame supercuspidal type} is an irreducible $\overline{\Q}_p$-representation of $\GL_2(\cO_{F_{v}})$ that arises by inflation from $\Theta(\psi)$ for some $\psi$ as above. 

In \cite{Gee-Geraghty} Gee--Geraghty developed an analogous theory for  $D^{\times},$ where $D$ is the nonsplit central quaternion algebra over $F_{v}.$ Let $\JL$ denote the Jacquet-Langlands correspondence giving a bijection from irreducible smooth representations of $D^{\times}$ over $\overline{\Q}_p$ to discrete series representations of $\GL_2(F_v)$ over $\overline{\Q}_p$.  Let $\tau$ be a discrete series inertial type. By the Jacquet-Langlands correspondence, there is an irreducible smooth representation $\pi_{D,\tau}$ of $D^{\times}$ such that $\rec_{v}(\JL(\pi_{D,\tau}))|_{I_{F_{v}}} \cong \tau.$ Since $F_{v}^{\times }\cO_D^{\times}$ has index two in $D^{\times},$ $\pi_{D,\tau}|_{\cO_D^{\times}}$ is either irreducible or a sum of two irreducible representations which are conjugate under the uniformizer $\varpi_D$ of $D.$ Let $\s_D(\tau)$ be {\em one} of the irreducible components of $\pi_{D,\tau}|_{\cO_D^{\times}}.$ If $\pi_D$ is a smooth irreducible $\overline{\Q}_p$-representation of $D^{\times}$, then $\Hom_{\cO_D^{\times}}(\s_D(\tau),\pi_D) \neq 0$ if and only if $\rec_{v}(\JL(\pi_D))|_{I_{F_{v}}} \cong \tau,$ in which case, $\Hom_{\cO_D^{\times}}(\s_D(\tau),\pi_D)$ is one-dimensional. If $\tau$ is a tame inertial type, then $\s(\tau)$ and $\s_D(\tau)$ can be defined over $E$ once $E$ is taken sufficiently large (and unramified), see the proof of \cite[Lem.~ 3.1.1]{EGS}.  Recall the following lemma.

\begin{lemma}\label{lem::tame-cuspidal-type}
Let $\psi : \F_{q_v^2}^{\times} \to E^{\times}$ with $\psi\neq \psi^{q_v}.$ Let $\tau: = \psi\oplus \psi^{q_v} $ be the supercuspidal inertial type associated to $\psi,$ where we denote by $\psi$ the composition $I_{F_v}   \onto \F_{q_v^2}^{\times}\To{\psi} E^{\times}.$ Then $\s(\tau) = \Theta(\psi)$ and $\pi_{D,\tau}|_{\cO_{D}^{\times}} = \psi\oplus \psi^{q_v}.$
\end{lemma}
\begin{proof}
The assertion on $\s(\tau)$ follows from Henniart's construction in \cite{Henniart}. The assertion on $\s_D(\tau )$ follows from the classical Jacquet-Langlands correspondence, see for example \cite[Ch.~13]{BH}.
\end{proof}

\subsection{Automorphic forms, Galois representations and the big patched modules}

We define the space of automorphic forms. Let $B$ be a quaternion algebra over $F.$ Fix a maximal order $\cO_B$ of $B.$ Let $\Sigma_B$ be the set of primes $v$ in $F$ at which $B$ is ramified. Let $\infty_{F}$ be a fixed infinite place of $F.$ We say $B$ is \emph{definite} if it is ramified at all infinite places; $B$ is \emph{indefinite} if it splits at $\infty_F$ and ramifies at all other infinite places. If $v$ is a finite place of $F,$ let $\cO_{B_{v}}^{\times}$ denote the maximal compact subgroup of $B_{v}^{\times} : = (B\otimes_F F_v)^{\times}.$ For $v \notin \Sigma_B,$ we fix an isomorphism $ B_{v}^{\times} \cong \GL_2(F_v)$ so that $\cO_{B_{v}}^{\times}$ is identified with $\GL_2(\cO_{F_v}).$  Let $\psi:F^\times\setminus \AM_{F,f}^\times\to \cO^\times$ be a continuous character. Via the global Artin map, $\psi$ induces a continuous character $G_{F}\to \cO^{\times}$ which, by abuse of notation, is again denoted by $\psi.$ Assume  moreover that $(F, B) \neq (\Q, \GL_{2}).$

Let $U$ be a compact open subgroup of $(B\otimes_{F} \AM_{F,f})^\times.$ We denote by $Y^B_U$ the finite set $B^\times\setminus (B\otimes_{F} \AM_{F,f})^\times/U$ if $B$ is definite. If $B$ is indefinite, let $Y^B_U$ denote the quotient of $X^B_U$ by the action of the finite group $\A_{F,f}^{\times} / (F^{\times} (\A_{F,f}^{\times} \cap U))$, where $X_U^B$ is the associated Shimura curve as in \cite{BreuilDiamond}, which is the same convention used in \cite{Em3} and \cite{Scholze} but is different from the convention used in \cite{BDJ}.

From now on till the end of the paper, we assume that $\Sigma_B$ and $\Sigma_p$ intersect at a unique place $v$ above $p.$ Fix $U^p = \prod_{w \nmid p} U_w$ a compact open subgroup of $(B\otimes_{F} \A^{\Sigma_p}_{F,f})^{\times}.$ For each place $w \in \Sigma_p \setminus \{v\}$, let $\s_w$ be a finite free $\cO$-module equipped with a continuous action of $U_w$ such that $F_w^{\times }\cap U_w$ acts by $\psi^{-1}|_{F_w^{\times}}.$ Denote  
\[
\s_p^v = \otimes_{ w \in \Sigma_p \setminus \{v\} } \s_w.
\]
Let $U^{v}= U^p U_p^{v} \subset (B \otimes_F \A^{\{v\}}_{F,f})^{\times}.$ Then $\s_p^v$ is equipped with an action of $U^v$ via the projection $U^v \onto U^{v}_p.$ We extend this action to $U^v \A_{F,f}^{\times}$ by letting $\A_{F,f}^{\times}$ act by $\psi^{-1}.$ Assume that $U_v$ is a compact open subgroup of $\GL_2(\cO_{F_v}) $ such that $\psi|_{U_{v}\cap \cO_{F_v}^\times}=1.$ Then $\s_p^v$ admits an action of $U^v U_v\A_{F,f}^{\times}$ by letting $U_v$ act trivially.

If $B$ is definite, set
\begin{align*}
\wt{H}^{0,B}_{\s_p^v , \psi} (U^{v}, \cO)  : = & \left \{  \vphantom{ U (\A_{F,f}^{\times}} f: B^\times\setminus (B\otimes_F \AM_{F,f})^\times \to \s_p^v ~|~ \text{$f$ is {\em continuous} and $f(g u) = u^{-1} f(g)$, } \right. \\
&  \left.  \text{$\forall g\in (B\otimes_F \AM_{F,f})^\times$,   $\forall u \in U^v U_v \A_{F,f}^{\times}$}  \right\}.
\end{align*}
If $B$ is indefinite, let $\mathcal{F}_{\s_p^v/\varpi^s}$ be the local system over $Y_{U^{v}U_{v}}^{B}$ associated to $\s_p^v/\varpi^s$ (see \cite{EmertonInvent}), and set
\[
\wt{H}^{1,B}_{\s_p^v , \psi} (U^{v}, \cO) :=  \plim_{s} \varinjlim_{U_{v}} H^1_{\textrm{\'et}} (Y_{U^{v}U_{v}}^{B} , \mathcal{F}_{\s_p^v/\varpi^s}).
\]
Both $\wt{H}^{0,B}_{\s_p^v , \psi} (U^{v}, \cO)$ and $\wt{H}^{1,B}_{\s_p^v , \psi} (U^{v}, \cO)$ carry an action of $(B\otimes_{F} F_{v})^\times.$

Let $S$ be a set of places of $F$ containing all places in $\Sigma_{\infty} \cup \Sigma_B\cup \Sigma_p,$ all places where $\psi$ is ramified, and all places $w$ such that $U_w$ is not $\cO_{B_w}^{\times}.$ Let $\overline{r}: G_{F} \to \GL_2(\F)$ be an absolutely irreducible totally odd representation. Assume $\overline{r}$ is unramified outside $S.$ Assume $\overline{\psi} := \psi \pmod{\varpi}$ is equal to $\omega \det \overline{r}.$ Denote  $\overline{r}_w \defn \overline{r}|_{G_{F_w}}.$ We make the following assumptions on $\overline{r}:$

(a) $\overline{r}$ is modular in the sense of \cite[\S3.1]{BreuilDiamond}, $\overline{r}|_{G_{F(\sqrt[p]{1})}}$ is absolutely irreducible and, if $p = 5,$ the image of $\overline{r}(G_{F(\sqrt[p]{1})})$ in ${\rm PGL}_2(\F)$ is not isomorphic to ${\rm PSL}_2(\F_5).$

(b) For $w \in S \setminus \Sigma_p$ the framed deformation ring of $\overline{r}_w$ is formally smooth over $\cO$ (cf.~\cite[Rem.~8.1.1]{BHHMS1}).

(c) If $w\nmid p$ and $w \in \Sigma_B,$ then $\overline{r}_w$ is either irreducible or  a twist of an extension of the trivial representation by $\overline{\e}.$

(d) If $w|p,$ $w\neq v,$ then $\overline{r}|_{I_{F_w}}$ is generic in the sense of \cite[Def.~11.7]{BP} (which is different from the genericity used in \S\ref{section-Morra}).

Assumption (c) is often called the compatibility condition between $B$ and $\overline{r}.$ By \cite[Corollaire 3.2.3]{BreuilDiamond}, the above assumptions guarantee the non-vanishing of $\pi^B(\overline{r}),$ where $\pi^B(\overline{r})$ is defined in \eqref{eqn:pi^B(r)}. For each $w \in \Sigma_p\setminus \{v\},$ we fix a tame inertial type $\tau_w$ over $E$ such that $\det(\tau_w)|_{I_{F_w}}=\psi|_{I_{F_w}}$ and $\JH( \overline{\s (\tau_w)}^{\rm ss})$ contains exactly one Serre weight in $W(\overline{r}_w(1))$ (\cite[Prop.~3.5.1]{EGS}). This is possible by our assumption (d) and (the proof of) \cite[Prop.~3.5.1]{EGS}. Let $\sigma^{\circ}(\tau_w)$ be an $\cO_{B_w}^{\times}$-stable $\cO$-lattice in $\sigma(\tau_w)$  and 
\begin{equation}\label{eqn::def-of-sigma}
\s^v_p :=  \otimes_{w\in \Sigma_p\setminus \{v \}} \s^{\circ}(\tau_w)^d.
\end{equation}

For $w$ a finite place of $F,$ let $R_w$ denote the universal {\em framed} deformation ring of $\overline{r}_w$ over $\OC.$ Let $R_{w}^{\psi \e^{-1}}$ denote the quotient of $R_{w}$ corresponding to liftings with determinant $(\psi|_{F^\times_w})\varepsilon^{-1}.$
If $w \in S\setminus \Sigma_p$, $R^{\psi\e^{-1}}_w$ is a formal power series ring in $3$-variables over $\cO$ by our assumption (b). If $w|p, ~w\neq v,$ let $R_{w}^{\psi \e^{-1}}((-1,0)_{\kappa},\tau_w)$ denote the reduced $p$-torsion-free quotient of $R_{w}^{\psi\e^{-1}}$ corresponding to potentially crystalline (framed) deformations of inertial type $\tau_w$ and Hodge-Tate weights $(-1,0)$ for all embeddings $\kappa: F_w\into E.$ By the choice of $\tau_w,$ $R_{w}^{\psi\e^{-1}}((-1,0)_{\kappa},\tau_w)$ is a formal power series ring  in $(3+[F_w:\Q_p])$-variables over $\cO$ \cite[Thm.~7.2.1]{EGS}. Let
\[
R_S := \wh{\otimes}_{w\in S}R_w^{\psi\e^{-1}}
\]
and
\[
R^{ \loc}:= R_v^{\psi\e^{-1}} \wh{\otimes }\big(\wh{\otimes}_{w|p, w\neq v } R_w^{\psi\e^{-1} }((-1,0)_{\kappa},\tau_w) \big)\wh{\otimes } \big(\widehat{\otimes}_{w\in S\setminus \Sigma_p }R_w^{\psi\e^{-1}}\big).
\]

Let $R^{\Box, \psi\e^{-1}}_{\overline{r},S}$ (resp. $R^{\psi\e^{-1}}_{\overline{r},S}$) be the framed (resp. universal) deformation ring of $\overline{r}$ parametrizing liftings (resp. deformations) of $\overline{r}$ which are unramified outside $S$ with determinant $\psi\varepsilon^{-1}$ as in \cite[\S~5.4.1]{Gee-Kisin}. Let $r^{\rm univ}$ denote the universal deformation of $\overline{r}$ over $R^{\psi\e^{-1}}_{\overline{r} , S}.$ Define $R^{\Box,\psi\e^{-1}, \loc}_{\overline{r},S}:=R^{\Box, \psi\e^{-1}}_{\overline{r},S}\wh{\otimes}_{R_S}R^{ \loc}.$ Let $R^{\psi\e^{-1},\loc}_{\overline{r} , S }$ denote the image of $R^{\psi\e^{-1}}_{\overline{r},S}$ in $R^{\Box,\psi\e^{-1},\loc}_{\overline{r},S}.$

By \cite[Lem.~4.11]{DDT}, there is a finite place $w_1\notin S$ with the following properties:
\begin{itemize}
\item $q_{w_1}\not\equiv 1 \pmod{p}$,

\item the ratio of the eigenvalues of $\overline{r}(\Frob_{w_1})$ is not equal to $q_{w_1}^{\pm 1}$,

\item the residue characteristic of $w_1$ is sufficiently large such that for any nontrivial root of unity $\zeta$ in a quadratic extension of $F,$ $w_1$ does not divide $\zeta+\zeta^{-1}-2.$
\end{itemize}

Let $U = \prod_w U_w \subset (B \otimes_F \A_{F,f})^{\times}$ be a compact open subgroup satisfying
\begin{itemize}
\item $U_w = \cO_{B_w}^{\times}$ for $w \notin S \cup \{w_1\},$
\item $U_{w_1}$ is contained is the subgroup of $(\cO_B)_{w_1}^{\times} = \GL_2(\cO_{F_{w_1}})$ consisting of matrices that are upper-triangular and unipotent modulo $\varpi_{w_1},$
\item For places over $p,$ $U_w = 1 +\varpi_w M_2(\cO_{F_w})$ if  $w |p,$ $w\neq v;$ $U_v$ is the subgroup $U^1_{B_v}$ defined in (\ref{eqn-U^i_D}). 
\end{itemize} By the choice of $U_{w_1},$ $U$ is sufficiently small in the sense of \cite[\S 3.3]{CHT}.

\cite{CEGGPS1} and \cite{Scholze} extend the Taylor-Wiles-Kisin method to construct the big patched modules. The detailed construction for Shimura curves in the minimal case is given in \cite[\S 6]{Dotto-Le}. By the arguments of \cite{Dotto-Le}, replacing $K^v$ in {\em loc.~cit.~}by $U^v,$ the representation $V = \bigotimes_{w \in S, w\neq v} V_w$ of $K^v$ in {\em loc.~cit.~}by the representation $\s^v_p$ of $U^v,$ forgetting the Hecke operators $T_w$ at places $w\in S',$ and allowing $B$ possibly ramifies at some places above $p,$ the same patching arguments produce a ``big'' patched module $M_{\infty}^B$ with the following data. (Let $j:=4 \# S - 1 $ and let $g,q$ be positive integers such that $q = g+ [F:\Q] - \# S +1.$)
\begin{itemize}
\item[$\bullet$] A formal power series ring in $q$-variables $\cO[\![z_1,\ldots,z_q]\!]$ with a homomorphism
\[
\cO[\![z_1,\ldots,z_q]\!] \to R^{\psi\e^{-1},\loc}_{\overline{ r}, S}
\]
which extends to a homomorphism from $S_{\infty}:= \cO[\![z_1,\ldots,z_q,y_1,\ldots,y_j]\!]$ to $R^{\Box , \psi\e^{-1},\loc}_{\overline{ r}, S}$.

\item[$\bullet$] There is a surjective homomorphism
\[
R_{\infty}^{ \psi\e^{-1}} \onto R^{\Box, \psi\e^{-1},\loc}_{\overline{r}, S},
\]
where $R_{\infty}^{ \psi\e^{-1}}:= R^{ \loc}[\![x_1,\ldots,x_g]\!].$ Let $\frak{m}_{\infty}$ be the maximal ideal of $R_{\infty}^{ \psi\e^{-1}}$.

\item[$\bullet$] An $\OC$-algebra homomorphism $S_{\infty}\to R_{\infty}^{ \psi\e^{-1}}$ such that
\[
R_{\infty}^{ \psi\e^{-1}}/\frak{a}_{\infty}\cong R^{\psi\e^{-1},\loc}_{\overline{r}, S },
\]
where $\frak{a}_{\infty}$ denotes the ideal $(z_1,\dots,z_{q},y_1,\ldots,y_j)$ of $S_{\infty}.$

\item[$\bullet$] A finitely generated Cohen-Macaulay $S_{\infty}[\![  \cO_{B_v}^{\times}  ]\!]$-module $M^{B}_{\infty}$ equipped with an action of $R_{\infty}^{ \psi\e^{-1}},$ so that the action of $S_{\infty}$ factors through $R_{\infty}^{ \psi\e^{-1}}.$ The module $M^B_{\infty}$ is also Cohen-Macaulay over $R_{\infty}^{ \psi\e^{-1}}[\![  \cO_{B_v}^{\times} ]\!]$ by \cite[Cor.~A29]{Gee-Newton}. Moreover, $M^{B}_{\infty}$ is projective in the category $\frak{C}_{\cO_{B_v}^{\times},\psi}(S_{\infty}).$ Note that projectivity in the case where $B$ ramifies at $v$ follows from the proof of \cite[Prop.~2.10]{CEGGPS1} using \cite[Prop.~5.6]{Newton}. Let $\frak{m}_{\overline{r}}$ be the maximal ideal of the abstract Hecke algebra associated to $\overline{r}$ as in \cite[\S5]{Scholze}. We have
\begin{equation} \label{eqn--M-infty}
M^{B}_{\infty} / \frak{a}_{\infty} = \left\{ {\begin{array}{ll}
\wt{H}^{0,B}_{\s_p^v, \psi} (U^{v}, \cO)^d_{\frak{m}_{\overline{r}}}  & \text{ if $B$ is definite} \\
\Hom_{\TM(U^v)_{\frak{m}_{\overline{r}}}[ G_{F}]} \left( r_{\frak{m}} , \wt{H}^{1,B}_{\s_p^v , \psi} (U^{v}, \cO)_{\frak{m}_{\overline{r}}} \right)^d  & \text{ if $B$ is indefinite,}\\
\end{array}}\right.
\end{equation}
where $\TM(U^v)_{\frak{m}_{\overline{r}}}$ denotes the Hecke algebra defined in the paragraph before \cite[Prop.~5.7]{Scholze} (by taking $\frak{p} = v$ and $\frak{m} = \frak{m}_{\overline{r}}$ in {\em loc.~cit.}), and $r_{\frak{m}}$ denotes the composite \[G_F \To{r^{\rm univ}} \GL_2(R^{\psi\e^{-1}}_{\overline{r},S}) \to \GL_2(\TM(U^v)_{\frak{m}_{\overline{r}}}).\]
\end{itemize}
 
\begin{remark}\label{rem:variant-M}
In the indefinite case, there is a variant of $M_{\infty}^B$ denoted by $N_{\infty}^B$, which is obtained by patching $\wt{H}^{1,B}_{\s_p^v , \psi} (U^{v}, \cO/\varpi^s)_{\frak{m}_{\overline{r}}}$ but without factorizing out $r^{\rm univ}$. Namely, we have
\[N_{\infty}^B/\fa_{\infty}\cong \big(\wt{H}^{1,B}_{\s_p^v , \psi} (U^{v}, \cO)_{\frak{m}_{\overline{r}}}\big)^d.\]
\end{remark}

Let ${\rm Mod}_{\cO^{\times}_{B_v},\psi}^{\rm fin}$ denote the category of finite $\cO$-modules with a continuous action of $  \cO_{B_v}^{\times}$ such that the $  \cO_{B_v}^{\times}$-action has central character $\psi|_{F_v^{\times}}.$ Define a functor $M_{\infty}^B(-)$ from ${\rm Mod}_{\cO^{\times}_{B_v},\psi}^{\rm fin}$ to the category of finitely generated $ R_{\infty}^{ \psi\e^{-1}}$-modules by letting
\begin{equation} \label{equ-def-functor}
M^{B}_{\infty}(\s):= \Hom_{  \cO_{B_v}^{\times} }^{\rm cont} ( M_{\infty }^{B}, \s^{\vee} )^{\vee} .
\end{equation}
By the projectivity of $M^{B}_{\infty}$ in $\frak{C}_{\cO_{B_v}^{\times},\psi}(\cO),$ $M^{B}_{\infty}(-)$ is an exact functor. Define
\begin{equation}\label{eqn:pi^B(r)}
\pi^B (\overline{ r}) := (M^{B}_{\infty}/ \frak{m}_{\infty})^{\vee} .
\end{equation}
By definition, we have
\begin{equation}\label{eqn:pi-and-M_infty}
(M^{B}_{\infty}(\s) / \frak{m}_{\infty})^{\vee} \cong \Hom_{  \cO_{B_v}^{\times} }(\s,\pi^B (\overline{ r})).
\end{equation}

At the place $v,$ let $\tau_v : I_{F_v}\to \GL_2(E)$ be an inertial type and $\mathbf{w} = (a_{\kappa} , b_{\kappa})_{\kappa: F_v \into E}$ be a Hodge type with $a_{\kappa} < b_{\kappa}$ for all $\kappa.$ Assume $\tau_v$ is a discrete series inertial type if $B$ ramifies at $v.$ Let $ R^{\psi\e^{-1}}_{v}(\mathbf{w}, \tau_v)$ (resp. $ R^{\psi\e^{-1}, {\rm cr}}_{v}(\mathbf{w}, \tau_v)$) denote the reduced $p$-torsion-free quotient of $R_v^{\psi\e^{-1}}$ which parametrizes potentially semistable (resp. potentially crystalline) liftings of $\overline{r}_v$ of Galois type $\tau_v$ and Hodge-Tate weights $\mathbf{w}.$ Following \cite{Gee-Geraghty} let $R^{\psi\e^{-1} , {\rm ds}}_{v}(\mathbf{w}, \tau_v)$ denote the maximal reduced $p$-torsion-free quotient of $ R^{\psi\e^{-1}}_{v}(\mathbf{w} ,\tau_v)$ which is supported on the irreducible components where the associated Weil-Deligne representation is generically of discrete series type. Let $ R_{\infty}^{ \psi\e^{-1}}(\mathbf{w},\tau_{v}):= R_{v}^{\psi\e^{-1} }(\mathbf{w},\tau_{v}) \widehat{\otimes}_{R_{v}^{\psi\e^{-1}}} R_{\infty}^{ \psi\e^{-1}},$ $ R_{\infty}^{\psi\e^{-1} ,{\rm cr}}(\mathbf{w}, \tau_{v}):= R_{v}^{\psi\e^{-1}, {\rm cr}}(\mathbf{w} ,\tau_{v}) \widehat{\otimes}_{R_{v}^{\psi\e^{-1}}} R_{\infty}^{ \psi\e^{-1}}$ and $ R_{\infty}^{ \psi\e^{-1}, {\rm ds}}(\mathbf{w},\tau_{v}):= R_{v}^{\psi\e^{-1}, {\rm ds}}(\mathbf{w}, \tau_{v}) \widehat{\otimes}_{R_{v}^{\psi\e^{-1}}} R_{\infty}^{ \psi\e^{-1}}.$
\begin{lemma}\label{lemma::ds-def-ring}
(i) If $\tau_v$ is a supercuspidal inertial type, then $R^{\psi\e^{-1},{\rm ds}}_{v}(\mathbf{w}, \tau_v) = R^{\psi\e^{-1} }_{v} (\mathbf{w},\tau_v)= R^{\psi\e^{-1} ,{\rm cr}}_{v}(\mathbf{w} , \tau_v ).$

(ii) If $\tau_v$ is a scalar type, then $R^{\psi\e^{-1},{\rm ds}}_{v} (\mathbf{w},  \tau_v)$ corresponds to the closure of potentially semistable but not potentially crystalline points in $\Spec  R^{\psi\e^{-1} }_{v}(\mathbf{w},\tau_v ).$
\end{lemma}
\begin{proof}
See \cite[\S5]{Gee-Geraghty}.
\end{proof}

We assume $B$ ramifies at $v$ and $F_v= \Q_p$ for the rest of this section. Let $\tau_v$ be supercuspidal and $\mathbf{w} = (a,b)$ be as above satisfying 
\begin{equation}\label{eq::Galois type}
 \varepsilon^{b+a-1}|_{I_{F_v}}\det(\tau_v) |_{I_{F_v}} \sim \psi|_{I_{F_v}}.
\end{equation}
We have a natural action of $B_{v}^{\times}$ on $\Sym^{b - a -1} E^2 \otimes {\det}^{a}$ as follows: we fix an embedding $B_{v}^{\times} \into \GL_2(\Q_{p^2}).$ Then $B_v^{\times}$ acts by the composite $B_{v}^{\times} \into \GL_2(\Q_{p^2}) \into \GL_2(E).$ Let $\Theta$ be any $\cO_{B_{v}}^{\times}$-stable $\cO$-lattice in 
\[
\s_{B_v}(\mathbf{w},\tau):=\s_{B_v}(\tau_{v}) \otimes \Sym^{b - a -1} E^2 \otimes {\det}^{a}.
\]
The homomorphism $R_{\infty}^{\psi \e^{-1}} \to \End(M^B_{\infty}(\Theta))$ factors through $  R_{\infty}^{ \psi\e^{-1}, {\rm ds}}(\mathbf{w},\tau_{v}),$ which is $R_{\infty}^{\psi\e^{-1}}(\mathbf{w}, \tau_{v} )$ by 
Lemma \ref{lemma::ds-def-ring}(i),  by the global Jacquet-Langlands correspondence and local-global compatibility. Since $S_{\infty}$ and $R_{\infty}^{ \psi\e^{-1}}(\mathbf{w}, \tau_{v})$ have the same Krull dimension, $M^B_{\infty}(\Theta)$ is maximal Cohen-Macaulay over $ R_{\infty}^{\psi\e^{-1}}(\mathbf{w}, \tau_{v} )$ by the same argument of the proof of \cite[Lem. 4.18]{CEGGPS1}.

Let $\delta:F^\times\setminus \AM_{F,f}^\times\to \cO^\times$ be a continuous character trivial on $U^{v}\cap  \AM_{F,f}^\times$ and trivial mod $\varpi$. Sending a lifting $r_w$ of $\overline{r}_w$ with determinant $\varepsilon^{-1}\psi|_{F^\times_w}$ to $ r_w\otimes \delta|_{F^{\times}_w}$ gives an isomorphism $\mathrm{tw}_{\delta|_{F_w^{\times}}}: R^{\psi \delta^2\varepsilon^{-1}}_w \simto R_w^{\psi \varepsilon^{-1}}.$ We hence have an isomorphism $\mathrm{tw}_{\delta}:= \otimes_w \mathrm{tw}_{\delta|_{F_w^{\times}}}:  R_{\infty}^{\psi \delta^2 \e^{-1}} \simto R_{\infty}^{\psi\e^{-1}}.$ We have the following analogue of Corollary \ref{cor::support}. 

\begin{lemma}\label{lem::isom-twist}
Let $(\mathbf{w},\tau_v)$ be as above satisfying $\varepsilon^{b+a-1}|_{I_{F_v}}\det(\tau_v) |_{I_{F_v}} \sim (\psi \delta^2)|_{I_{F_v}}.$ Let $\Theta$ be any $\cO_{B_{v}}^{\times}$-stable $\cO$-lattice in $\s_{B_v}(\mathbf{w},\tau_v) \otimes (\delta|_{F_v^{\times}}\circ {\rm Nrd})^{-1}.$ Then $ R_{\infty}^{\psi \e^{-1}}/\Ann_{R_{\infty}^{\psi \e^{-1}}}(M^B_{\infty}(\Theta))$ is equal to $\mathrm{tw}_{\delta} (R_{\infty}^{\psi\delta^2\e^{-1}}(\mathbf{w}, \tau_{v} ))$ if $R_{\infty}^{\psi\delta^2\e^{-1}}(\mathbf{w}, \tau_{v} )$ is an integral domain.
\end{lemma}
\begin{proof}
Note that $M_{\infty}^{B,\delta^{-1}}: = M_{\infty}^B \otimes \delta^{-1} \circ {\rm Nrd}$ is a ``big'' patched module, which is finitely generated and Cohen-Macaulay over $S_{\infty}[\![  \cO_{B_v}^{\times}  ]\!]$ equipped with a compatible action of $R_{\infty}^{ \psi \delta^2\e^{-1}}.$ This is because $\delta$ is trivial mod $\varpi,$ and when $B$ is definite, the map $f\mapsto [g\mapsto f(g)(\delta\circ{\rm Nrd})^{-1}(g)]$ induces an isomorphism 
\begin{equation}\label{eqn::twist-AF}
\wt{H}^{0,B}_{\s_p^v , \psi\delta^2} (U(N)^{v}, \cO) \cong \wt{H}^{0,B}_{\s_p^v , \psi} (U(N)^{v}, \cO)\otimes \delta\circ{\rm Nrd}  ,
\end{equation}
which is compatible with the action of the Hecke algebra. Here $U(N)^v$ denotes the group $U(N)^{\frak{p}}$ in the proof of \cite[Thm. 8.10]{DPS-Crelle}. The isomorphism similar to \eqref{eqn::twist-AF} in the indefinite case  follows from the proof of \cite[Lem. 2.3]{BDJ}. Since $\Theta\otimes \delta|_{F_v^{\times}}\circ {\rm Nrd}$ is an $\cO_{B_v}^{\times}$-stable lattice in the locally algebraic representation $ \s_{B_v}(\mathbf{w},\tau_v),$ the action of  $R_{\infty}^{ \psi \delta^2\e^{-1}}$ on $M_{\infty}^{B,\delta^{-1}}(\Theta\otimes \delta|_{F_v^{\times}}\circ {\rm Nrd})$ factors through $R_{\infty}^{\psi\delta^2\e^{-1}}(\mathbf{w}, \tau_{v} ).$ Then the  $R_{\infty}^{\psi\delta^2\e^{-1}}(\mathbf{w}, \tau_{v} )$-module
\[
M_{\infty}^{B,\delta^{-1}}(\Theta\otimes \delta|_{F_v^{\times}}\circ {\rm Nrd})  = M_{\infty}^B (\Theta ) \otimes \delta^{-1}|_{F_v^{\times}}\circ {\rm Nrd}
\]
is supported on a union of irreducible components of $\Spec(R_{\infty}^{\psi\delta^2\e^{-1}}(\mathbf{w}, \tau_{v} )),$ which is all of $\Spec(R_{\infty}^{\psi\delta^2\e^{-1}}(\mathbf{w}, \tau_{v} ))$ if $R_{\infty}^{\psi\delta^2\e^{-1}}(\mathbf{w}, \tau_{v} )$ is an integral domain. The construction of the ``big'' patched module and the relation with the associated Galois representation imply that the quotient ring  $ R_{\infty}^{\psi \e^{-1}}/\Ann_{R_{\infty}^{\psi \e^{-1}}}(M^B_{\infty}(\Theta))$ is equal to $\mathrm{tw}_{\delta} (R_{\infty}^{\psi\delta^2\e^{-1}}(\mathbf{w}, \tau_{v} )).$
\end{proof}

\section{The Gelfand-Kirillov dimension of $ \pi^B (\overline{r}) $} \label{sec:GK}

In this section we maintain the assumptions made in \S \ref{section-autom-forms}. In particular, $B$ and $\overline{r}$ satisfy the compatibility condition (H0) of \cite{BreuilDiamond}, which implies that $\pi^B(\overline{r})\neq 0$ by \cite[Corollaire~3.2.3]{BreuilDiamond}. Since our main applications are for the quaternion algebra over $\Q_p,$ we assume further that  $F_v \cong \Q_p,$ where $v $ is the unique place over $p$ at which $B$ is ramified. We denote by $D:= B_v$ the quaternion algebra over $\Q_p.$ We prove our main results on the Gelfand-Kirillov dimension of $\pi^B (\overline{r})$ which is defined by (\ref{eqn:pi^B(r)}). Assume $p\geq 5$. 

\subsection{Serre weights for quaternion algebras}

Let $W_B(\overline{r})$ denote the set of modular quaternionic Serre weights at $v$ defined in \cite[\S3.1]{BreuilDiamond}. Recall that an irreducible smooth representation of $\cO_{D}^{\times} $ over $\F$, equivalently a smooth character $\chi:\cO_D^{\times}\ra\F^{\times}$, is in $W_B(\overline{r})$ if 
\[
\Hom_{\cO_{D}^{\times}} (\chi, \pi^B(\overline{r})) \neq 0,
\]
equivalently $M_{\infty}^B(\chi) \neq 0$ by (\ref{eqn:pi-and-M_infty}). Moreover $
\dim_{\F} \Hom_{\cO_D^{\times}} (\chi , \pi^B(\overline{r})) =  \dim_{\F} M_{\infty}^B (\chi) / \frak{m}_{\infty}.$

Let $\brho: = \overline{r}_v(1).$ Note that by our assumption $\brho$ is a two dimensional continuous representation of $G_{\Q_p}.$ We recall the definition of $W_{D}(\brho)$, the set of predicted quaternionic Serre weights for $\brho,$ which is denoted by $W^?(\brho)$ in \cite[Def.~3.4]{Gee-Savitt}. A character $\psi: \cO_D^{\times}\onto \F_{p^2}^{\times}\to \F^{\times}$ is in $W_D(\brho)$ if and only if $\brho$ has a potentially Barsotti-Tate lift of type $[\psi] \oplus [\psi]^{p}$ if $\psi\neq \psi^p,$ and  $\brho$ has a potentially semistable lift of Hodge-Tate weights $(0,1)$ and  type $[\psi] \oplus [\psi]$ which is not potentially crystalline if $\psi = \psi^p.$

We have the following description of the set $W_D(\brho)$.

\begin{proposition}\label{thm:Serre-D}
Recall $\xi:\F_{p^2}^{\times}\into \F^{\times}$ the character introduced in \S \ref{sec::lattices-in-tame-types}. Let $\zeta$ denote the character $\xi^{p+1},$ and let $\a $ denote the character $\xi^{p-1}.$

\begin{enumerate}
\item[(i)] Assume $\brho$ is in \ref{Case1} of \S\ref{section-GL2-Serrewt}. We have
\begin{enumerate}
\item[(i-a)] If $r \neq 0,p-1$ then $\chi \in W_D(\brho)$ if and only if $\chi \in  \{ \xi^r \zeta^{s+1},~ \xi^{pr} \zeta^{s+1}, ~ \xi^{r}\a^{-1} \zeta^{s+1} ,~\xi^{pr} \a \zeta^{s+1} \}.$

\item[(i-b)] If $r = 0$ or $p-1,$ then $\chi \in W_D(\brho)$ if and only if $\chi \in \{ \a^{-1} \zeta^{s+1} ,~ \a \zeta^{s+1} \}.$
\end{enumerate}
\end{enumerate}

\begin{enumerate}
\item[(ii)] Assume $\brho$ is in \ref{Case2} of \S\ref{section-GL2-Serrewt}. We have
\begin{enumerate}

\item[(ii-a)] If $r=0,$ ${\rm unr}_1 = {\rm unr}_2$ and $\brho$ is tr\`es ramifi\'e, then $\chi \in W_D(\brho)$ if and only if $\chi = \zeta^{s+1} .$

\item[(ii-b)] If $r=0,$ ${\rm unr}_1 = {\rm unr}_2$ and $\brho$ is peu ramifi\'e, then $\chi \in W_D(\brho)$ if and only if  $\chi \in \{ \zeta^{s+1}  ,~\a^{-1} \zeta^{s+1} ,~ \a \zeta^{s+1} \}.$

\item[(ii-c)] For other $\brho,$ $\chi \in W_D(\brho)$ if and only if  $\chi \in \{ \xi^{r} \a^{-1} \zeta^{s+1} ,~ \xi^{pr}\a \zeta^{s+1}  \}.$
\end{enumerate}
\end{enumerate}

\begin{enumerate}
\item[(iii)] Assume $\brho$ is in \ref{Case3} of \S\ref{section-GL2-Serrewt}. We have
\begin{enumerate}
\item[(iii-a)]  If $r= 0$ and ${\rm unr}_1 = {\rm unr}_2,$ then $\chi \in W_D(\brho)$ if and only if $\chi \in \{\zeta^{s+1} , \a^{-1} \zeta^{s+1} , \a \zeta^{s+1} \}.$

\item[(iii-b)] For other $\brho,$ $\chi \in W_D(\brho)$ if and only if  $\chi \in \{ \xi^{r} \a^{-1} \zeta^{s+1} ,~ \xi^{pr}\a \zeta^{s+1}  \}.$
\end{enumerate}
\end{enumerate}
\end{proposition}
\begin{proof}
This follows from the definition of $W_D(\brho)$. More precisely, the Breuil-M\'ezard conjecture (\cite{Breuil-Mezard}), proved in \cite{Kisin-FM}, \cite{Paskunas-BM}, \cite{HT}, \cite{Sander-BM}, tells exactly when the involved deformation rings are nonzero in terms of $W(\brho)$ (cf.~\S\ref{section-GL2-Serrewt}).   We take the case (ii-b) as an example, so that $W(\brho)=\{\sigma_{0,s+1},\sigma_{p-1,s+1}\}$.  Up to twist we may assume $s=0$.  It follows from the Breuil-M\'ezard conjecture and Proposition \ref{prop:Diamond}(ii) that $\brho$ has a potentially Barsotti-Tate lift of type $[\xi^2]\oplus [\xi^{2p}]$, so we obtain $\alpha^{-1}\zeta,\alpha\zeta\in W_D(\brho)$. On the other hand, $\brho$ has a potentially semistable lift of Hodge-Tate weights $(0,1)$ and  type $[\zeta]\oplus [\zeta]$ which is not potentially crystalline (see \cite[Th\'eor\`eme 1.2]{Breuil-Mezard}), which gives $\zeta\in W_D(\brho)$. Conversely, using the Breuil-M\'ezard conjecture again one checks  that these exhaust all the Serre weights in $W_{D}(\brho)$. 
\end{proof}

\begin{proposition}\label{lemma--quat-Serre-wt}
We have $W_B(\overline{r}) = W_D(\brho).$
\end{proposition}
\begin{proof}
The inclusion $W_B(\overline{r}) \subseteq W_D(\brho)$ follows from \cite[Lem.~3.3]{Gee-Savitt}. Note that \cite{Gee-Savitt} works only with definite $B$ which ramifies at all places above $p,$ but the argument also works in our case. By \cite[Thm.~8.3]{Gee-Savitt}, the two sets  $W_B(\overline{r})$ and $W_D(\brho)$ are identical in most cases with exception possibly when $\brho$ is an unramified twist of $\bigl(\begin{smallmatrix}
 \omega & * \\ 0& 1
\end{smallmatrix} \bigr) \otimes \omega^{s+1}$ and $\chi = \zeta^{s+1} \in W_D(\brho).$ In this exceptional case,  $\brho$ has a potentially semistable lift of type $[\chi] \oplus [\chi]$ which is not potentially crystalline. Applying \cite[Th\'eor\`eme~3.2.2]{BreuilDiamond} (by taking $[r_v,N_v] = [[\chi] \oplus [\chi], N_v\neq 0]$ in {\em loc. cit.}), there is a Hilbert modular form over $F$ of parallel weight $(2,\ldots,2)$ special at $v$ which gives $\overline{r}.$ By global Jacquet-Langlands correspondence, as in the proof of \cite[Lem.~3.3]{Gee-Savitt}, we have $\chi \in W_B(\overline{r}).$
\end{proof}

\begin{remark}\label{rem::Khare}
The question of determining the quaternionic Serre weights is first studied by Khare \cite{Khare}. More precisely, \cite[Thm.~7]{Khare} proves that if $B_0$ denotes the definite quaternion algebra over $\Q$ which is ramified exactly at $p$ and $\infty,$ then $W_{B_0}(\overline{r}) = W_D(\brho).$ 
\end{remark}

\subsection{Lattices in some locally algebraic representations of $\cO_D^{\times}$}

Let $\chi$ be any character of $\cO_D^{\times}$ over $\F.$     Recall that $W_{\chi,n}$ denotes $ (\Proj_{ \F[\![\cO_D^{\times}/Z^1_D]\!]}\chi) / \frak{m}_{D}^3$ for $n\geq 1$, where $\frak{m}_{D}$ denotes the maximal ideal of the Iwasawa algebra $\F[\![U^1_D / Z^1_D]\!]$. We construct suitable lattices $\mathcal{L}$ in  locally algebraic representations of $\cO_{D}^{\times}$ over $E$ so that $\mathcal{L} / p\mathcal{L}$ is a quotient of $W_{\chi, 3} =  (\Proj_{ \F[\![\cO_D^{\times}/Z^1_D]\!]}\chi) / \frak{m}_{D}^3.$ The construction of these lattices is much easier than the case considered in \S \ref{section-lattices-GL2}. 

Recall that $\cO_D^{\times}$ embeds into $\GL_2(\Z_{p^2})$ and then embeds into $\GL_2(\cO)$ via the embedding $\GL_2(\Z_{p^2}) \subset \GL_2(\cO)$. An explicit embedding is given by (cf.~\eqref{eq:division-D})
\[
\varpi_D\mapsto \matr{0}1p0,\ \ \ a\mapsto \matr{a}00{\sigma(a)},\ \ a\in\Q_{p^2}.
\]
Let $\cO_D^{\times}$ act on $\Sym^1\cO^2$ and $\Sym^1 E^2$ via the above embedding. Precisely, for $a,b\in\Z_{p^2}$,
\begin{equation}\label{eq:action-Sym1}(a+\varpi_Db)\cdot X=aX+p\sigma(b)Y,\ \ \ (a+\varpi_Db)\cdot Y=bX+\sigma(a)Y.\end{equation}
Equipped with this action, $\Sym^1 \cO^2$ and $\Sym^1 E^2$ are continuous representations of $\cO_D^{\times}.$
Let  $\mathrm{pr}:\Q_p^{\times}\ra 1+p\Z_p$ denotes the projection sending $p$ to $1$. Let $\underline{\Sym}^1 \cO^2$ (resp. $\underline{\Sym}^1 E^2$) denote the continuous $\cO_D^{\times}$-module $\Sym^1\cO^2\otimes (\mathrm{pr}\circ{\rm Nrd}_D)^{-1/2}$ (resp. $\Sym^1 E^2\otimes (\mathrm{pr}\circ{\rm Nrd}_D)^{-1/2}$), where ${\rm Nrd}_D:D^{\times} \to \Q_p^{\times}$ is the reduced norm. One checks that $Z_D^1$ acts trivially on $\underline{\Sym}^1 E^2.$ Note that $(\underline{\Sym}^1 \cO^2) /p \cong\Sym^1 \F^2$ with semisimplification $(\Sym^1 \F^2)^{\rm ss} = \chi_1 \oplus \chi_2,$ where $\chi_1,\chi_2$ are characters of $\cO_D^{\times}$  determined by $\chi_1 (t) = t,$ $\chi_2 (t) = t^p$ for all $t \in \F_{p^2}^{\times}.$ In particular $\chi_1=\chi_2\alpha^{-1}$.  By Proposition \ref{prop-Ext1-U1} we have
\[\dim_{\F}\Ext^1_{\cO_D^{\times}/Z_D^1}(\chi_1,\chi_2)=\dim_{\F}\Ext^1_{\cO_D^{\times}/Z_D^1}(\chi_2,\chi_1)=1,\]
  so there exist (up to isomorphism) unique nonsplit extensions $(\chi_1\ligne\chi_2)$ and $(\chi_2\ligne\chi_1)$.

\begin{lemma}\label{lem:Sym1-D}
There exist $\cO_D^{\times}$-stable $\cO$-lattices $L,L'$ in $\underline{\Sym}^1 E^2$ such that

(a) $pL\subset L'\subset L$;

(b) $L/pL\cong (\chi_1\ligne\chi_2)$ and $L'/pL'\cong (\chi_2\ligne \chi_1)$.
\end{lemma}
\begin{proof}
We take $L=\underline{\Sym}^1\cO^2=\cO( Y \otimes 1)\oplus \cO (X\otimes 1)$ and $L'=\cO (X\otimes 1)\oplus p\cO (Y\otimes 1)$. The properties are easily checked using \eqref{eq:action-Sym1}.
\end{proof}

Let $\chi : \cO_D^{\times} \to \F^{\times}$ be a character. Then there exist integers $-2 \leq a \leq p-2,$ $b\in \Z$ such that
\[
[\chi] = [\xi]^{a+2 + (p+1)b},
\]
where $[-]$ denotes the Teichm\"uller lift. We write
\begin{equation}\label{eqn:psi_i}
\psi_1 := [\chi] = [\xi]^{a+2 + (p+1)b},~ \psi_{2} := [\xi]^{a+3 + (p+1)(b-1)},~ \psi_{3} := [\xi]^{a+1 + (p+1)b}.
\end{equation}
Let $\Theta_1: = \psi_1$, viewed as an $\cO_D^{\times}$-stable lattice in $V_1 := \psi_1 \otimes_{\cO} E.$ For $i =2,3,$ let 
\begin{equation}\label{eqn::twist-on-D}
V_i: = \underline{\Sym}^1 E^2 \otimes \psi_i.
\end{equation}
Note that $Z_D\cap \cO_D^{\times}$ acts on $V_i$ by the same character $[\chi]$. The $\cO_D^{\times}$-representations $V_i,~1\leq i\leq 3$ are irreducible and
\[
\overline{V_1} = \chi, \quad \overline{V_2}^{\rm ss} = \chi \oplus \chi \a^{-1},\quad \overline{V_3}^{\rm ss} = \chi \oplus \chi \a.
\]
An analogue of Proposition \ref{prop-lattice-EGS} implies that  there exists a unique (up to homothety)     $\cO_D^{\times}$-stable $\cO$-lattice in $V_i$, say $\Theta_i$, such that $\mathrm{cosoc}_{\cO_D^{\times}}(\Theta_i / p \Theta_i) = \chi$ for $i =2,3.$  We have surjective maps
\begin{align*}
r_1 &: \Theta_1 \onto \Theta_1 / p\Theta_1 \cong  \chi, \\
r_i & : \Theta_i \onto \Theta_i / p\Theta_i \onto \rcosoc(\Theta_i / p\Theta_i)  \cong  \chi,~~ i=2,3.
\end{align*}
Let $\Theta_i':=\Ker(r_i)$ for $i=2,3$. Then by Lemma \ref{lem:Sym1-D} we have
\[\Theta_2'/p\Theta_2'\cong (\chi\ligne \chi\alpha^{-1}),\quad \Theta_3'/p\Theta_3'\cong (\chi\ligne \chi\alpha).\]
 Since every irreducible representation of $\cO_D^{\times}$ over $\F$ is $1$-dimensional, $\Theta_1 / p\Theta_1$ is killed by $\frak{m}_{D},$ while $\Theta_i / p\Theta_i$ and $\Theta'_i / p\Theta'_i$ are killed by $\frak{m}^2_{D}$ for $i=2,3.$ By construction, $\Theta_1,$ $\Theta_2$ and $\Theta_3$ are quotients of $\Proj_{ \cO[\![\cO_D^{\times}/Z^1_D]\!]}\chi.$

We now glue the three lattices $\Theta_1,$ $\Theta_2$ and $\Theta_3.$ We first glue $\Theta_1$ and $\Theta_2$ along $\chi$, namely define $\Theta$ by the short exact sequence
\begin{equation}\label{equation-Dlattice-ses1} 0\ra \Theta\ra \Theta_1\oplus \Theta_2\To{r_1-r_2} \chi\ra0.\end{equation}
 
\begin{proposition}\label{prop-lattice-Theta}
(i) There is a short exact sequence $0\ra \Theta_2'/p\Theta_2'\ra \Theta/p\Theta\ra \chi\ra0$.

(ii) The cosocle of $\Theta/p\Theta$ is isomorphic to $\chi$.
Moreover, the  cosocle filtration of $\Theta/p\Theta$ is
\[\chi\ligne \chi\alpha^{-1}\ligne \chi.\]
\end{proposition}
\begin{proof}
Clearly, Lemmas \ref{lem:glue-L/p} and Lemma \ref{lem:cosoc-glue} remain true if we are considering $\cO_D^{\times}$-representations instead of $\GL_2(\Z_p)$-representations.  The results follow from them.
\end{proof}

Let $r$ denote the map $\Theta \twoheadrightarrow \Theta/p\Theta \twoheadrightarrow \chi$
where the second map is as in  Proposition \ref{prop-lattice-Theta}(i).
Denote by $\wt{\Theta}$ the lattice in  $V_1 \oplus V_2 \oplus V_3$ obtained by gluing $\Theta$ and $\Theta_3$ along $\chi$. Namely, $\wt{\Theta}$ is defined by the following short exact sequence
\begin{equation}\label{equation-Dlattice-ses2}0\ra \wt{\Theta}\ra \Theta\oplus \Theta_3\To{r-r_3} \chi\ra0.\end{equation}

\begin{proposition}\label{prop-lattice-wtTheta}
(i) The cosocle of $\wt{\Theta}/p\wt{\Theta}$ is isomorphic to $\chi$.

(ii) $\wt{\Theta}/p \wt{\Theta}$ is a quotient of $W_{\chi,3}.$ More precisely, $\wt{\Theta}/p\wt{\Theta}$ is isomorphic to  $\overline{W}_{\chi,3}:=W_{\chi,3}/(\chi\alpha^2\oplus \chi\alpha^{-2})$.
\end{proposition}
\begin{proof}
(i)  Note that the cosocle of $\Ker(r)$ is   $\chi\oplus \chi\alpha^{-1}$, while that of $\Ker(r_3)$ is $\chi\alpha$, so the result follows from Lemma \ref{lem:cosoc-glue}.

(ii)  It follows from Lemma \ref{lem:glue-L/p} that
there are short exact sequences
\[0\ra \Ker(r)/p\Ker(r)\ra \wt{\Theta}/p\wt{\Theta}\ra \Theta_3/p\Theta_3\ra0\]
\[0\ra \Theta_3'/p\Theta_3'\ra \wt{\Theta}/p\wt{\Theta}\ra \Theta/p\Theta\ra0.\]
Using Proposition \ref{prop-lattice-Theta}(ii), we deduce that  $\wt{\Theta}/p\wt{\Theta}$ admits both  the nonsplit extensions $(\chi\alpha^{-1}\ligne\chi)$ and $(\chi\alpha\ligne\chi) $ as quotients. Combined with (i), this implies that $\wt{\Theta}/p\wt{\Theta}$ admits a quotient isomorphic to $W_{\chi,2}$; let $\Ker$ be the corresponding kernel.  Comparing the Jordan-H\"older factors, we have $(\Ker)^{\rm ss}\cong \chi\oplus \chi$. However, we know $\Ext^1_{\cO_D^{\times}/Z_D^1}(\chi,\chi)=0$ by Proposition \ref{prop-Ext1-U1}, hence  $\Ker \cong \chi\oplus \chi.$
In particular, $\wt{\Theta}/p\wt{\Theta}$ is killed by $\fm_{D}^3$.  The last statement is a consequence of Corollary \ref{cor:structure-W_chi,3}.
\end{proof}

\subsection{The Gelfand-Kirillov dimension} \label{section:GK-dim}

Assume $\brho := \overline{r}_v(1)$ is of the form \ref{C1} or \ref{C2} in \S\ref{section-def-ring}. Recall that for any character $\chi: \cO_D^{\times} \to \F^{\times} $, we have constructed $\mathcal{L} \in \{\Theta_1, \Theta_2, \Theta_3, \Theta, \wt{\Theta}\}$ such that $\rcosoc (\mathcal{L} / p\mathcal{L}) = \chi.$ The construction depends on the choice of $(a,b)$ in (\ref{eqn:psi_i}). From now on, we assume $\chi \in W_D(\brho),$ and  make our choice of $(a,b)$ as follows:
\[
\left\{ {\begin{array}{ll}
(a,b) = (r,s)  & \text{ if $\chi =  \xi^{r}\a^{-1} \zeta^{s+1}$;} \\
(a,b) = (p-3-r,r+s+1)  & \text{ if $\chi =  \xi^{pr} \a \zeta^{s+1} $;} \\
(a,b) = (r - 2,s+1)  & \text{ if $\chi =  \xi^{r} \zeta^{s+1}$;} \\
(a,b) = (p-1-r,r+s)  & \text{ if $\chi =  \xi^{pr} \zeta^{s+1}$.} \\
\end{array}}\right.
\]
Let $\psi_i$ be given by (\ref{eqn:psi_i}) for $i=0,1,2$. Then one may check directly that $\psi_i \neq \psi_i^p.$ Let $\tau_i$ be a tame supercuspidal inertial type so that $\s(\tau_i) = \Theta(\psi_i).$ Let $\psi:F^\times\setminus \AM_{F,f}^\times\to \cO^\times$ be a continuous character as in \S \ref{section-autom-forms} satisfying $\psi|_{Z_D\cap \cO_D^{\times }} = \psi_1|_{Z_D\cap \cO_D^{\times }}.$ For $\mathcal{L} \in \{\Theta_1, \Theta_2, \Theta_3, \Theta, \wt{\Theta}\},$ let $I_{\mathcal{L} }: = \Ann_{R_{\infty}^{ \psi\e^{-1}}}(M_{\infty}^B(\mathcal{L}))$ denote the annihilator of $M_{\infty}^B(\mathcal{L})$ in $R_{\infty}^{\psi\e^{-1}}.$ 

Let $R$ be any commutative ring and $M$ be an $R$-module. Following \cite[\S 8.2]{BHHMS1} we say $M$ is free of rank $m$ over its scheme-theoretic support if it is isomorphic to $(R/ \Ann_R(M))^m.$ 

\begin{proposition}\label{prop--freeness}
Assume $\chi \in W_B(\overline{r}).$ Let $m := \dim_{\F} \Hom_{\cO_D^{\times}} (\chi , \pi^B(\overline{r})).$ Then the $R_{\infty}^{ \psi\e^{-1}}$-module $M_{\infty}^B(\Theta_1)$ (resp. $M_{\infty}^B(\Theta_2),$ resp. $M_{\infty}^B(\Theta_3)$) is free of rank $m$ over its scheme-theoretic support. In particular, $I_{\Theta_1} = I_{R_1} R_{\infty}^{ \psi\e^{-1}},$ $I_{\Theta_2} = I_{R_2} R_{\infty}^{ \psi\e^{-1}}$ and $I_{\Theta_3} = I_{R_3} R_{\infty}^{ \psi\e^{-1}},$ where $I_{R_1},$ $I_{R_2}$ and $I_{R_3} $ are given in (\ref{eqn::I_R}).
\end{proposition}
\begin{proof}
Twisting by the cyclotomic character and taking into account the framed variables, we have an isomorphism
\[
R_v^{\psi\e^{-1}} (  (-1,0),\tau_1(-1)) \cong R_{\brho}^{\psi\e}  ((0,1),\tau_1)[\![ X_1,X_2, X_3]\!],
\]
 where $R_{\brho}^{\psi\e}(  (0,1),\tau_1)$ is a regular local ring by Proposition \ref{thm-regular-HT01}. Hence $R_{\infty}^{\psi\e^{-1}}( (-1,0),\tau_1(-1)),$ as a formal power series ring over $ R_v^{\psi\e^{-1}}( (-1,0),\tau_1(-1)),$ is also a regular local ring. Since $M_{\infty}^B(\Theta_1)$ is finite maximal Cohen-Macaulay over  $R_{\infty}^{ \psi\e^{-1} }( (-1,0),\tau_1(-1))$, the Auslander-Buchsbaum formula   implies that $M_{\infty}^B (\Theta_1)$ is in fact finite free over $R_{\infty}^{\psi \e^{-1}}( (-1,0),\tau_1(-1))$ of rank $ m,$ and hence if $m \neq 0,$ $I_{\Theta_1} = I_{R_1} R^{\psi\e^{-1}}_{\infty}.$

Assume $\mathcal{L} \in \{\Theta_2, \Theta_3\}.$ Let $\delta : \Q_p^{\times} \to \cO^{\times}$ be the character sending $x\in \Q_p^{\times}\mapsto \mathrm{pr}(x)^{1/2}\in 1+p\Z_p, $ viewed also as a character of $F^\times\setminus \AM_{F,f}^\times$. We also view $\delta$ as a character of $G_{\Q_p},$ via the local class field theory. Then $(\mathcal{L}\otimes \delta\circ{\rm Nrd})[1/p]$ is a locally algebraic representation of $\cO_{B_v}^{\times}$. By Theorem \ref{thm-regular-HT02} and Lemma \ref{lem::isom-twist}, $R_{\infty}^{ \psi\e^{-1} }/I_{\mathcal{L}}$ is a regular local ring. We show as above that $M_{\infty}^B (\mathcal{L})$ is finite free over $R_{\infty}^{ \psi\e^{-1} }/I_{\mathcal{L}}$ of some rank $n.$ It follows from Lemma \ref{lem::isom-twist} and Corollary \ref{cor::support} that $I_{\mathcal{L}}$ has the description as in the statement of the proposition. We are left to prove $n = m.$

If $\JH(\mathcal{L}/p \mathcal{L}) \cap  W_D(\brho) = \{\chi\},$ then
\[
M_{\infty}^B(\mathcal{L} / p\mathcal{L}) \simto M_{\infty}^B(\chi)
\]
which is free of rank $m$ over its scheme-theoretic support. Hence $n = m.$

Now assume both Jordan-H\"older factors of $\mathcal{L} / p \mathcal{L}$ are in $W_D(\brho) .$ By Proposition \ref{thm:Serre-D} this can only happen when $\brho$ is absolutely irreducible. We assume $\chi = \xi^{r}\a^{-1} \zeta^{s+1}$ and $\mathcal{L} = \Theta_3,$ the other cases can be handled in the same way. Since $\delta\equiv 1\pmod
{\varpi}$ and the following discussion only involves $\F$-representations, the twisting by $\delta$ will not change anything. 

Since
\[
\Theta_3 / p \Theta_3 = (\xi^{r} \zeta^{s+1} \ligne \xi^{r}\a^{-1} \zeta^{s+1}).
\]
Applying the patching functor $M_{\infty}^B(-),$ we obtain a short exact sequence
\begin{equation}\label{ses-section-quaternion0}
0 \to M_{\infty}^B(\xi^{r} \zeta^{s+1}) \to M_{\infty}^B(\Theta_3 / p \Theta_3) \to M_{\infty}^B(\xi^{r}\a^{-1} \zeta^{s+1}) \to 0,
\end{equation}
where all the modules in the sequence are finite free over their scheme-theoretic support. We must show the modules have the same rank. For this we use the knowledge on $\GL_2$-side to study their support.

According to Proposition \ref{prop-reduction-L}, there exists a $K$-stable $\cO$-lattice $L$ in $ \unSym^1 E^2\otimes\Theta(\psi_3)$ such that $L/pL$ is a nonsplit extension of $ (\s_{p-3-r, r+s+2}  \ligne \s_{r, s+1})$ by $(   \s_{p-1-r, r+s+1} \ligne \s_{r-2, s+2} ).$ Let $\s^{\circ}(\tau_1)$ denote the unique (up to homothety) $K$-stable $\cO$-lattice in $\Theta(\psi_1)$ so that $\s^{\circ}(\tau_1)/ p \s^{\circ}(\tau_1) = (\s_{p-3-r, r+s+2}  \ligne \s_{r, s+1}).$ Let $\tau$ be a tame supercuspidal inertial type so that there is a $K$-stable $\cO$-lattice $\s^{\circ}(\tau)$ of $\s(\tau)$ satisfying $\s^{\circ}(\tau) / p\s^{\circ}(\tau) =  ( \s_{p-1-r, r+s+1} \ligne  \s_{r-2, s+2}  ).$ Applying Pa\v{s}k\={u}nas' functor $M(-)$ in \S\ref{section-Morra}, we obtain a short exact sequence
\begin{equation}\label{ses-section-quaternion1}
0\to M(\s^{\circ}(\tau)/ p \s^{\circ}(\tau)) \to M(L/pL) \to M(\s^{\circ}(\tau_1)/ p \s^{\circ}(\tau_1)) \to 0.
\end{equation}
Note that the three $R_{\brho}^{\psi\e}$-modules in the above short exact sequence are all cyclic by Lemma \ref{lem:regular} and Remark \ref{rem:cyclic-L}. Then by Theorem \ref{prop::support} we obtain the following short exact sequence
\begin{equation}\label{ses-section-quaternion2}
0\to R_{\brho}^{\psi\e}( (0,1),\tau) \otimes_{\cO} \F \to R_{\brho}^{\psi\e}( (0,2),\tau_3) \otimes_{\cO} \F \to R_{\brho}^{\psi\e}( (0,1),\tau_1) \otimes_{\cO} \F \to 0.
\end{equation}
On the other hand, $\s_{r-2,s+2},\s_{p-3-r,r+s+2}\notin 
W(\brho)$ and the extension $(\s_{p-1-r, r+s+1}\ligne \s_{r, s+1})$ occurs in $L/pL$ by Proposition \ref{prop-reduction-L}. Let $\tau'$ denote a tame inertial type so that $\s(\tau')$ is isomorphic to the principal series tame type $I([x]^{s+1}, [x]^{r+s+1})$ defined in Proposition \ref{prop:Diamond}. Let $\s^{\circ}(\tau')$ be the unique (up to homothety) $K$-stable $\cO$-lattice in $ I([x]^{s+1}, [x]^{r+s+1})$ such that $\s^{\circ}(\tau')/ p \s^{\circ}(\tau') = (\s_{p-1-r, r+s+1}\ligne \s_{r, s+1}).$ Then the short exact sequence (\ref{ses-section-quaternion1}) can be identified with the following short exact sequence
\begin{equation}\label{ses-section-quaternion3}
0\to M(\s_{p-1-r, r+s+1}) \to M(\s^{\circ}(\tau')/p\s^{\circ}(\tau')) \to M(\s_{r, s+1}) \to 0.
\end{equation}
And the short exact sequence (\ref{ses-section-quaternion2}) becomes
\begin{equation}\label{ses-section-quaternion4}
0\to R_{\brho}^{\psi\e ,{\rm cr} }( (r+s+1,p+s+1), \ide) \otimes_{\cO} \F \to R_{\brho}^{\psi\e}((0,1),\tau') \otimes_{\cO} \F \to R_{\brho}^{\psi\e ,{\rm cr}}( (s+1,r+s+2),\ide) \otimes_{\cO} \F \to 0.
\end{equation}
By \cite[Thm.~7.2.1]{EGS} $R_{\brho}^{\psi\e}( (0,1),\tau') \otimes_{\cO} \F$ is isomorphic to a formal power series ring over $\F[\![ X,Y]\!]/(XY),$ and $R_{\brho}^{\psi\e ,{\rm cr}} ( (r+s+1,p+s+1),\ide)\otimes \F$ (resp. $R_{\brho}^{\psi\e, {\rm cr}} ( (s+1,r+s+2), \ide) \otimes_{\cO} \F $) is the quotient of $R_{\brho}^{\psi\e}( (0,1),\tau') \otimes_{\cO} \F$ by $X$ (resp. $Y$). Therefore $\Spec \left(R^{\psi \e}_{\brho} ((0,2) ,\tau_3) \otimes_{\cO} \F\right)$ has two irreducible components. By Lemma \ref{lem::isom-twist} and Corollary \ref{cor::support}, $ \Spec\left(( R_{\infty}^{ \psi\e^{-1} }/I_{\Theta_3}) \otimes_{\cO}\F\right) $ also has two irreducible components.

Back to the short exact sequence (\ref{ses-section-quaternion0}). By the discussion of the first paragraph of the proof, $M_{\infty}^B(\xi^{r} \zeta^{s+1})$ and $M_{\infty}^B(\xi^{r}\a^{-1} \zeta^{s+1})$ are supported on $\Spec (R_{\infty}^{\psi \e^{-1}}((-1,0),\tau(-1)) \otimes_{\cO} \F)$ and $\Spec(R_{\infty}^{\psi\e^{-1}}( (-1,0),\tau_1(-1)) \otimes_{\cO} \F)$ respectively. Hence $M_{\infty}^B(\xi^{r} \zeta^{s+1})$ and $M_{\infty}^B(\xi^{r}\a^{-1} \zeta^{s+1})$ are supported on different irreducible components of $\Spec\left( (R_{\infty}^{ \psi\e^{-1} }/I_{\Theta_3}) \otimes_{\cO}\F\right). $ We deduce that 
\begin{align}
{\rm rank}_{R_{\infty}^{ \psi\e^{-1} }((-1,0),\tau_1(-1)) \otimes_{\cO} \F} (M_{\infty}^B(\xi^{r}\a^{-1} \zeta^{s+1})) & = {\rm rank}_{R_{\infty}^{ \psi\e^{-1}}((-1,0),\tau(-1)) \otimes_{\cO} \F} (M_{\infty}^B(\xi^{r} \zeta^{s+1}))  \label{identity-rank} \\
& = {\rm rank}_{(R_{\infty}^{ \psi\e^{-1} }/I_{\Theta_3}) \otimes_{\cO}\F} (M_{\infty}^B(\Theta_3 / p \Theta_3)). \nonumber
\end{align}
Consequently $m=n.$
\end{proof}

\begin{corollary}\label{Cor-quat-Serre-wt}
For any $\chi_1, \chi_2 \in W_D(\brho),$ we have  $$\dim_{\F}\Hom_{\cO_D^{\times}}(\chi_1,\pi^B(\overline{r}))=\dim_{\F}\Hom_{\cO_D^{\times}}(\chi_2,\pi^B(\overline{r})).$$ 
\end{corollary}
\begin{proof}
In view of (\ref{identity-rank}), it remains to treat the case $\chi_2 = \chi_1^p$, equivalently $\chi_2$ equals to the conjugation of $\chi_1$ by $\varpi_D$. But this is clear since $\pi^B(\overline{r})$ is a representation of $D^{\times}$, hence is stable under taking the conjugation by $\varpi_D$. See also \cite[Lem.~2.3]{Gee-Savitt} which is based on an observation of Serre. 
\end{proof}

\begin{theorem}\label{thm--freeness-wtTheta}
The $R_{\infty}^{ \psi\e^{-1}}$-module $M_{\infty}^B(\wt{\Theta})$ is free of rank $m$ over its scheme-theoretic support with $I_{\wt{\Theta}} = I_{\Theta_1}\cap I_{\Theta_2} \cap I_{\Theta_3}.$
\end{theorem}
\begin{proof}
We first prove $M_{\infty}^B(\Theta)$ is free of rank $m$ over its scheme-theoretic support with $I_{\Theta} = I_{\Theta_1} \cap I_{\Theta_2}$. By the short exact sequence (\ref{equation-Dlattice-ses1}) and the exactness of the functor $M_{\infty}^B(-),$ we have
\[
M_{\infty}^B(\Theta) \simto M_{\infty}^B(\Theta_1 )\times_{M_{\infty}^B(\Theta_1/p \Theta_1 )}  M_{\infty}^B( \Theta_2).
\]
As for a commutative ring $A$ and two ideals $I_1,I_2 \subset A,$
\[
A/ I_1\cap I_2 \cong A/I_1 \times_{A/ (I_1+I_2)} A/I_2,
\]
by Proposition \ref{prop--freeness} we are reduced to checking
\[
I_{\Theta_1} + I_{\Theta_2} = (p,I_{\Theta_1}) = \Ann_{R_{\infty}^{ \psi\e^{-1}}}\left( M_{\infty}^B (\Theta_1/ p \Theta_1) \right).
\]
Using Corollary \ref{cor-ideals-relation}(i) and Proposition \ref{prop--freeness} again, we have
\[
I_{\Theta_1} + I_{\Theta_2} =   (I_{R_1 } + I_{R_2}) R_{\infty}^{ \psi\e^{-1}} =  (p, I_{R_1}) R_{\infty}^{ \psi\e^{-1}}  = (p, I_{\Theta_1}) = \Ann_{R_{\infty}^{ \psi\e^{-1}}}\left( M_{\infty}^B (\Theta_1/ p \Theta_1) \right).
\]
In particular, we obtain
\begin{equation}\label{ideal-I_Theta}
I_{\Theta}=I_{\Theta_1} \cap I_{\Theta_2} = I_{R_1} R_{\infty}^{ \psi\e^{-1}}  \cap I_{R_2} R_{\infty}^{ \psi\e^{-1}} =(I_{R_1} \cap I_{R_2}) R_{\infty}^{ \psi\e^{-1}} = I_R R_{\infty}^{ \psi\e^{-1}}
\end{equation}
where the third equality holds by \cite[Thm.~7.4(ii)]{Matsumura-CRT} because $R_{\infty}^{\psi\e^{-1}} $ is flat over $R_{\brho}^{\psi\e}$.

Now we prove $M_{\infty}^B(\wt{\Theta})$ is free of rank $m$ over its scheme-theoretic support. Using (\ref{equation-Dlattice-ses2})   we have similarly
\[
M_{\infty}^B(\wt{\Theta}) \simto M_{\infty}^B(\Theta )\times_{M_{\infty}^B(\Theta_1/p \Theta_1 )}  M_{\infty}^B( \Theta_3)
\]
and it suffices to check
\[
I_{\Theta} + I_{\Theta_3}  =  \Ann_{R_{\infty}^{ \psi\e^{-1}}}\left( M_{\infty}^B (\Theta_1/ p \Theta_1) \right) = (p, I_{\Theta_1}) .
\]
This easily follows from (\ref{ideal-I_Theta}) and Corollary \ref{cor-ideals-relation}(ii).  
\end{proof}

\begin{corollary}\label{cor-cyclic-projchi}
For any $\chi \in W_D(\brho),$ the natural inclusion
\begin{equation}\label{eqn-criterion}
\Hom_{\cO_D^{\times}} (\chi,\pi^B (\overline{r}) ) \into \Hom_{\cO_D^{\times}} (\overline{W}_{\chi, 3 },\pi^B (\overline{r}) )
\end{equation}
is an isomorphism, where the structure of $\overline{W}_{\chi, 3 }$ is given in Proposition \ref{prop-lattice-wtTheta}.
\end{corollary}

\begin{proof}
The mod $p$ reduction of the lattice $\wt{\Theta}$ is isomorphic to $\overline{W}_{\chi,3}$ by Proposition \ref{prop-lattice-wtTheta}.  The result then follows from  Theorem \ref{thm--freeness-wtTheta}.
\end{proof}

The main result of this section is the following.

\begin{theorem}\label{thm:GKdim-B}
Maintain all the assumptions we have made on $F,$ $B,$ and $\overline{r}.$ Assume $\brho = \overline{r}_v(1)$ satisfies \ref{C1} or \ref{C2}. Then $\dim_{\cO_D^{\times }} \left( \pi^B (\overline{r}) \right) = 1. $
\end{theorem}
\begin{proof}
Since $ \pi^B (\overline{r})$ is of infinite dimension over $\F$ by \cite[Cor.~3.2.4]{BreuilDiamond} (or \cite[Thm.~7.8]{Scholze}), $\dim_{\cO_D^{\times }} \left( \pi^B (\overline{r})\right) $ is at least one. The other inequality follows from Corollary \ref{cor-cyclic-projchi} and Corollary \ref{cor-gkdim-control}.
 \end{proof}

\begin{remark}
Although we have excluded the case $r=0$ in \ref{C2}, this case (at least when $B$ is indefinite) can be deduced from the case $r=p-3$. The proof uses Scholze's functor introduced in \cite{Scholze} and the mod $p$ local-global compatibility (\`a la Emerton), see Corollary \ref{cor:GKdim-comp}.
\end{remark}

\subsection{The graded module $\gr(\pi^B(\overline{r})^{\vee})$}

Following \cite[\S3.1]{BHHMS2}, we consider the category $\mathcal{C}$ of admissible smooth representations $\pi$ of $D^{\times}$ over $\F$ with a central character and such that there exists a good filtration on the $\pi^{\vee}$ such that the $\gr\F[\![U_D^1/Z_D^1]\!]$-module $\gr (\pi^{\vee})$ is annihilated by some power of the ideal $(yz,zy)$, where $y=y_0$, $z=z_0$ are as in \S\ref{subsection:GK}.\footnote{This category $\mathcal{C}$ is not exactly the one considered in \cite{BHHMS2}, but compare \cite[Prop.~3.1.2.11]{BHHMS2}.} It is clear that $\mathcal{C}$ is an abelian category and is stable under subquotients and extensions.

\begin{definition}
For each $\chi\in W_D(\brho)$, we define an ideal $\mathfrak{a}(\chi)$ of $\F[y,z]$ as follows. 
\begin{itemize}
\item If $\chi\alpha^{-1}\in W_D(\brho)$, then $\mathfrak{a}(\chi):=(y)$; if $\chi\alpha\in W_D(\brho)$, then $\mathfrak{a}(\chi):=(z)$.
\item If neither of $\chi\alpha$, $\chi\alpha^{-1}$ lies in $W_D(\brho)$, then $\mathfrak{a}(\chi):=(yz)$.
\end{itemize}
 \end{definition}

\begin{theorem}
Maintain all the assumptions we have made on $F,$ $B,$ and $\overline{r}.$ Assume $\overline{r}_v$ satisfies \ref{C1} or \ref{C2}. Then there exists a surjective graded morphism \[\Big(\bigoplus_{\chi\in W_D(\brho)}\chi^{\vee}\otimes \F[y,z]/\mathfrak{a}(\chi)\Big)^{\oplus m}\twoheadrightarrow \gr(\pi^B(\overline{r})^{\vee})\]
where the integer $m$ is as in Proposition \ref{prop--freeness}. 
\end{theorem}
 \begin{proof}
This is an easy consequence  of Corollary \ref{cor-cyclic-projchi}. 
\end{proof}

\section{Application to Scholze's functor}\label{section:application} 

\subsection{Results of Scholze and Pa\v{s}k\={u}nas}

Let $L$ be a finite extension of $\Q_p$,  $G: = \GL_n(L)$ and $G_L= \Gal(\overline{L}/L).$ Let $D$ be the central division algebra over $L$ of dimension $n^2$ and invariant $1/n.$ To any $\pi \in  \Mod_G^{\rm adm}(\cO),$ Scholze \cite{Scholze} associated a Weil-equivariant sheaf $\cF_{\pi}$ on the \'etale site of the adic space $\bP^{n-1}_{\C_p}.$ We collect some results of Scholze \cite{Scholze} and Pa\v{s}k\={u}nas \cite{Paskunas-JL}.

\begin{theorem}\label{thm-Scholze}
Let $\pi \in \Mod_{G}^{\rm adm}(\cO).$ 

(i) For any $i\geq 0$ the \'etale cohomology group $H^i_{\mathrm{\acute{e}t}}(\bP^{n-1}_{\C_p}, \cF_{\pi})$ carries a continuous $ G_L \times D^{\times}$-action. Moreover, the restriction of $H^i_{\et}(\bP^{n-1}_{\C_p}, \cF_{\pi})$ to $D^{\times}$ is an admissible smooth representation of $D^{\times}.$ 

(ii) $H^i_{\et}(\bP^{n-1}_{\C_p}, \cF_{\pi}) = 0$ for $i > 2(n-1).$

(iii) Assume $\pi$ admits central character $\psi: Z_G \to \F^{\times}.$ If $\pi$  is injective in $\Mod_{\GL_n(\cO_L),\psi}^{\rm adm}(\cO),$ then $H^i_{\et}(\bP^{n-1}_{\C_p}, \cF_{\pi}) = 0$ for $i > n-1.$

(iv) The natural map
\[
H^0_{\et}(\bP^{n-1}_{\C_p}, \cF_{\pi^{\SL_n(L)}}) \into H^0_{\et}(\bP^{n-1}_{\C_p}, \cF_{\pi})
\]
is an isomorphism. In particular if $\pi^{\SL_n(L)} = 0$ then $H^0_{\et}(\bP^{n-1}_{\C_p}, \cF_{\pi})  = 0.$

(v) If $\pi = \ide_{G}$ is the trivial representation of $G$ over $\F,$ then
\[
H^i_{\mathrm{\acute{e}t}}(\bP^{n-1}_{\C_p}, \cF_{\ide_G}) =  \left\{ {\begin{array}{ll}
\o^{-i/2}\otimes \ide_{D^{\times}}  & \text{ if $i$ is even and $0\leq i \leq 2(n-1)$;} \\
0  & \text{ if $i$ is odd.}\\
\end{array}}\right.
\]
\end{theorem}

\begin{proof}
(i) and (ii) are proved in \cite[Thm.~3.2]{Scholze}. 

(iii)\footnote{We thank J. Ludwig for her help with the proof.} It is proved in \cite[Thm.~3.2]{Scholze} that if $\pi \in \Mod_{G}^{\rm adm}(\cO)$ such that $\pi|_{\GL_n(\cO_L)}$ is injective, then  $H^i_{\et}(\bP^{n-1}_{\C_p}, \cF_{\pi}) = 0$ for $i > n-1.$ We need to prove a similar result for $\pi$ which admits a central character. Twisting by a character, we may assume the central character of $\pi$ is trivial. Examining the proof of \cite[Thm.~3.2]{Scholze}, it suffices to show that $\cM_{\infty}/L^{\times}$ is a perfectoid space, where $\cM_{\infty}$ is the infinite level Lubin-Tate space. Passing to the connected components, it suffices to show that $\cM^{(0)}_{\infty}/\cO_L^{\times}$ is a perfectoid space, where $\cM^{(0)}_{\infty}$ is the perfectoid space denoted by $\cM^{(0)}_{\bf 1}$ in \cite[\S4.1]{JLH}. We will deduce this from \cite[Prop.~4.1.1]{JLH} which proves that $\cM^{(0)}_{\infty}/P(\cO_L)$ is a perfectoid space, where $P\subset \GL_n$ is the parabolic subgroup of block form $(n-1,1).$

We use freely the notation of \cite[\S4.1]{JLH}. Note that $\cM_{\infty}$ is denoted by $\cM_{\bf 1}$ in {\em loc.~cit.} Let $H\subseteq \GL_n(\cO_L)$ be a closed subgroup. As in \cite[\S4.1]{JLH}, we set 
\[
\cM^{(0)}_{H} : =\plim_{U\supseteq H} (\cM_U^{(0)})^{\diamondsuit},
\]
 where $U$ ranges over open subgroups of $\GL_n(\cO_L)$ containing $H.$ By \cite[Prop.~4.1.1]{JLH}, $\cM^{(0)}_{P(\cO_L)}$ is the quotient $\cM^{(0)}_{\infty}/ P(\cO_L)$ in Huber's category $\mathcal{V},$ and is a perfectoid space. Now assume $H\subset\GL_n(\cO_L)$ is a closed subgroup contained in $P(\cO_L).$ The same argument as in the proof of \cite[Prop.~3.2.1]{JLH} shows that $\cM^{(0)}_{H}$ is a perfectoid space. More precisely, if $H$ is of finite index in $P(\cO_L),$ then $\cM^{(0)}_{H}$ is finite \'etale over $\cM^{(0)}_{P(\cO_L)},$ and the result then follows. In general $\cM^{(0)}_{H} = \plim_{H'} \cM^{(0)}_{H'}$ where $H'$ ranges over closed subgroups with $H\subseteq H' \subseteq P(\cO_L)$ and $H'\subseteq P(\cO_L)$ has finite index, and the result follows. We can then argue as in the proof of \cite[Lem.~3.3.4]{JLH} to prove that $\cM^{(0)}_{\infty}$ is an $H$-torsor over $\cM^{(0)}_{H}.$ Finally, it follows from \cite[Lem.~3.3.5]{JLH} that $\cM^{(0)}_{H}$ is the quotient $\cM^{(0)}_{\infty}/ H$ in Huber's category $\mathcal{V}.$  We finish the proof by taking $H = \cO_L^{\times}.$

(iv) is proved in \cite[Prop.~4.7]{Scholze}. For (v), we note that $\cF_{\ide_G}$ is the trivial local system on $\bP^{n-1}_{\C_p}.$ It follows from \cite[Thm.~3.8.1]{Huber} that the cohomology of $\bP^{n-1}_{\C_p}$ (with the Galois action) is as in the classical case. As $D^{\times}$ acts on $\bP^{n-1}_{\C_p}$ via an embedding $D^{\times }\into \GL_n(L^{\rm un}),$ $D^{\times}$ acts trivially on the cohomology.
\end{proof}
 
Let $\pi $ be a locally admissible $\cO$-torsion representation of $G$.  The construction of the sheaf $\mathcal{F}_{\pi}$ in \cite[Prop.~3.1]{Scholze} extends to such $\pi$. Write $\pi =  \ilim_{\pi'} \pi'$ where the limit is taken over all admissible subrepresentations of $\pi.$ By  \cite[Eq.~(9)]{Paskunas-JL}, we have
\[
H^i_{\et}(\bP^{n-1}_{\C_p}, \cF_{\pi}) = \ilim_{\pi'}H^i_{\et}(\bP^{n-1}_{\C_p}, \cF_{\pi'}) .
\]
We denote by $\cS^i$ the cohomological covariant $\d$-functor
\[
\cS^i:\Mod_G^{\rm l.adm}(\cO) \to \Mod_{G_{L}\times D^{\times}}^{\rm l.adm}(\cO) ,\ \  \pi \mapsto H^i_{\et} (\bP^1_{\C_p} , \cF_{\pi}),
\]
where $\Mod_{G_{L}\times D^{\times}}^{\rm l.adm}(\cO)$ is the category of locally admissible representations of $D^{\times}$ on $\cO$-torsion modules equipped with a continuous commuting $G_{L}$-action. As in \cite{Paskunas-JL}, it is more convenient  to work on pseudocompact modules rather than smooth representations via the Pontryagin duality. Namely, we  consider the covariant homological $\d$-functor $\{\check{\cS}^i \}_{i \geq 0}$ defined by
\[
\check{\cS}^i : \frak{C}_G(\cO) \to \frak{C}_{G_{L}\times D^{\times}}(\cO), \quad  M \mapsto H^i_{\et} (\bP^1_{\C_p} , \cF_{M^{\vee}})^{\vee}.
\]

If $R$ is a  complete local noetherian $\cO$-algebra with residue field $\F$, we extend the $\delta$-functor $\check{\cS}^i$ to $ \frak{C}_G(R)$ (defined in \S\ref{subsection:Endo}) in a similar way.

\subsection{Local-global compatibility (\`a la Scholze)}

From now on, we follow the notation of \S\ref{section-autom-forms}.   Let $B$ be an indefinite quaternion algebra over the totally field $F$ such that $B$ is ramified at the fixed place $v$ above $p.$ Let $B'$ be the definite quaternion algebra over $F$ which splits at $v$ and has the same ramification behavior as $B$ at all the other finite places. Fix an isomorphism $B^{\times}(\A_{F,f}^{v}) \cong B'^{\times}(\A_{F,f}^{v}).$ Fix an open compact subgroup $U^{v}\subset B^{\times}(\A_{F,f}^{v}) \cong B'^{\times}(\A_{F,f}^{v}).$ Let $\overline{r} : G_{F} \to \GL_2(\F)$ be a modular Galois representation and let $\frak{m}_{\overline{r}} $ be the non-Eisenstein maximal ideal associated to $\overline{r}$ as in \S\ref{section-autom-forms}. Let $\s_p^{v}$ be the finite $\cO[\![U^{v}]\!]$-module as in (\ref{eqn::def-of-sigma}). If $A$ is a topological $\cO$-algebra, let
\begin{align*}
\wt{S}_{\s_p^v, \psi} (U^{v} , A) & : = \wt{H}^{0,B'}_{\s_p^v , \psi} (U^{v}, \cO) \otimes_{\cO} A\\
\wt{H}^i_{\s_p^{v}, \psi} (U^{v} , A) & :=\wt{H}^{i,B}_{\s_p^v , \psi} (U^{v}, \cO) \otimes_{\cO} A,\ \ i\geq 0.
\end{align*}
The Hecke algebra  $\TM(U^v)_{\frak{m}_{\overline{r}}}$   acts faithfully and continuously on $\wt{S}_{\s_p^v, \psi} (U^{v} , \cO)_{\frak{m}_{\overline{r}}}$ and $\wt{H}^i_{\s_p^{v}, \psi} (U^{v} , \cO)_{\frak{m}_{\overline{r}}}$ (see \cite[Cor.~7.3]{Scholze}).

Let $M_{\infty} $ denote the big patched module $M_{\infty}^{B'}$ in \S\ref{section-autom-forms}, so that
 \[ M_{\infty}/\fa_{\infty}\cong \wt{S}_{\sigma_p^v,\psi}(U^v,\cO)_{\fm_{\overline{r}}}^d \text{ and } M_{\infty}/\fm_{\infty}\cong \wt{S}_{\sigma_p^v,\psi}(U^v,\F)[\fm_{\overline{r}}]^{\vee}
\]
where $\fa_{\infty}$ denotes the ideal $(z_1,\dots,z_q,y_1,\dots,y_j)$ of $S_{\infty}$ and $\fm_{\infty}$ denotes the maximal ideal of $R_{\infty}^{\psi\e^{-1}}$.

On the other hand, let $N_{\infty}$ be the variant of $M_{\infty}^B$ which is obtained by the same patching process as $M_{\infty}^{B}$, but without ``factorizing out'' the Galois representation, see Remark \ref{rem:variant-M}. Similarly to (\ref{eqn--M-infty}) we have
\[
N_{\infty} / \frak{a}_{\infty} =  \wt{H}^{1}_{\s_p^v , \psi} (U^{v}, \cO)^d_{\frak{m}_{\overline{r}}},\quad N_{\infty}/\fm_{\infty}\cong \wt{H}^{1}_{\s_p^v , \psi} (U^{v}, \F)[\frak{m}_{\overline{r}}]^{\vee}.
\]

\begin{theorem}\label{thm--K-injectivity}
Denote the restriction of $\psi$ to $F_v^{\times}$ again by $\psi$. 

(i) $\wt{S}_{\s_p^v ,\psi}(U^{v} , E/\cO)_{\frak{m}_{\overline{r}}}$ lies in $\Mod_{G,\psi}^{\rm adm}(\cO),$ and its restriction to $ K:= \GL_2(\cO_{F_v})$ is injective in $\Mod_{K,\psi}^{\rm sm}(\cO).$ Equivalently by taking dual,  $\wt{S}_{\s_p^v ,\psi}(U^{v} , \cO)^d_{\frak{m}_{\overline{r}}}$ is finitely generated over $\cO[\![K]\!]$ and is projective in $\Mod_{K,\psi}^{\rm pro}(\cO).$
 
(ii) $\wt{H}^1_{\s_p^v, \psi} (U^{v} , E / \cO)_{\frak{m}_{\overline{r}}}$ lies in $\Mod_{D^{\times},\psi}^{\rm adm}(\cO),$ and its restriction to $ \cO_{D}^{\times}$ is injective  in $\Mod_{ \cO_{D}^{\times},\psi}^{\rm sm}(\cO).$ Equivalently, $\wt{H}^1_{\s_p^v, \psi} (U^{v} ,  \cO)^d_{\frak{m}_{\overline{r}}}$ is a finitely generated $\cO[\![\cO_{D}^{\times}]\!]$-module and is projective in $\Mod_{\cO_{D}^{\times} ,\psi}^{\rm pro}(\cO).$

(iii)  For $0\leq i\leq 2$, there is a canonical isomorphism of $\TM(U^{v})_{\frak{m}_{\overline{r}}}[G_{F_v} \times D^{\times}]$-modules
\[
\check{\cS}^i \big(\wt{S}_{\s_p^v, \psi}(U^{v} , \cO)^d_{\frak{m}_{\overline{r}}}\big)  \cong  \wt{H}^i_{\s_p^v, \psi} (U^{v} , \cO)^d_{\frak{m}_{\overline{r}}}.
\]

(iv) There is a canonical $R_{\infty}^{\psi \e^{-1}}[G_{F_{v}} \times D^{\times}]$-equivariant isomorphism  
\[
\check{\cS}^1 ( M_{\infty} ) = N_{\infty}.
\]
\end{theorem}

\begin{proof}
(i) is \cite[Lem.~5.3, Prop.~5.4]{Paskunas-JL}. (ii) is proved in \cite[Prop.~5.6]{Newton} and \cite[Prop.~6.4]{Paskunas-JL}. (iii) is \cite[Prop.~6.3]{Paskunas-JL}. (iv) follows from (the proof of) \cite[Cor.~9.3]{Scholze}, see \cite[Thm.~8.10 (4)]{DPS-Crelle} for details.
\end{proof}

\begin{lemma}\label{lem:Ihara}
We have $\check{\cS}^0(\wt{S}_{\s_p^v , \psi} (U^{v}, \cO)^d_{\frak{m}_{\overline{r}}})=0$ and $\check{\cS}^0(M_{\infty})=0$.
\end{lemma}
\begin{proof}
The first statement is a direct consequence of Theorem \ref{thm--K-injectivity}(iii) because $\wt{H}^0_{\s_p^v, \psi} (U^{v} , \cO)^d_{\frak{m}_{\overline{r}}}=0$ (as $\frak{m}_{\overline{r}}$ is non-Eisenstein). The second statement follows from this and the patching construction (cf.~\cite[Cor.~9.3]{Scholze}).
\end{proof}

 Define 
\[\pi^{B'}(\overline{r}):= (M_{\infty}/\fm_{\infty})^{\vee},\quad 
\pi^{B} (\overline{r}): = \Hom_{G_F}(\overline{r}, (N_{\infty}/\fm_{\infty})^{\vee}).
\]
Note that $(N_{\infty}/\fm_{\infty})^{\vee}$ is $\overline{r}$-typic, so  we have a $G_{F}\times D^{\times}$-equivariant isomorphism $(N_{\infty}/\fm_{\infty})^{\vee}\cong \overline{r}\otimes\pi^B(\overline{r})$. 
The following result is motivated by \cite[Prop.~3.7, Prop.~4.1]{Paskunas-JL}.
\begin{proposition}\label{prop:equiv}
Assume that $R_{v}^{\psi\e^{-1}}$ is formally smooth and that $\dim_{K}(\pi^{B'}(\overline{r}))=[F_v:\Q_p]$. Then $M_{\infty}$ is a flat $R_{\infty}^{\psi\e^{-1}}$-module.  Moreover, the following statements are equivalent:

\begin{enumerate}
\item[(i)] $\dim_{\cO_D^{\times }} \left( \pi^B (\overline{r}) \right) = [F_v:\Q_p]$;
\item[(ii)] $ N_{\infty}$ is flat over $R_{\infty}^{\psi \e^{-1}}$;
\item[(iii)] $\cS^2 (\pi^{B'} (\overline{r}) ) = 0.$
\end{enumerate}
\end{proposition}

\begin{proof}
Since $R_{v}^{\psi\e^{-1}}$ is formally smooth by assumption, it is isomorphic to a power series ring in $(3 + 3[F_v:\Q_p])$-variables over $\cO$. Consequently,  $R_{\infty}^{\psi \e^{-1}}$ is a regular local ring of Krull dimension equal to $ \dim S_{\infty} + 2 [F_v : \Q_p].$

 Since $M_{\infty}$ is finite projective over $S_{\infty}[\![K/Z_1]\!]$, where $Z_1$ is the centre of $K,$ $\d_{S_{\infty}[\![K]\!]}(M_{\infty}) = \dim S_{\infty} + \dim_{\Q_p} (K/Z_1)$ by \cite[Lem.~A.15]{Gee-Newton}, see \S\ref{Sec::notation} for the notation. Since $\dim_{\Q_p}(K/Z_1)=3[F_v:\Q_p]$ and $\dim_K(\pi^{B'}(\overline{r}))=[F_v:\Q_p]$ by assumption, we deduce  
\[
\dim_{K} (\pi^{B'}(\overline{r})) + \dim R_{\infty}^{\psi \e^{-1}} = \d_{S_{\infty}[\![K]\!]}(M_{\infty}).
\]
It follows from the miracle flatness criterion \cite[Prop.~A.30]{Gee-Newton} that $M_{\infty}$ is flat over $R_{\infty}^{\psi \e^{-1}}.$

Now we prove the equivalence between the three statements. The equivalence (i) $\Leftrightarrow$ (ii) is proved as above by replacing $K/Z_1$ by $\cO_D^{\times}/Z_D^1$ and noting that $\dim_{\Q_p}\cO_D^{\times}/Z_D^1 =3[F_v:\Q_p]$.

We prove (ii) implies (iii). Since $R_{\infty}^{\psi\e^{-1}}$ is regular, we may choose a regular system of parameters of $\fm_{\infty}$, say  $\underline{s}$. Since $M_{\infty}$ is flat over $R_{\infty}^{\psi \e^{-1}}$, the Koszul complex $K_{\bullet}(\underline{s}, M_{\infty}) $ gives a resolution of $ \pi^{B'}(\overline{r})^{\vee}=M_{\infty}/\fm_{\infty}$:
\[
\cdots\to K_2(\underline{s} , M_{\infty}) \To{d_2} K_1(\underline{s} , M_{\infty}) \To{d_1}  K_0(\underline{s} , M_{\infty}) \To{d_0} M_{\infty}/\fm_{\infty} \to 0.
\]
It follows from Lemma \ref{lem:Ihara}(ii) that $\check{\cS}^0(K_i(\underline{s},M_{\infty}))=0$ for any $ i$, hence $\check{\cS}^0(\mathrm{Im}(d_i))=0$ as well. It is then easy to deduce that the following sequence
\begin{equation}\label{eq:exact-d1}
\check{\cS}^1(K_2(\underline{s} , M_{\infty})) \to\check{\cS}^1( K_1(\underline{s} , M_{\infty})) \to \check{\cS}^1(Q) \to 0
\end{equation}
is exact, where $Q:=\mathrm{Im}(d_1)=\Ker(d_0)$.
On the other hand, the functor $\check{\cS}^1$ is $R_{\infty}^{\psi\e^{-1}}$-equivariant, so the complex $\check{\cS}^1 (K_{\bullet}(\underline{s}, M_{\infty}))$ is isomorphic to $ K_{\bullet}(\underline{s},\check{\cS}^1 ( M_{\infty}))$, the Koszul complex with respect to $\underline{s}$ and $\check{\cS}^1 (M_{\infty})$. Since $ \check{\cS}^1 (M_{\infty})\cong N_{\infty}$ is flat over $R_{\infty}^{\psi \e^{-1}}$ by (ii),  the complex $\check{\cS}^1 (K_{\bullet}(\underline{s}, M_{\infty}))$ is again exact. Together with \eqref{eq:exact-d1} this implies that the map \begin{equation}\label{eq:inj-Q}\check{\cS}^1 (Q) \to \check{\cS}^1( K_0(\underline{s} , M_{\infty}) )\end{equation} is injective.

The short exact sequence $0\to Q \to K_0(\underline{s} , M_{\infty}) \To{d_0} M_{\infty}/\fm_{\infty}\to 0$ induces an exact sequence
\[
\check{\cS}^2( K_0(\underline{s} , M_{\infty}) ) \to \check{\cS}^2( M_{\infty}/\fm_{\infty}) \to\check{\cS}^1 (Q) \to \check{\cS}^1( K_0(\underline{s} , M_{\infty}) )
\]
in which the first morphism is surjective by the injectivity of \eqref{eq:inj-Q}. Since $M_{\infty}^{\vee}|_{K}$ is injective in $\Mod^{\rm sm}_{K,\psi}(\cO)$, Theorem \ref{thm-Scholze}(iii) implies $\check{\cS}^2( K_0(\underline{s} , M_{\infty}) )=0$, thus $\check{\cS}^2 (M_{\infty}/\fm_{\infty}) = 0 $ as required.

We prove (iii) implies (ii). This essentially follows from the above argument. Indeed, we deduce from (iii)  the injectivity of \eqref{eq:inj-Q}, which together with \eqref{eq:exact-d1} implies the exactness of
\[\check{\cS}^1(K_2(\underline{s} , M_{\infty})) \to\check{\cS}^1( K_1(\underline{s} , M_{\infty})) \to \check{\cS}^1(K_0(\underline{s} , M_{\infty})) \to 0.\]
In other words, the Koszul complex $\check{\cS}^1 (K_{\bullet}(\underline{s}, M_{\infty}))$ is exact at degree $1$, thus $\underline{s}$ is $N_{\infty}$-regular by a standard argument. 
\end{proof}

\begin{remark}
Under some (stronger) genericity condition on $\overline{r}|_{G_{F_v}}$, the   assumption on $\dim_K(\pi^{B'}(\overline{r}))$  of Proposition \ref{prop:equiv} is verified in \cite{BHHMS1}, \cite{Hu-Wang}.   
\end{remark}

We recall the following important result of Scholze. 
\begin{proposition}\label{prop:Sch-7.7}
There is a $G_{\Q_p}\times D^{\times}$-equivariant inclusion
\[\cS^1(\pi^{B'}(\overline{r}))\subset (\overline{r}|_{G_{F_v}})\otimes \pi^B(\overline{r}),
\] whose cokernel is annihilated by $(\cO_{D}^{\times})_1,$ where $(\cO_{D}^{\times})_1$ denotes the reduced norm $1$ elements of $\cO_D^{\times}.$ As a consequence, the cokernel is finite dimensional over $\F$ and $\dim_{\cO_D^{\times}}\cS^1(\pi^{B'}(\overline{r}))=\dim_{\cO_D^{\times}}\pi^B(\overline{r})$.
\end{proposition}

\begin{proof}
The first assertion is a restatement of \cite[Prop.~7.7]{Scholze} and \cite[Lem.~6.1]{Paskunas-JL}. The second assertion  follows from the first (note that we have fixed the central character).
\end{proof}

\subsection{Local-global compatibility (\`a la Emerton)}
\label{ss:lg}
 In this subsection, we assume $F_v\cong \Q_p$. Assume $\End_{G_{\Q_p}}(\brho)= \F$ where $\brho=\overline{r}_v(1)$, and let $\pi(\brho)$ be the admissible smooth representation of $G=\GL_2(\Q_p)$ attached to $\brho$ (cf. \S\ref{ss:LLC}).

\begin{theorem}\label{thm:lg}
We have $\pi^{B'}(\overline{r})\cong \pi(\brho)^{\oplus d}$ for some $d\geq 1$.
\end{theorem}

\begin{proof}
If $\brho\nsim \smatr{\chi}{*}0{\chi\omega}$ for any character $\chi$, the result  is essentially a consequence of \cite{Em3} (which treats the case of $\GL_{2/\Q}$). In the definite quaternion algebra setting, the proof is carried out in \cite[Appendix]{Dosp-Lebras}. Note that in \cite{Dosp-Lebras} the quaternion algebra is assumed to be over $\Q$, but the argument goes through in our setting, under the assumption that $F_v$ is isomorphic to $\Q_p$. Another assumption made in \cite{Dosp-Lebras} is that $\brho$ is irreducible, but the only places where this assumption is needed are as follows.
\begin{itemize}
\item On page 403, the proof of Lemma 13.6. In our case, the proof goes through if we replace the vector $v$ (in \emph{loc.~cit.}) by a finite dimensional subspace which generates $\pi(\brho)$ over $G$ (compare the proof of \cite[Thm.~6.3.12]{Em3}). 
\item On page 404, the proof of Lemma 13.9, to ensure that $r(\p)|_{G_{\Q_p}}$ is absolutely  irreducible for $\p$ in a suitable set $\mathcal{C}$ defined before Lemma 13.8. But this can be avoided by replacing $\mathcal{C}$ by the subset of ``allowable'' points as in \cite[Def.~5.4.7]{Em3}.

\item On page 405, the proof of the injectivity of
\[\pi(\brho)\otimes \Hom_{G}\big(\pi(\brho),\wt{S}_{\s_p^v, \psi} (U^{v} , \F)_{\fm_{\overline{r}}}\big)\ra\wt{S}_{\s_p^v, \psi} (U^{v} , \F)_{\fm_{\overline{r}}}.\]
 This can be proved as in the proof of \cite[Thm.~6.4.16]{Em3} for $\brho\nsim \smatr{\chi}{*}0{\chi\omega}$. 

\end{itemize}
If $\brho\sim \smatr{\chi}{*}0{\chi\omega}$, the result follows from \cite[\S4]{CEGGPS2}. Remark that a multiplicity one assumption is made in \emph{loc.~cit.}, but the necessary modification is given in \cite[\S5]{Gee-Newton}.  
\end{proof}

\begin{remark}\label{rem:lg}
In fact, \cite[Thm.~4.32]{CEGGPS2}  and its generalization   \cite[Cor.~5.3.2]{Gee-Newton} prove a  much stronger statement  than Theorem \ref{thm:lg}. Namely, assuming moreover that $\brho\nsim \smatr{\chi\omega}{*}0{\chi}$ for any character $\chi$,  there is an isomorphism in $\frak{C}_{G,\psi}(R_{\infty}^{\psi \e^{-1}})$
\[M_{\infty}\cong R_{\infty}^{\psi \e^{-1}}\wh{\otimes}_{R_{\brho}^{\psi\e}}N^{\oplus d}\]
where $N:= N_{\psi}\in \frak{C}_{G,\psi}(\cO)$ is the object attached to $\brho$ in \S\ref{section-Morra}, and $d$ is the integer in Theorem \ref{thm:lg}. 
\end{remark}

 \begin{corollary}\label{cor:GKdim-comp}
Maintain the global assumptions we have made in Theorem \ref{thm:GKdim-B}, and assume up to twist $\brho \sim\smatr{\mathrm{unr}_1 \omega}{*}0{\mathrm{unr}_2}$. Then $\dim_{\cO_D^{\times}}(\pi^B(\overline{r}))=1$.
\end{corollary}

\begin{proof}
As in the proof of Theorem \ref{thm:GKdim-B}, it suffices to prove $\dim_{\cO_D^{\times}}(\pi^B(\overline{r}))\leq 1$. We reduce the result to a situation covered by Theorem \ref{thm:GKdim-B}.  

Let $\brho'\sim \smatr{\mathrm{unr}_2}{*}0{\mathrm{unr}_1\omega}$ with $*\neq 0$.  Choose a global setup, namely a totally real field $\wt{F},$ an indefinite quaternion algebra $\wt{B}$ over $\wt{F}$ which is ramified at $v,$ and a modular absolutely irreducible Galois representation $\overline{r}'$ as in Theorem \ref{thm:GKdim-B}, such that $\brho'\cong \overline{r}_v'(1)$. 
Then $\brho'$ satisfies \ref{C2} in \S\ref{section-def-ring}, and so $\dim_{\cO_D^{\times}}(\pi^{\wt{B}}(\overline{r}'))=1$ by Theorem \ref{thm:GKdim-B}(ii).  Combining Theorem \ref{thm:lg} and Proposition \ref{prop:Sch-7.7}, we deduce that  $\dim_{\cO_D^{\times}}\cS^1(\pi(\brho'))\leq 1$. The structure of $\pi(\brho')$ is recalled in \S\ref{ss:LLC}. In particular, the set $\JH(\pi(\brho'))$ consists of non-supersingular representations. Using Theorem \ref{thm-Scholze}(iv)  and Ludwig's result \cite{Ludwig}, we have \begin{equation}\label{eq:nongeneric-S2}\dim_{\cO_D^{\times}}\cS^0(\pi)=\dim_{\cO_D^{\times}}\cS^2(\pi)=0\end{equation}
for any non-supersingular irreducible representation $\pi$ of $G$.  We deduce that $\dim_{\cO_D^{\times}}\cS^1(\pi)\leq 1$ for any   $\pi\in \JH(\pi(\brho'))$. It is clear from the definition of $\pi(\brho)$ (see Proposition \ref{prop:kappa-rho1})  that $\JH(\pi(\brho))$ differs from $\JH(\pi(\brho'))$ by at most $1$-dimensional representations. Hence, using \eqref{eq:nongeneric-S2}  we obtain \[\dim_{\cO_D^{\times}}\cS^1(\pi(\brho))\leq 1.\] 
By Theorem \ref{thm:lg} and Proposition \ref{prop:Sch-7.7} again, this implies $\dim_{\cO_D^{\times}}(\pi^B(\overline{r}))\leq 1$. 
\end{proof}

\subsection{Vanishing for supersingular representations}

Ludwig \cite{Ludwig} has proved that $\cS^2(\pi)=0$ if $\pi$ is a principal series of $\GL_2(\Q_p).$ Together with Theorem \ref{thm:GKdim-B}, we deduce the following vanishing result when $\pi$ is supersingular.
\begin{corollary}\label{cor:S2-ss}
Assume that $\pi = \pi(\brho)$ is supersingular with $2\leq r\leq p-3$ in the notation of \ref{C1}.   Then $\cS^2(\pi)=0$. Moreover, we have  
 $\dim_{\cO_D^{\times}}\cS^1(\pi)=1$.
\end{corollary}
\begin{proof}
As $\brho$ is irreducible, the ring $R_{v}^{\psi\e^{-1}}$ is formally smooth. It is proved in \cite[Lem.~5.16]{Paskunas-JL} that $\dim_{K}(\pi)=1$, so the assumptions of Proposition \ref{prop:equiv} hold via Theorem \ref{thm:lg}. The existence of a suitable $B'$ and $\overline{r}$ is well-known, see for example \cite{DiamondTaylor}.  Thus, the vanishing of $\cS^2(\pi)$ follows from this and Proposition \ref{prop:equiv}.  Finally,  Theorem \ref{thm:GKdim-B} and Proposition \ref{prop:Sch-7.7} imply that $\dim_{\cO_D^{\times}}\cS^1(\pi^{B'}(\overline{r}))=1$, hence also $\dim_{\cO_D^{\times}}\cS^1(\pi)=1$ via Theorem \ref{thm:lg} again.  
\end{proof}
  
\section{Further studies on Scholze's functor}
\label{sec:JL}
In this section, we study the behavior of $\cS^i$ on some non-supersingular representations of $\GL_2(\Q_p)$. Recall that $p\geq 5$.

\subsection{Preparations}

\subsubsection{Some definitions}

Recall that $D$ is the nonsplit quaternion algebra over $\Q_p$ and $U_D^1=1+\p_D$.

\begin{definition}\label{def:V-fd}
Given an admissible smooth $\F$-representation $V$ of $D^{\times}$, let $V_{\rm fd}$ be the largest finite dimensional quotient of $V$.
\end{definition}

\begin{remark}\label{rem:Vfd}
That $V_{\rm fd}$ is well-defined can be seen as follows. Let $V^{\vee}$ be the Pontryagin dual of $V$. Then by the general theory of finitely generated modules over $\F[\![U_D^1]\!]$(see e.g.~\cite[\S3.1]{Ven}) $V^{\vee}$ has a largest submodule of $\delta$-dimension $0$ (i.e.~finite dimensional over $\F$). Clearly this submodule is $D^{\times}$-stable because $V^{\vee}$ carries a compatible action of $D^{\times}$. Taking the dual back gives  $V_{\rm fd}$ in Definition \ref{def:V-fd}.
\end{remark}

We give some basic properties of $(\cdot)_{\rm fd}$. For $i\geq 0$ and $M$ a finitely generated $\F[\![U_D^1]\!]$-module, set \[\EE^i(-):=\Ext^i_{\F[\![U_D^1]\!]}(-,\F[\![U_D^1]\!]).\] Note that $\EE^i(-)=0$ for $i\geq 5$, as $\F[\![U_D^1]\!]$ is an Auslander regular ring of global dimension $4$. Also recall that $M$ (when it is nonzero) is called \emph{Cohen-Macaulay} if there exists exactly one $i$ such that $\EE^i(M)\neq0$; in this case $i$ equals to the grade of $M$.

\begin{lemma}\label{lem:CM-Vfd=0}
Let $V$ be an admissible smooth $\F$-representation of $D^{\times}$.  If $V$ is infinite dimensional (as $\F$-vector space) and $V^{\vee}$ is Cohen-Macaulay as $\F[\![U_D^1]\!]$-module, then $V_{\rm fd}=0$.
\end{lemma}
\begin{proof}
By assumption, $V^{\vee}$ is Cohen-Macaulay with $\dim_{\cO_D^{\times}}(V)\geq 1$, thus $\EE^4(V^{\vee}) = 0$.
If $V_{\rm fd}\neq0$, then the inclusion $(V_{\rm fd})^{\vee}\hookrightarrow V^{\vee}$ induces a surjection
\[\EE^4(V^{\vee})\ra \EE^4((V_{\rm fd})^{\vee})\ra0.\]
This gives a contradiction as $\EE^4((V_{\rm fd})^{\vee})\neq0$.
\end{proof}

\begin{lemma}\label{lem:V-fd}
If $0\ra V'\ra V\ra V''\ra 0$ is a short exact sequence of admissible smooth $\F$-representations of $D^{\times}$, then $(V')_{\rm fd}\ra V_{\rm fd}\ra (V'')_{\rm fd}\ra0$ is exact. If moreover $V''$ is finite dimensional over $\F$, then
\[0\ra (V')_{\rm fd}\ra V_{\rm fd}\ra (V'')_{\rm fd}\ra0\]
is exact.
\end{lemma}
\begin{proof}
It is obvious from Definition \ref{def:V-fd}.
\end{proof}

 \begin{lemma}
Let $V$ be an admissible smooth $\F$-representation of $D^{\times}$. Assume that $V$ carries an $\F$-linear continuous action of $G_{\Q_p}$ which commutes with the action of $D^{\times}$. Then $V_{\rm fd}$ is also stable under $G_{\Q_p}$.
 \end{lemma}

\begin{proof}
Consider the Pontryagin dual $V^{\vee}$, so that $(V_{\rm fd})^{\vee}$ is identified with the largest finite dimensional submodule of $V^{\vee}$, see Remark \ref{rem:Vfd}. It suffices to prove the following statement: if $x\in V^{\vee}$ such that $\langle D^{\times}. x\rangle$ is finite dimensional, then so is $\langle D^{\times}. (gx)\rangle $ for any $g\in G_{\Q_p}$. This is clear because the actions of $G_{\Q_p}$ and of $D^{\times}$ commute.
\end{proof}

We now recall the notion of being $\sigma$-typic from \cite[Def.~5.2]{Scholze}, adapted to our situation. Let $G$ be a group, $\sigma:G\ra \GL_n(\F)$ be an $n$-dimensional  representation and $M$ an $\F[G]$-module.
   Then $M$ is said to be \emph{$\sigma$-typic} if one can write $M$ as a tensor product
\[M=\sigma\otimes_{\F}M_0,\]
such that $G$ acts on $\sigma\otimes_{\F}M_0$ through its action on $\sigma$.

\begin{lemma}\label{lem:typic}
Assume that $\End_{\F[G]}(\sigma)=\F$.

(i) If $M$ is $\sigma$-typic, then $M_0\cong\Hom_{\F[G]}(\sigma,M)$.

(ii) Let   $M'\subset M$  be $\F[G]$-modules and assume that $M'$ is a direct summand of $M$. If $M$ is $\sigma$-typic, then so is $M'$.
\end{lemma}
\begin{proof}
(i) It follows from the same proof of \cite[Prop.~5.3]{Scholze}. In \emph{loc.~cit.} $\sigma$ is assumed to be absolutely irreducible, but in the proof only the assumption $\End_{\F[G]}(\sigma)=\F$ is needed.

(ii) Since $M$ is $\sigma$-typic by assumption, the natural map $\sigma\otimes \Hom_{\F[G]}(\sigma,M)\ra M$ is an isomorphism. Since $M'$ is a direct summand of $M$,  the map
\[\sigma\otimes\Hom_{\F[G]}(\sigma,M')\ra M'\]
is also an isomorphism by functoriality.
\end{proof}

\subsubsection{Complements on Scholze's functor}
Keep the notation in \S\ref{section:application} and assume $F_v\cong \Q_p$. To simplify notation we write
\[\wt{S}(U^v,\F)=\wt{S}_{\s_p^v, \psi} (U^{v} , \F), \quad \wt{H}^1(U^{v} , \F)=\wt{H}^1_{\s_p^v, \psi} (U^{v} , \F).\]
It is a consequence of  \cite[Prop.~5.8]{Scholze} that  $\wt{H}^1(U^{v} , \F)[\fm_{\overline{r}}]$ is $\overline{r}$-typic, so 
\begin{equation}\label{eq:r-JL}
\wt{H}^1(U^{v} , \F)[\fm_{\overline{r}}]\cong\overline{r}\otimes\Hom_{G_{F}}\big(\overline{r},\wt{H}^1(U^{v} , \F)[\fm_{\overline{r}}]\big).
\end{equation}
It is proved in \cite[Prop.~7.7]{Scholze} and \cite[Lem.~6.1]{Paskunas-JL} that there is a $G_{\Q_p}\times D^{\times}$-equivariant inclusion
\begin{equation}\label{eq:S-inclusion}\cS^1\big(\wt{S}(U^v,\F)[\fm_{\overline{r}}]\big)\subset \wt{H}^1(U^{v} , \F)[\fm_{\overline{r}}],
\end{equation} whose cokernel is finite dimensional over $\F$, cf.~Proposition \ref{prop:Sch-7.7}.

\begin{corollary}\label{cor:S-equality}
If $\big(\wt{H}^1(U^{v} , \F)[\fm_{\overline{r}}]\big)^{\vee}$ is a Cohen-Macaulay $\F[\![U_D^{1}]\!]$-module, then \eqref{eq:S-inclusion} becomes an equality.
\end{corollary}

\begin{proof}
Since $\wt{H}^1 (U^{v} , \F)[\fm_{\overline{r}}]$ is always infinite dimensional, see \cite[Cor.~3.2.4]{BreuilDiamond} or  \cite[Thm.~7.8]{Scholze}, the assumption implies that $\big(\wt{H}^1(U^{v} , \F)[\fm_{\overline{r}}]\big)_{\rm fd}=0$ by Lemma \ref{lem:CM-Vfd=0}. The result follows. 
\end{proof}

\begin{remark}
\cite[Lem.~6.1]{Paskunas-JL} also proves a criterion for \eqref{eq:S-inclusion} to be an equality.  Corollary \ref{cor:S-equality} can be viewed as a complement to it.
\end{remark}

\begin{proposition}\label{prop:GN-flat}
Assume that $R_{v}^{\psi\e^{-1}}$ is formally smooth and \[\dim_{\cO_D^{\times}}\big(\wt{H}^1(U^{v} , \F)[\fm_{\overline{r}}]\big)=1.\]
Then $(\wt{H}^1(U^{v} , \cO)_{\fm_{\overline{r}}})^{d}$ is a faithfully flat $\mathbb{T}(U^v)_{\fm_{\overline{r}}}$-module, and  $\big(\wt{H}^1(U^{v} , \F)[\fm_{\overline{r}}]\big)^{\vee}$ is a Cohen-Macaulay $\F[\![U_D^{1}]\!]$-module. In particular, \eqref{eq:S-inclusion} becomes an equality. Moreover, $\cS^2\big(\wt{S}(U^{v} , \F)[\fm_{\overline{r}}]\big)=0$.  
\end{proposition}

\begin{proof}
The faithful flatness is proved by the same argument of \cite[Thm.~B(3)]{Gee-Newton}. 

Since $N_{\infty}$ is a projective object in $\frak{C}_{\cO_{D}^{\times},\psi}(S_{\infty})$, it is a Cohen-Macaulay $S_{\infty}[\![U_D^1]\!]$-module, thus is also Cohen-Macaulay over $R_{\infty}[\![U_D^1]\!]$ by \cite[Lem.~A.29]{Gee-Newton}.  The formal smoothness of $R_{v}^{\psi\e^{-1}}$ ensures that $R_{\infty}$ is formally smooth, namely its maximal ideal $\fm_{\infty}$ is generated by a regular sequence. By the proof of  \cite[Prop.~A.30]{Gee-Newton},  $\big(\wt{H}^1(U^{v} , \F)[\fm_{\overline{r}}]\big)^{\vee}\cong N_{\infty}/\fm_{\infty}$ is a Cohen-Macaulay $\F[\![U_D^1]\!]$-module. 

The last assertion follows from Proposition \ref{prop:equiv}.
\end{proof}
 
For simplicity and clarity we  make the following assumption  in \S\ref{ss:S-generic} and \S\ref{ss:S-nongeneric} below. The general case will be treated  in \S\ref{ss:S-non-minimal}.

(\textbf{H})\ \  Assume $d=1$ in Theorem \ref{thm:lg}, i.e. $\widetilde{S}(U^v,\F)[\fm_{\overline{r}}]\cong\pi(\brho)$.
\medskip

For notational convenience, we make the following definition. Let $\brho : = \overline{r}_v(1).$  
\begin{definition} \label{def:JL}
We define
\begin{equation}\label{eq:def-JL}\JL(\brho):=\Hom_{G_{F}}\big(\overline{r},\wt{H}^1(U^{v} , \F)[\fm_{\overline{r}}]\big)\end{equation}
which is an admissible  smooth $\F$-representation of $D^{\times}$. Then \eqref{eq:r-JL} restricts to a $G_{\Q_p}\times D^{\times}$-isomorphism
\begin{equation}\label{eq:rho-JL}\wt{H}^1(U^{v} , \F)[\fm_{\overline{r}}]\cong \brho(-1)\otimes \JL(\brho).\end{equation}
\end{definition}

Finally we recall the following important results which will be repeatedly used later on.

\begin{theorem}\label{thm:Sch-Lud}
Let $\pi$ be an admissible smooth $\F$-representation of $G$.

(i) The natural morphism $\cS^0(\pi^{\SL_2(\Q_p)})\ra \cS^0(\pi)$ is an isomorphism.

(ii) If $\pi\cong \Ind_{B(\Q_p)}^G\chi$ is a principal series (for some smooth character $\chi:B(\Q_p)\ra \F^{\times}$), then $\cS^2(\pi)=0$.
\end{theorem}
\begin{proof}
(i) is a special case of  Theorem \ref{thm-Scholze}(iv) and (ii) is \cite[Thm.~4.6]{Ludwig}. 
\end{proof}

\subsection{The generic case in the minimal case}\label{ss:S-generic}

In this subsection, we assume $\brho \sim \smatr{\chi_1}*0{\chi_2}$ is reducible nonsplit such that $\chi_1 \chi_2^{-1} \neq \ide,~\o^{\pm 1}$.

\begin{theorem}\label{thm:JL-generic}
Let $\brho$ be as above.
Then $\JL(\brho)$  depends only on $\brho^{\rm ss}$. 
\end{theorem}

\begin{proof}
Write $\brho_1$ (resp. $\brho_2$) for the nonsplit extension $\smatr{\chi_1}*0{\chi_2}$ (resp. $\smatr{\chi_2}*0{\chi_1}$). Combining Theorem \ref{thm:lg} and \cite[Prop.~6.7]{Paskunas-JL},\footnote{We can also apply Theorem \ref{thm:GKdim-B} if $\brho$ satisfies \ref{C2}.} we see that $\dim_{\cO_D^{\times}}\cS^1\big(\wt{S}(U^v,\F)[\fm_{\overline{r}}]\big)=1$, hence \[\dim_{\cO_D^{\times}}\big(\wt{H}^1(U^{v} , \F)[\fm_{\overline{r}}]\big)=1\] by \eqref{eq:S-inclusion}. By Proposition \ref{prop:GN-flat},  \eqref{eq:rho-JL} and assumption (\textbf{H}), we obtain for $i\in\{0,1,2\}$
\begin{equation}\label{eq:S-rhoi}\cS^1(\pi(\brho_i))=\brho_i(-1)\otimes\JL(\brho_i).\end{equation}

Recall from \S\ref{ss:LLC} that
there exist exact sequences \[0\ra \pi_1\ra \pi(\brho_1)\ra \pi_2\ra0\]
\[0\ra \pi_2\ra \pi(\brho_2)\ra \pi_1\ra 0,\]
where $\pi_1:=\Ind_{B(\Q_p)}^G\chi_2\otimes\chi_1\omega^{-1}$ and $\pi_2:=\Ind_{B(\Q_p)}^G\chi_1\otimes\chi_2\omega^{-1}$. Note that $\cS^0(\pi_i)=\cS^2(\pi_i)=0$ for $i\in\{1,2\}$, by Theorem \ref{thm:Sch-Lud}.  Hence, by applying the functor $\cS^i$ and using \eqref{eq:S-rhoi}, we obtain
\begin{equation}\label{eq:seq-Srho1}0\ra \cS^1(\pi_1)\overset{\iota_1}{\ra} \brho_1(-1)\otimes\JL(\brho_1)\ra \cS^1(\pi_2)\ra 0\end{equation}
\begin{equation}\label{eq:seq-Srho2}0\ra \cS^1(\pi_2)\overset{\iota_2}{\ra} \brho_2(-1)\otimes\JL(\brho_2)\ra \cS^1(\pi_1)\ra0.\end{equation}
Since $\brho_1$ is nonsplit, we have
\[\Hom_{G_{\Q_p}}\big(\chi_2,\brho_1\otimes \JL(\brho_1)\big)=0.\]

As a consequence,
$\Hom_{G_{\Q_p}}(\chi_2\omega^{-1},\cS^1(\pi_1))=0$ by \eqref{eq:seq-Srho1}  and applying $\Hom_{G_{\Q_p}}(\chi_2\omega^{-1},-)$ to \eqref{eq:seq-Srho2} gives isomorphisms
\[\Hom_{G_{\Q_p}}(\chi_2\omega^{-1},\cS^1(\pi_2))\cong\Hom_{G_{\Q_p}}\big(\chi_2\omega^{-1},\brho_2(-1)\otimes\JL(\brho_2)\big)\cong \JL(\brho_2),\]
where the last isomorphism follows from the definition of $\brho_2$.  This gives  a $G_{\Q_p}\otimes D^{\times}$-equivariant embedding
\begin{equation}\label{eq:embed-rho2}\chi_2\omega^{-1}\otimes \JL(\brho_2)\hookrightarrow \cS^1(\pi_2).\end{equation}
One checks that its composition with $\iota_2$ (in \eqref{eq:seq-Srho2}) coincides with the morphism obtained by tensoring the inclusion $\chi_2\omega^{-1}\hookrightarrow \brho_2(-1)$ with $\JL(\brho_2)$.
Combined with the short exact sequence
\[0\ra \chi_2\omega^{-1}\otimes\JL(\brho_2)\ra \brho_2(-1)\otimes\JL(\brho_2)\ra \chi_1\omega^{-1}\otimes\JL(\brho_2)\ra0,\]
a diagram chasing gives a surjection
\begin{equation}\label{eq:surj-rho2}\chi_1\omega^{-1}\otimes\JL(\brho_2)\twoheadrightarrow \cS^1(\pi_1).\end{equation}
In particular, when restricted to $G_{\Q_p}$, $\cS^1(\pi_1)$ is semisimple and any  irreducible subquotient of $\cS^1(\pi_1)$ is isomorphic to $\chi_1\omega^{-1}$.

On the other hand, the same argument as above implies an embedding (analogous to \eqref{eq:embed-rho2})
\begin{equation}\label{eq:embed-rho1}\chi_1\omega^{-1}\otimes\JL(\brho_1)\hookrightarrow \cS^1(\pi_1).\end{equation}
We claim that \eqref{eq:embed-rho1} is an isomorphism. Indeed, $\iota_1$ in \eqref{eq:seq-Srho1} induces a $G_{\Q_p}\times D^{\times}$-equivariant embedding
\[\cS^1(\pi_1)/\big(\chi_1\omega^{-1}\otimes \JL(\brho_1)\big)\hookrightarrow \big(\brho_1(-1)\otimes\JL(\brho_1)\big)/\big(\chi_1\omega^{-1}\otimes\JL(\brho_1)\big)\cong \chi_2\omega^{-1}\otimes \JL(\brho_1).\]
However, as shown in the last paragraph, $\cS^1(\pi_1)$ admits only $ \chi_1\omega^{-1}$ as irreducible subquotient (when restricted to $G_{\Q_p}$),  while $\chi_2\omega^{-1}\otimes \JL(\brho_1)$ admits only $\chi_2\omega^{-1}$ as irreducible subquotients. Since $\chi_1\neq \chi_2$, this forces $\cS^1(\pi_1)/\big(\chi_1\omega^{-1}\otimes \JL(
\brho_1)\big)=0$, proving the claim.
In a similar way, the embedding \eqref{eq:embed-rho2} is also an isomorphism and consequently  \eqref{eq:surj-rho2} is an isomorphism.

In summary, we have proven that
\[\chi_1\omega^{-1}\otimes \JL(\brho_2)\overset{\eqref{eq:surj-rho2}}{\cong}\cS^1(\pi_1)\overset{\eqref{eq:embed-rho1}}{\cong} \chi_1\omega^{-1}\otimes \JL(\brho_1). \] Hence, by applying $\Hom_{G_{\Q_p}}(\chi_1\omega^{-1},-)$ we obtain a $D^{\times}$-equivariant isomorphism $\JL(\brho_1)\cong\JL(\brho_2)$.
\end{proof}

\begin{remark}\label{rem:test}
It might be strange that $\JL(\brho_i)$ only carries the information of $\brho^{\rm ss}$. This  can be explained as follows. On the one hand, since $\cS^1(\pi(\brho_i))=\brho_i(-1)\otimes \JL(\brho_i)$, the information of $\brho_i$ is indeed caught by  the functor $\cS^1$. On the other hand, comparing  the quaternionic Serre weights (cf. Propositions \ref{thm:Serre-D}, \ref{lemma--quat-Serre-wt}), $\JL(\brho_1)$ and $\JL(\brho_2)$ have the same set of quaternionic Serre weights. However, we don't expect this phenomenon happens once $L\neq \Q_p$.
\end{remark}

\subsection{The non-generic case in the minimal case}\label{ss:S-nongeneric}

In this subsection, we extend  the result in \S \ref{ss:S-generic} to the case $\brho^{\rm ss}\sim \omega\oplus \ide$ (up to twist). Below, we will denote by $\ide_{G_{\Q_p}}$, $\ide_{G}$ and $\ide_{D^{\times}}$ the trivial representation of $G_{\Q_p}$, $G$ and $D^{\times}$ respectively; sometimes we will omit the subscript if no confusion is caused.
 
Let  $\brho_1\sim \smatr{\omega}{*}0{\ide}$   be a nonsplit extention of $\ide$ by $\omega$;  we don't make assumption on the extension type of $\brho_1$ (i.e. peu ramifi\'e or tr\`es ramifi\'e). On the other hand, $\Ext^1_{G_{\Q_p}}(\omega,\ide)$ is $1$-dimensional;  let $\brho_2\sim \smatr{\ide}{*}0{\omega}$ be the unique nonsplit extension of $\omega$ by $\ide$.

 Let  $\tau_1$ be the universal extension of $\ide_{G}$ by $\Sp$, i.e.
\begin{equation}\label{eq:def-tau1}0\ra \Sp\ra \tau_1\ra \ide_{G}^{\oplus2}\ra0\end{equation}
with $\soc_{G}\tau_1=\Sp$.
Recall from \S\ref{ss:LLC} that there is a short exact sequence
\[0\ra \pi_{\alpha}\ra\pi(\brho_2)\ra\tau_1\ra0\]
where
$\pi_{\alpha}:=\Ind_{B(\Q_p)}^{G}(\omega\otimes\omega^{-1}).$

It is shown in \cite[\S10.1]{PaskunasIHES} that $\dim_{\F}\Ext^1_{G/Z_G}(\pi_{\alpha},\ide_G)=1.$ Thus there exists a unique (up to isomorphism) nonsplit extension
\begin{equation}\label{eq:def-kappa}0\ra \ide_{G}\ra \kappa\ra \pi_{\alpha}\ra0.\end{equation}
On the other hand, there is a natural isomorphism 
 $\Ext^1_{G/Z_G}(\ide_G,\Sp)\cong\Hom(\Q_p^{\times},\F)$  by \cite[Thm.~VII.4.18]{Co}; we denote by $E_{\phi}$ the extension corresponding to $\phi\in\Hom(\Q_p^{\times},\F)$. The next result gives the structure of $\pi(\brho_1)$.

\begin{proposition}\label{prop:kappa-rho1}
We have $\soc_G\pi(\brho_1)\cong\Sp$ and there exist nonsplit extensions
\[0\ra E_{\phi}\ra \pi(\brho_1)\ra \pi_{\alpha}\ra0\]
\[0\ra \Sp\ra \pi(\brho_1)\ra \kappa\ra0.\]
\end{proposition}
\begin{proof}
See \cite[Lem.~6.7]{Paskunas-BM}.
\end{proof}
 
\begin{proposition}\label{prop:S2Sp=0}
The following statements hold.

(i)  $\cS^0(\ide_G)\cong\NEW\otimes\ide_{D^{\times}}$,  $\cS^1(\ide_G)=0$, $\cS^2(\ide_G)\cong \omega^{-1}\otimes \ide_{D^{\times}}$.

(ii)  $\cS^0(\Sp)=\cS^2(\Sp)=0$.

(iii) $\cS^0(\pi_{\alpha})=\cS^2(\pi_{\alpha})=0$.
\end{proposition}

\begin{proof}
(i) follows from Theorem \ref{thm-Scholze}(v). (ii) and (iii) are special cases of Theorem \ref{thm:Sch-Lud},
except  for $\cS^2(\Sp)$ which is  \cite[Cor.~4.7]{Ludwig}.
\end{proof}

\begin{corollary}\label{cor:S0-S2=0}
Let $\pi\in \Mod_{G/Z}^{\rm l.adm}(\cO)$. Assume that each of the irreducible subquotients of $\pi$ lies in   $\{\Sp,\ide_G,\pi_{\alpha}\}$.  Then  $\cS^0(\pi)$ (resp. $\cS^2(\pi)$) admits only  $\NEW$ (resp. $\omega^{-1}$) as subquotients when restricted to $G_{\Q_p}$.
\end{corollary}

\begin{proof}
It is a direct consequence of  Proposition \ref{prop:S2Sp=0}. 
\end{proof}

\begin{proposition}\label{prop:S-rho1-CM}
(i) $\JL(\brho_1)^{\vee}$ is a Cohen-Macaulay $\F[\![U_D^1]\!]$-module.

(ii) We have
$\cS^1(\pi(\brho_1))=\brho_1(-1)\otimes \JL(\brho_1)$
and $\cS^2(\pi(\brho_1))=0$. 
\end{proposition}

\begin{proof}
 Since $R_{\brho_1}^{\psi\e}$ is formally smooth, the assertions follow from Corollary \ref{cor:GKdim-comp} and Proposition \ref{prop:GN-flat}.
\end{proof}
 
\begin{corollary}\label{cor:S0-kappa}
We have $\cS^0(\kappa)\cong \NEW\otimes \ide_{D^{\times}}$ and $\cS^2(\kappa)=0$.
\end{corollary}

\begin{proof}
Since $\cS^0(\pi_{\alpha})=0$, the first assertion is a direct consequence of Proposition \ref{prop:S2Sp=0}(i) via \eqref{eq:def-kappa}. Since $\kappa$ is a quotient of $\pi(\brho_1)$, the second assertion is a consequence of Proposition \ref{prop:S-rho1-CM}(ii).
\end{proof}
 
By Proposition \ref{prop:S2Sp=0}, Proposition \ref{prop:S-rho1-CM} and Corollary \ref{cor:S0-kappa}, the sequence
$0\ra \Sp\ra \pi(\brho_1)\ra \kappa\ra0$ (see Proposition \ref{prop:kappa-rho1})
induces an exact sequence
\begin{equation}\label{eq:S1-rho1}0\ra \NEW\otimes \ide_{D^{\times}}\ra \cS^1(\Sp)\ra \brho_1(-1)\otimes \JL(\brho_1)\ra \cS^1(\kappa)\ra0.\end{equation}
 Similarly, the sequence $0\ra E_{\phi}\ra \pi(\brho_1)\ra \pi_{\alpha}\ra0$ induces an exact sequence
 \begin{equation}\label{eq:S1-rho1-2}
 0\ra \cS^1(E_{\phi})\ra \brho_1(-1)\otimes \JL(\brho_1)\ra \cS^1(\pi_{\alpha})\ra \omega^{-1}\otimes\ide_{D^{\times}}\ra0.
 \end{equation}

\begin{lemma}\label{lem:1-S1Sp}
 We have $\Hom_{G_{\Q_p}}(\omega^{-1},\cS^1(\Sp))=0$,  and $\Hom_{G_{\Q_p}}(\omega^{-1},\cS^1(\tau_1))$ is finite dimensional.
\end{lemma}
\begin{proof}
As $\brho_1\sim \smatr{\omega}{*}0{\ide}$ is assumed to be nonsplit, we have $\Hom_{G_{\Q_p}}(\omega^{-1},\brho_1(-1)\otimes \JL(\brho_1))=0$, which implies the first assertion via  \eqref{eq:S1-rho1}. For the second assertion, we note that the short exact sequence $0\ra \Sp\ra \tau_1\ra (\ide_{G})^{\oplus 2}\ra0$ induces an exact sequence
\begin{equation}\label{eq:Sp-tau1}0\ra (\NEW\otimes\ide_{D^{\times}})^{\oplus 2}\ra \cS^1(\Sp)\ra\cS^1(\tau_1)\ra0\end{equation}
by Proposition \ref{prop:S2Sp=0}(i). By applying $\Hom_{G_{\Q_p}}(\omega^{-1},-)$ to \eqref{eq:Sp-tau1}, we obtain
\[0=\Hom_{G_{\Q_p}}(\omega^{-1},\cS^1(\Sp))\ra\Hom_{G_{\Q_p}}(\omega^{-1},\cS^1(\tau_1))\ra \Ext^1_{G_{\Q_p}}(\omega^{-1},\NEW^{\oplus2})\]
from which the result easily follows.
\end{proof}

\begin{proposition}\label{prop:S1-rho2}
There exists a short exact sequence
\begin{equation}\label{eq:JL-rho2}0\ra \cS^1(\pi(\brho_2))\ra \brho_2(-1)\otimes \JL(\brho_2)\ra (\NEW\otimes\ide_{D^{\times}})^{\oplus 2}\ra0.\end{equation}
As a consequence, $\Hom_{G_{\Q_p}}(\omega^{-1}, \cS^1(\pi(\brho_2)))\cong \JL(\brho_2)$.
\end{proposition}

\begin{proof}
We need to show that the cokernel of \eqref{eq:S-inclusion}  is isomorphic to $(\NEW\otimes \ide_{D^{\times}})^{\oplus 2}$. For this we need a refined version of \cite[Prop.~7.7]{Scholze}, which we put separately in Lemma \ref{lem:Sch-7.7} below. In our situation with $A=\mathbb{T}(U^v)_{\fm_{\overline{r}}}$, $I=\fm_{\overline{r}}$ and $P:=(\wt{S}(U^{v}, \cO)_{\fm_{\overline{r}}})^{d}$,  we
are left to show \begin{equation}\label{eq:nongeneric-Tor}\Tor_1^{\mathbb{T}(U^v)_{\fm_{\overline{r}}}}(\mathbb{T}(U^v)_{\fm_{\overline{r}}}/\fm_{\overline{r}},P)\cong (\ide_G^{\vee})^{\oplus 2}\end{equation} by Proposition \ref{prop:S2Sp=0}(i) (here  we use \cite[Prop.~5.4]{Paskunas-JL} to ensure that $P$ satisfies the assumption (c) of Lemma \ref{lem:Sch-7.7}). This is a consequence of \cite[Prop.~3.30]{HuJEMS},  as we explain below. After enlarging $\F$, we may assume $\mathbb{T}(U^v)_{\fm_{\overline{r}}}/\fm_{\overline{r}}\cong \F$.

To be able to apply \cite[Prop.~3.30]{HuJEMS}, we need to relate $P$ with $N$, where $N$ is the object introduced in \S\ref{section-Morra} for $\brho_2$. We do this by passing to $M_{\infty}$. On the one hand, by  Remark \ref{rem:lg} and  assumption (\textbf{H}) we have   $M_{\infty}\cong R_{\infty}^{\psi \e^{-1}}\widehat{\otimes}_{R_{  \brho_2}^{\psi\e}}N$. Since $R_{\infty}^{\psi \e^{-1}}$ is flat over $R_{\brho_2}^{\psi\e} $, we deduce
\begin{equation}\label{eq:Tor-isom-1}\Tor_1^{R_{\brho_2}^{\psi\e}}(\F,N)\cong \Tor_1^{R_{\infty}^{\psi \e^{-1}}}(\F,M_{\infty}).\end{equation}
On the other hand,  $R_{\infty}^{\psi \e^{-1}} $ acts on $P$ via the isomorphism \eqref{eqn--M-infty} $M_{\infty}/\frak{a}_{\infty}\cong P$,  and the action  factors through  \[R_{\infty}^{\psi \e^{-1}}\twoheadrightarrow R_{\infty}^{\psi \e^{-1}}/\fa_{\infty}\cong R_{\overline{r},\cS}^{\psi\e^{-1}}\twoheadrightarrow\mathbb{T}(U^v)_{\fm_{\overline{r}}}.\]   Recall that $\frak{a}_{\infty}$  is generated by an $M_{\infty}$-regular sequence $z_1,\ldots,z_q, y_1,\ldots,y_j$. By Proposition \ref{prop:faithful}, this sequence is also $R_{\infty}^{\psi \e^{-1}}$-regular and $R_{\infty}^{\psi \e^{-1}}/\fa_{\infty}$ acts faithfully on $P$. But $\mathbb{T}(U^v)_{\fm_{\overline{r}}}$ also acts faithfully on $P$, so the surjection $R_{\infty}^{\psi \e^{-1}}/\fa_{\infty}\twoheadrightarrow \mathbb{T}(U^v)_{\fm_{\overline{r}}}$ is actually an isomorphism.\footnote{This gives a ``big $R=\mathbb{T}$'' result, as mentioned in Remark \ref{rem:faithful}.} Consequently,
\[\Tor_1^{R_{\infty}^{\psi \e^{-1}}}(\F,M_{\infty})\cong \Tor_1^{\mathbb{T}(U^v)_{\fm_{\overline{r}}}}(\F,P).\]
Combining this with \eqref{eq:Tor-isom-1}, we deduce \eqref{eq:nongeneric-Tor}  from \cite[Prop.~3.30]{HuJEMS}.
\end{proof}

\begin{lemma}\label{lem:Sch-7.7}
Let $(A,\fm)$ be a complete noetherian local $\cO$-algebra with $A/\fm\cong \F$ and $P\in \frak{C}_{G/Z_G}(A)$. Assume that
\begin{itemize}
\item[(a)]  $P$ is projective in the category of pseudo-compact $\cO[\![K/Z_1]\!]$-modules;
\item[(b)] $P_{\SL_2(\Q_p)}=0$;
\item[(c)] each of the irreducible subquotients of $P^{\vee}$ lies in $\{\Sp,\ide_G,\pi_{\alpha}\}$.
\end{itemize}
Let $I$ be an ideal of $A$. Then there exists an exact sequence\[0\ra  \check{\cS}^0(\Tor_1^A(A/I,P))\ra A/I \otimes_A \check{\cS}^1(P)\ra \check{\cS}^1(A/I\otimes_AP)\ra 0.\]
\end{lemma}

\begin{proof}
Choose a finite free resolution of $A/I$: $\cdots\ra  F_1\ra F_0\ra A/I\ra 0$. By applying $-\otimes_AP$ to it, we obtain a chain complex
\begin{equation}\label{eq:resol-A/I}\cdots \overset{d_2}{\ra} F_1\otimes_A P\overset{d_1}{\ra} F_0\otimes_AP\overset{d_0}{\ra} A/I\otimes_A P\ra0\end{equation}
whose homology computes $\Tor_i^A(A/I,P)$.
Since each $F_i$ is a finite free $A$-module (for $i\geq 0$),  assumption (a) implies that each   $F_i\otimes_A P$ is projective when restricted to $K$, hence $\check{\cS}^2(F_i\otimes_A P)=0$ by Theorem \ref{thm-Scholze}(iii).  Assumption (b) implies that $\check{\cS}^0(F_i\otimes_A  P)=0$ by Theorem \ref{thm-Scholze}(iv). As a consequence, $\check{\cS}^0(\mathrm{Im}(d_i))=0$ for any $i\geq 0$. On the other hand, since $\check{\cS}^3(-)=0$, we have $\check{\cS}^2(\mathrm{Im}(d_i))=0$ for $i\geq 1$.

We may split (part of) the complex \eqref{eq:resol-A/I}  as
\[0\ra \mathrm{Im}(d_1)\ra F_0\otimes_AP\ra  A/I\otimes_A P\ra0,\ \ \ 0\ra \Ker(d_1)\ra  F_1\otimes_A P\ra \mathrm{Im}(d_1)\ra0\]
from which we deduce long exact sequences
\[0\ra \check{\cS}^2(A/I\otimes_A P)\ra\check{\cS}^1(\mathrm{Im}(d_1))\overset{f}{\ra}\check{\cS}^1(F_0\otimes_AP)\ra \check{\cS}^1(A/I\otimes_A P)\ra 0,\]
\[0\ra\check{\cS}^1(\Ker(d_1))\ra\check{\cS}^1(F_1\otimes_A P)\overset{g}{\ra} \check{\cS}^1(\mathrm{Im}(d_1))\ra \check{\cS}^0(\Ker(d_1))\ra0.\]
Note that $\check{\cS}^1(F_i\otimes_A P)\cong F_i\otimes_A \check{\cS}^1(P)$ (as $F_i$ is a finite free $A$-module), and  that there is an exact sequence 
\[ F_1\otimes_A\check{\cS}^1( P)\overset{f\circ g}{\lra}   F_0\otimes\check{\cS}^1(P)\ra A/I\otimes_A\check{\cS}^1(P)\ra0\]
by tensoring the sequence $F_1\ra F_0\ra A/I\ra0$ with $\check{\cS}^1(P)$.
Recall that a variant of the snake lemma shows that there is a long exact sequence
\[0\ra \Ker(g)\ra \Ker(f\circ g)\ra \Ker(f)\overset{\partial}{\ra} \Coker(g)\ra \Coker(f\circ g)\ra \Coker(f)\ra0.\]
In our situation, this gives (by considering the last four nonzero terms)
\[\check{\cS}^2(A/I\otimes_A P)\overset{\partial}{\ra}\check{\cS}^0(\Ker(d_1))\ra A/I \otimes_A\check{\cS}^1(P)\ra \check{\cS}^1(A/I\otimes_A P)\ra0.\]
By Corollary \ref{cor:S0-S2=0}, assumption (c) implies that $\partial$ is identically zero. Hence, we are left to show \[\check{\cS}^0(\Ker(d_1))=\check{\cS}^0(\Tor_1^A(A/I,P))\]
which follows from the exact sequence
$0\ra \mathrm{Im}(d_2)\ra \Ker(d_1)\ra \Tor_1^A(A/I,P)\ra0$
 (recall $\check{\cS}^0(\mathrm{Im}(d_2))=0$ from the first paragraph of the proof).
\end{proof}

By Theorem \ref{thm:Sch-Lud} the short exact sequence $0\ra \pi_{\alpha}\ra \pi(\brho_2)\ra \tau_1\ra0$ induces an exact sequence
\begin{equation}\label{eq:S1-rho2}0\ra \cS^1(\pi_{\alpha})\ra \cS^1(\pi(\brho_2))\ra \cS^1(\tau_1)\ra0.
\end{equation}

\begin{lemma} \label{lem:S1-kappa}
 Both $\cS^1(\pi_{\alpha})$ and $\cS^1(\kappa)$ are $\omega^{-1}$-typic (when restricted to $G_{\Q_p}$).
 \end{lemma}

\begin{proof}
We claim that  $\Hom_{G_{\Q_p}}(\NEW,\cS^1(\pi_{\alpha}))=\Hom_{G_{\Q_p}}(\NEW,\cS^1(\kappa))=0$. Combining \eqref{eq:S1-rho2} with Proposition \ref{prop:S1-rho2}, we obtain an embedding
\[\cS^1(\pi_{\alpha})\hookrightarrow \brho_2(-1)\otimes \JL(\brho_2).\]
As $\Hom_{G_{\Q_p}}(\NEW,\brho_2(-1))=0$, we deduce that $\Hom_{G_{\Q_p}}(\NEW,\cS^1(\pi_{\alpha}))=0$, as claimed. Using   Proposition \ref{prop:S2Sp=0}(i) and Corollary \ref{cor:S0-kappa}, the sequence $0\ra \ide_G\ra \kappa\ra \pi_{\alpha}\ra0$ induces an exact sequence
\begin{equation}\label{eq:kappa-pi}0\ra \cS^1(\kappa)\ra \cS^1(\pi_{\alpha})\ra \omega^{-1}\otimes\ide_{D^{\times}}\ra0\end{equation}
which implies the claim for $\cS^1(\kappa)$.

The claim implies that the surjection $\brho_1(-1)\otimes\JL(\brho_1)\twoheadrightarrow \cS^1(\kappa)$ in \eqref{eq:S1-rho1} must factor as
\[\brho_1(-1)\otimes\JL(\brho_1)\twoheadrightarrow\omega^{-1}\otimes \JL(\brho_1)\twoheadrightarrow \cS^1(\kappa),\]
where the first quotient map is induced by the natural projection $\brho_1(-1)\sim \smatr{\ide}{*}0{\omega^{-1}}\twoheadrightarrow \omega^{-1}$.
In particular, $\cS^1(\kappa)$ is $\omega^{-1}$-typic. Note that, being a subrepresentation of $\brho_2(-1)\otimes \JL(\brho_2)$, $\cS^1(\pi_{\alpha})$ does not admit any $G_{\Q_p}$-subquotient isomorphic to a nontrivial self-extension of $\omega^{-1}$, so  $\cS^1(\pi_{\alpha})$ is also $\omega^{-1}$-typic by \eqref{eq:kappa-pi}.
\end{proof}

As a consequence of \eqref{eq:JL-rho2}, there exists a $D^{\times}$-equivariant surjection
\begin{equation}\label{eq:def-V2}\JL(\brho_2)\twoheadrightarrow (\ide_{D^{\times}})^{\oplus2}.\end{equation}
We denote  its kernel by $V_2$. Then $\brho_2(-1)\otimes \JL(\brho_2)$ can be filtered by subrepresentations such that the graded pieces are isomorphic to
\[\omega^{-1}\otimes V_2,\quad \big(\omega^{-1}\otimes(\ide_{D^{\times}})^{\oplus 2}\big)\oplus (\NEW\otimes V_2),\quad \NEW\otimes (\ide_{D^{\times}})^{\oplus 2}.\] Using again \eqref{eq:JL-rho2}, we obtain the following short exact sequences
\begin{equation}\label{eq:S1-rho2-V2}
0\ra \omega^{-1}\otimes \JL(\brho_2)\ra \cS^1(\pi(\brho_2))\ra \NEW\otimes V_2\ra0,\end{equation}
\begin{equation}\label{eq:V2-S1rho2}0\ra \omega^{-1}\otimes V_2\ra \cS^1(\pi(\brho_2))\ra (\omega^{-1}\otimes \ide_{D^{\times}})^{\oplus 2}\oplus(\NEW\otimes V_2)\ra0.\end{equation}
Recall   the definition of $V_{\rm fd}$ for an admissible smooth $D^{\times}$-representation $V$ from Definition \ref{def:V-fd}, and that taking $(-)_{\rm fd}$ is right exact by Lemma \ref{lem:V-fd}.

\begin{corollary}\label{cor:S1-kappa-fd}
The following statements hold:

(i) $(\cS^1(\kappa))_{\rm fd}=0$ and $(\cS^1(\pi_{\alpha}))_{\rm fd}\cong \omega^{-1}\otimes\ide_{D^{\times}}$.

(ii) $(\cS^1(E_{\phi}))_{\rm fd}$ is $\omega^{-1}$-typic.

(iii) $(\cS^1(\tau_1))_{\rm fd}$ is $\omega^{-1}$-typic.

(iv) $(V_2)_{\rm fd}=0$.
\end{corollary}

\begin{proof}
(i) Since $\cS^1(\kappa)$ is a quotient of $\brho_1(-1)\otimes \JL(\brho_1)$ by \eqref{eq:S1-rho1} and $(\JL(\brho_1))^{\vee}$ is Cohen-Macaulay by Proposition \ref{prop:S-rho1-CM}, we have $\JL(\brho_1)_{\rm fd}=0$ by Lemma \ref{lem:CM-Vfd=0}, hence $(\cS^1(\kappa))_{\rm fd}=0$ as well by the right exactness of $(-)_{\rm fd}$.  The second assertion follows from this, by applying  Lemma \ref{lem:V-fd} to \eqref{eq:kappa-pi}.

(ii) Recall the exact sequence  \eqref{eq:S1-rho1-2}
 \[  0\ra \cS^1(E_{\phi})\ra \brho_1(-1)\otimes \JL(\brho_1)\ra \cS^1(\pi_{\alpha})\ra \omega^{-1}\otimes\ide_{D^{\times}}\ra0.
 \]
Since $\cS^1(\pi_{\alpha})$ is $\omega^{-1}$-typic by
Lemma \ref{lem:S1-kappa}, the morphism $\brho_1(-1)\otimes \JL(\brho_1)\ra \cS^1(\pi_{\alpha})$ factors through the quotient $\omega^{-1}\otimes \JL(\brho_1)$.
 Let $W$ be the  admissible $\F$-representation of $D^{\times}$ such that
 \[\omega^{-1}\otimes W=\Ker\big(\omega^{-1}\otimes \JL(\brho_1)\ra \cS^1(\pi_{\alpha})\big).\] Then one checks that $\cS^1(E_{\phi})$ fits in the following exact sequence
 \begin{equation}\label{eq:Sp-X}0\ra \NEW\otimes \JL(\brho_1)\ra \cS^1(E_{\phi})\ra \omega^{-1}\otimes  W\ra0.\end{equation}   Since $(\JL(\brho_1))_{\rm fd}=0$ as seen in (i), we deduce \[(\cS^1(E_{\phi}))_{\rm fd}= (\omega^{-1}\otimes W)_{\rm fd}\cong\omega^{-1}\otimes W_{\rm fd}.\] In particular, $(\cS^1(E_{\phi}))_{\rm fd}$ is $\omega^{-1}$-typic.

(iii) Note that there is a short exact sequence $0\ra E_{\phi}\ra \tau_1\ra \ide_{G}\ra0$ by definition of $\tau_1$, see  \eqref{eq:def-tau1}.   By Proposition \ref{prop:S2Sp=0}(i) it induces an exact sequence
\begin{equation}\label{eq:Ephi-tau1}0\ra \NEW\otimes\ide_{D^{\times}}\ra \cS^1(E_{\phi})\ra \cS^1(\tau_1)\ra0.\end{equation}
The assertion then follows from (ii) using Lemma \ref{lem:V-fd}.

(iv)   We view $\cS^1(\pi_{\alpha})$ as a subrepresentation of $\cS^1(\pi(\brho_2))$ via  \eqref{eq:S1-rho2}. Since $\cS^1(\pi_{\alpha})$ is $\omega^{-1}$-typic by Lemma \ref{lem:S1-kappa}, it is contained in $\omega^{-1}\otimes \JL(\brho_2)$, see \eqref{eq:S1-rho2-V2}.
As a consequence, the snake lemma applied to \eqref{eq:S1-rho2} and \eqref{eq:S1-rho2-V2} implies that $\NEW\otimes V_2$ is a quotient of $\cS^1(\tau_1)$, thus  $(\NEW\otimes V_2)_{\rm fd}$ is a quotient of $(\cS^1(\tau_1))_{\rm fd}$. However,  $(\cS^1(\tau_1))_{\rm fd}$ is $\omega^{-1}$-typic by (iii), which forces $(\NEW\otimes V_2)_{\rm fd}=0$ or equivalently $(V_2)_{\rm fd}=0$.
\end{proof}

\begin{corollary}\label{cor:S1kappa=V2}
We have   isomorphisms $\cS^1(\kappa)\cong \omega^{-1}\otimes V_2$ and
\[\cS^1(\tau_1)\cong(\omega^{-1}\otimes \ide_{D^{\times}})\oplus(\NEW\otimes V_2).\]
\end{corollary}
\begin{proof}

We may identify $\cS^1(\kappa)$ with a subrepresentation of $\cS^1(\pi(\brho_2))$ via  \eqref{eq:S1-rho2} and \eqref{eq:kappa-pi}. As in the proof of Corollary \ref{cor:S1-kappa-fd}(iv),   $\cS^1(\kappa)$ is contained in $\omega^{-1}\otimes \JL(\brho_2)$. However, since $(\cS^1(\kappa))_{\rm fd}=0$ by Corollary \ref{cor:S1-kappa-fd}(i), $\cS^1(\kappa)$ is in fact contained in $\omega^{-1}\otimes V_2$ by the definition of $V_2$, see \eqref{eq:def-V2}.
Denote by $\iota$ the inclusion \[\iota:\ \cS^1(\kappa)\hookrightarrow \omega^{-1}\otimes V_2. \]
We need to prove  that $\iota$ is an isomorphism, or equivalently $\Coker(\iota)=0$. Since $(V_2)_{\rm fd}=0$ by Corollary \ref{cor:S1-kappa-fd}(iv), it suffices to prove that $\Coker(\iota)$ is finite dimensional.

Denote by $\widetilde{\iota}$ the embedding $\cS^1(\kappa)\hookrightarrow \cS^1(\pi(\brho_2))$. Then \eqref{eq:S1-rho2} and \eqref{eq:kappa-pi} imply
\begin{equation}\label{eq:coker-tiota}0\ra \omega^{-1}\otimes\ide_{D^{\times}}\ra\Coker(\widetilde{\iota})\ra \cS^1(\tau_1)\ra0.\end{equation}
Since $\Hom_{G_{\Q_p}}(\omega^{-1},\cS^1(\tau_1))$ is finite dimensional by Lemma \ref{lem:1-S1Sp}, so is $\Hom_{G_{\Q_p}}(\omega^{-1},\Coker(\widetilde{\iota}))$.
On the other hand, using \eqref{eq:V2-S1rho2} we have a commutative diagram
\[\xymatrix{0\ar[r]&\cS^1(\kappa)
\ar^{\iota}[d]\ar^{\widetilde{\iota}\ \ }[r]&\cS^1(\pi(\brho_2))\ar[r]\ar@{=}[d]& \Coker(\widetilde{\iota})\ar[r]\ar[d]&0\\
0\ar[r]&\omega^{-1}\otimes V_2\ar[r]&\cS^1(\pi(\brho_2))\ar[r]&(\omega^{-1}\otimes \ide_{D^{\times}})^{\oplus 2}\oplus(\NEW\otimes V_2)\ar[r]&0}\]
hence an exact sequence
\begin{equation}\label{eq:coker-iota}0\ra \Coker(\iota)\ra\Coker(\widetilde{\iota})\ra (\omega^{-1}\otimes \ide_{D^{\times}})^{\oplus 2}\oplus(\NEW\otimes V_2)\ra0.\end{equation}
Consequently, $\Hom_{G_{\Q_p}}(\omega^{-1},\Coker(\iota))$ is finite dimensional. However, since $\Coker(\iota)$ is $\omega^{-1}$-typic (being a quotient of $\omega^{-1}\otimes V_2$), this implies that $\Coker(\iota)$ is itself finite dimensional. As explained in last paragraph, we deduce that $\iota$ is an isomorphism and consequently by \eqref{eq:coker-iota}
\[\Coker(\widetilde{\iota})\cong (\omega^{-1}\otimes \ide_{D^{\times}})^{\oplus 2}\oplus(\NEW\otimes V_2).\]
Finally, the second isomorphism in the corollary follows from this by using \eqref{eq:coker-tiota}.
\end{proof}

We note the following consequence of the proof of Corollary \ref{cor:S1kappa=V2}.
\begin{corollary}\label{cor:S1-pialpha}
(i) There exists a short exact sequence of $G_{\Q_p}\times D^{\times}$-representations
\[0\ra \omega^{-1}\otimes V_2\ra \cS^1(\pi_{\alpha})\ra \omega^{-1}\otimes \ide_{D^{\times}}\ra0.\]

(ii) There exists a $G_{\Q_p}\times D^{\times}$-equivariant surjection $\cS^1(\Sp)\twoheadrightarrow \omega^{-1}\otimes\ide_{D^{\times}}$ whose kernel admits only $\NEW$ as subquotients when restricted to $G_{\Q_p}$. In particular, $\Hom_{G_{\Q_p}}(\cS^1(\Sp),\omega^{-1})$ is $1$-dimensional.
\end{corollary}

\begin{proof}
(i) It follows from \eqref{eq:kappa-pi} and Corollary \ref{cor:S1kappa=V2}.

(ii) It follows from \eqref{eq:Sp-tau1} and Corollary \ref{cor:S1kappa=V2}.
\end{proof}

Recall from Propositions \ref{thm:Serre-D}, \ref{lemma--quat-Serre-wt} that we always have  $\ide_{\cO_D^{\times}}\in W_D(\brho_1)$ (no matter $\brho_1$ is peu ramifi\'e or tr\`es ramifi\'e), so
\[\Hom_{\cO_D^{\times}}(\ide_{\cO_D^{\times}},\JL(\brho_1))\neq0.\]
Let $W_1$ be the $\ide_{\cO_D^{\times}}$-typic component of $\soc_{\cO_D^{\times}}\JL(\brho_1)$. It is easy to see that $W_1$ is stable under $D^{\times}$.  Define $V_1$ to be the quotient
\begin{equation}\label{eq:def-V1}
V_1:=\JL(\brho_1)/W_1.\end{equation}

 The main result of this subsection is the  following.
\begin{theorem}\label{thm:nongeneric-V1=V2}
There exists a $D^{\times}$-equivariant isomorphism $V_1\cong V_2$.
\end{theorem}

\begin{proof}
Recall the exact sequence \eqref{eq:S1-rho1}
\[0\ra \NEW\otimes \ide_{D^{\times}}\ra \cS^1(\Sp)\ra \brho_1(-1)\otimes \JL(\brho_1)\overset{j}{\ra} \cS^1(\kappa)\ra0.\]
By Corollary \ref{cor:S1-pialpha}(ii),   $\Hom_{G_{\Q_p}}(\cS^1(\Sp),\omega^{-1})$ is $1$-dimensional over $\F$, so the last sequence shows that  $\Hom_{G_{\Q_p}}(\Ker(j),\omega^{-1})$ is also $1$-dimensional.
Since $\cS^1(\kappa)$ is $\omega^{-1}$-typic by Lemma \ref{lem:S1-kappa}, the surjection $j$ factors as
\[\brho_1(-1)\otimes \JL(\brho_1)\twoheadrightarrow \omega^{-1}\otimes  \JL(\brho_1)\overset{j'}{\twoheadrightarrow} \cS^1(\kappa).\]
We clearly have a short exact sequence
\[
0\ra \NEW\otimes \JL(\brho_1)\ra \Ker(j)\ra \Ker(j')\ra0,
\]
which implies that $\Hom_{G_{\Q_p}}(\Ker(j'),\omega^{-1})$ is also $1$-dimensional. Moreover, by Corollary \ref{cor:S1-pialpha}(ii) again, it is easy to see that the $1$-dimensional $\omega^{-1}$-typic quotient of $\Ker(j')$ is isomorphic to $\omega^{-1}\otimes \ide_{D^{\times}}$. But $\Ker(j')$ is itself $\omega^{-1}$-typic (being a subrepresentation of   $\omega^{-1}\otimes \JL(\brho_1)$), so $\Ker(j')$ is in fact isomorphic to $\omega^{-1}\otimes\ide_{D^{\times}}$.

On the other hand, since $\cS^1(\kappa)\cong V_2$ as representations of $D^{\times}$ by Corollary \ref{cor:S1kappa=V2}, we have
\begin{equation}\label{eq:soc-S1kappa}
\soc_{\cO_D^{\times}}\cS^1(\kappa)=\soc_{\cO_D^{\times}}V_2\subset \soc_{\cO_D^{\times}}\JL(\brho_2)\cong (\alpha\oplus\alpha^{-1})^{\oplus m_2},
\end{equation}
for some integer $m_2\geq 1,$ where the last isomorphism is given by Proposition \ref{thm:Serre-D}, Proposition \ref{lemma--quat-Serre-wt} and Corollary \ref{Cor-quat-Serre-wt}. Indeed, taking $r=p-3$ and $s=0$ in (ii-c) of Proposition \ref{thm:Serre-D}, we get $W_D(\brho_2)=\{\xi^{p-3}\alpha^{-1}\zeta,\xi^{p(p-3)}\alpha\zeta\}=\{\alpha,\alpha^{-1}\}$. We deduce that the composition
\[\omega^{-1} \otimes W_1\hookrightarrow \omega^{-1}\otimes \JL(\brho_1)\overset{j'}{\twoheadrightarrow} \cS^1(\kappa)\]
is zero,
where the first morphism is induced from the natural inclusion $W_1\hookrightarrow \JL(\brho_1)$, see \eqref{eq:def-V1}. In other words, $\Ker(j') $ contains $ \omega^{-1}\otimes W_1$. Combining with what has been proved in last paragraph, this implies
$\Ker(j')=\omega^{-1}\otimes W_1.$ 
In particular, $ W_1\cong\ide_{D^{\times}}$ and
\[\cS^1(\kappa)\cong (\omega^{-1}\otimes\JL(\brho_1))/(\omega^{-1}\otimes W_1)\overset{\eqref{eq:def-V1}}{=}\omega^{-1}\otimes V_1.\]
 Taking into account Corollary \ref{cor:S1kappa=V2}, we obtain
 \[V_1\cong\Hom_{G_{\Q_p}}(\omega^{-1},\cS^1(\kappa))\cong V_2\] as  representations of $D^{\times}$.
\end{proof}

\begin{lemma}\label{lem:Exti-JL}
Let $\chi:\cO_D^{\times}\ra \F^{\times}$ be a smooth character. 

(i) If  $\chi\notin W_D(\brho_1)$, then $\Ext^i_{\cO_D^{\times}/Z_D^1}(\chi,\JL(\brho_1))=0$ for $i\geq 0$.

(ii) If $\chi\notin W_D(\brho_2)$, then $\Hom_{\cO_D^{\times}}(\chi,\JL(\brho_2))=\Ext^1_{\cO_D^{\times}/Z_D^1}(\chi,\JL(\brho_2))=0$.
\end{lemma}
\begin{proof}
(i) The proof is as in \cite[Prop.~10.10(i)]{Hu-Wang}. The point is that $R_{\brho_1}^{\psi\e}$ is formally smooth, so  by Proposition \ref{prop:GN-flat} $\wt{H}^1 (U^{v} , \cO)_{\fm_{\overline{r}_1}}^{d}$ is flat over $\mathbb{T}(U^v)_{\fm_{\overline{r}_1}}$ with fiber isomorphic to $\JL(\brho_1)^{\vee}$. Here we write $\overline{r}_1$ instead of $\overline{r}$ to indicate that we are considering the case where $\overline{r}_v(1)=\brho_1$. 

(ii) The difference with (i) is that $R_{\brho_2}^{\psi\e}$ is \emph{not} formally smooth.   It is clear that $\Hom_{\cO_D^{\times}}(\chi,\JL(\brho_2))=0$ for $\chi\notin W_D(\brho_2)$. For the vanishing of $\Ext^1_{\cO_D^{\times}/Z_D^1}(\chi,\JL(\brho_2))$, choose  a set of generators $(f_1,\dots,f_m)$ of $\fm_{\overline{r}_2}$, then they induce an exact sequence (recall assumption (\textbf{H}))
\[0\ra \JL(\brho_2)\ra \Pi_2\ra \prod_{i=1}^m\Pi_2\]
where $\Pi_2:=\wt{H}^1 (U^{v} , \F)_{\fm_{\overline{r}_2}}$. 
Since $\Hom_{\cO_D^{\times}}(\chi,\Pi_2)=0$ and since $\Pi_2$ is an injective  representation of $\cO_D^{\times}/Z_D^1$  by  Theorem \ref{thm--K-injectivity}(ii), the result easily follows.  
\end{proof}

Thanks to Theorem \ref{thm:nongeneric-V1=V2}, we  write $V$ for $V_1$ and $V_2$ from now on.

\begin{corollary}\label{cor:JL-rho1-V}
There exists a short exact sequence
\begin{equation}\label{eq:JL-rho1-V}0\ra \ide_{D^{\times}}\ra \JL(\brho_1)\ra V\ra0.\end{equation}
\end{corollary}
\begin{proof}
This is a direct consequence of the proof of Theorem \ref{thm:nongeneric-V1=V2}.
\end{proof}

\begin{corollary}\label{cor:socV}
The following statements hold:

(i) $V_{\rm fd}=0$ and
$\rsoc_{D^{\times}}(V)\cong \Ind_{\cO_D^{\times}Z_D}^{D^{\times}} \alpha.$

(ii) $\dim_{\F}\Ext^1_{D^{\times}/Z_D}(\Ind_{\cO_D^{\times}Z_D}^{D^{\times}} \alpha,V)=1$.

(iii) $\dim_{\F}\Ext^1_{D^{\times}/Z_D}(\ide_{D^{\times}},V)=2$.
\end{corollary}

\begin{proof}
(i) The first assertion is just Corollary \ref{cor:S1-kappa-fd}(iv). For the second assertion, by Frobenius reciprocity it suffices to show   $\Hom_{\cO_D^{\times}}(\alpha,V)$ has dimension $1$.  We take $\brho_1$ to be tr\`es ramifi\'e. 
By Theorem \ref{thm:Serre-D}, we have $\chi\in W_D(\brho_1)$ if and only if $\chi = \ide_{\cO_D^{\times}}$, so that $\Hom_{\cO_D^{\times}}(\alpha,\JL(\brho_1))=0$. Hence, by applying $\Hom_{\cO_D^{\times}}(\alpha,-)$ to \eqref{eq:JL-rho1-V}  we obtain a long exact sequence
  \begin{multline}\label{eq:Ext1-V}0\ra \Hom(\alpha,V)\ra \Ext^1(\alpha,\ide_{\cO_D^{\times}})\ra \Ext^1(\alpha,\JL(\brho_1))\ra \Ext^1(\alpha,V)\\ \To{\partial} \Ext^2(\alpha,\ide_{\cO_D^{\times}})\ra \Ext^2(\alpha,\JL(\brho_1)),
    \end{multline}  
where $\Ext^i$ means $\Ext^i_{\cO_D^{\times}/Z_D^1}$. Since $\alpha\notin W_D(\brho_1)$, see Theorem \ref{thm:Serre-D}(ii-a), we have $\Ext^1(\alpha,\JL(\brho_1))=0$ by Lemma \ref{lem:Exti-JL}(i). The result follows as $\dim_{\F}\Ext^1(\alpha,\ide_{\cO_D^{\times}})=1$ by Proposition \ref{prop-Ext1-U1}.

(ii) By Frobenius reciprocity, it is equivalent to proving $\dim_{\F}\Ext^1_{\cO_D^{\times}/Z_D^1}(\alpha,V)=1$. Since $\alpha\notin W_D(\brho_1)$ (again we take $\brho_1$ to be tr\`es ramifi\'e),  Lemma \ref{lem:Exti-JL}(i) implies that the map  $\partial$ in \eqref{eq:Ext1-V} is an isomorphism.   On the other hand, using Propositions \ref{prop-Ext1-U1} and  \ref{prop-Poincare} we know that $\dim_{\F}\Ext^2_{\cO_D^{\times}/Z_D^1}(\alpha,\ide_{\cO_D^{\times}})=1$, from which the assertion follows.

(iii) Note that $\ide_{\cO_D^{\times}}\notin W_D(\brho_2)$, see \eqref{eq:soc-S1kappa}. Using Lemma \ref{lem:Exti-JL}(ii)  this implies  $$\Ext^i_{\cO_D^{\times}/Z_D^1}(\ide_{\cO_D^{\times}},\JL(\brho_2))=0,$$ hence by Frobenius reciprocity $\Ext^i_{D^{\times}/Z_D}\big(\Ind_{\cO_D^{\times}Z_D}^{D^{\times}}\ide,\JL(\brho_2)\big)=0$ for $i=0,1$. Since $\ide_{D^{\times}}$ is a direct summand of $\Ind_{\cO_D^{\times}Z_D}^{D^{\times}}\ide$ as $[D^{\times}:\cO_D^{\times}Z_D]=2$ and $p>2$, we deduce 
\[\Hom_{D^{\times}}(\ide_{D^{\times}},\JL(\brho_2))=\Ext^1_{D^{\times}/Z_D}(\ide_{D^{\times}},\JL(\brho_2))=0.\] 
 
Now, applying $\Hom_{D^{\times}}(\ide_{D^{\times}},-)$ to $0\ra V\ra \JL(\brho_2)\ra(\ide_{D^{\times}})^{\oplus 2}\ra 0$ gives the result.
\end{proof}
 
Together with Theorem \ref{thm:Serre-D}, we deduce the $D^{\times}$-socle of $\JL(\brho_i)$.
\begin{corollary}\label{cor:socle-JL}
(i) If $\brho_1$ is peu ramifi\'e, then $\soc_{D^{\times}}\JL(\brho_1)\cong\ide_{D^{\times}}\oplus \Ind_{\cO_D^{\times}Z_D}^{D^{\times}}\alpha$;
if $\brho_1$ is tr\`es ramifi\'e, then $\soc_{D^{\times}}\JL(\brho_1)\cong \ide_{D^{\times}}$.

(ii) $\soc_{D^{\times}}\JL(\brho_2)\cong \Ind_{\cO_D^{\times}Z_D}^{D^{\times}}\alpha$.
\end{corollary}
\begin{remark}
Unlike the generic case treated in \S\ref{ss:S-generic}, we see that $\JL(\brho_1)$ detects the extension type of $\brho_1$. 

We also deduce from Corollary \ref{cor:socle-JL} that $\soc_{\cO_D^{\times}}\JL(\brho_1)$ and  $\soc_{\cO_D^{\times}}\JL(\brho_2)$ are multiplicity free. This corresponds to the fact that $\soc_{K}\pi(\brho_1)$  and $\soc_{K}\pi(\brho_2)$ are multiplicity free, and seems to be a non-trivial fact. 
\end{remark}
\begin{remark}
We can show that the kernel of $\cS^1(\Sp)\twoheadrightarrow \omega^{-1}\otimes\ide_{D^{\times}}$ in Corollary \ref{cor:S1-pialpha}(ii), which we denote by $U$, is $\ide_{G_{\Q_p}}$-typic, i.e. it does not admit  self-extensions of $\ide_{G_{\Q_p}}$ as subquotients   when restricted to $G_{\Q_p}$. Indeed, take $\brho_1$ to be tr\`es ramifi\'e and $\brho_1'$ to be peu ramifi\'e, we obtain two embeddings  
\[i,i':\ \ide_{G_{\Q_p}}\otimes \ide_{D^{\times}}\hookrightarrow\cS^1(\Sp)\]
from \eqref{eq:S1-rho1}. One checks that 
\[0\ra \mathrm{Im}(i)\ra U\ra \ide_{G_{\Q_p}}\otimes\JL(\brho_1)\ra0\]
\[0\ra \mathrm{Im}(i')\ra U\ra \ide_{G_{\Q_p}}\otimes \JL(\brho_1')\ra0.\] As a consequence,  $\mathrm{Im}(i)\neq \mathrm{Im}(i')$  because $\JL(\brho_1)$ and $\JL(\brho_1') $ are non-isomorphic by Corollary \ref{cor:socle-JL}(i).   It is then easy to deduce that $U$ is isomorphic to $\ide_{G_{\Q_p}}\otimes (\JL(\brho_1)\times_{V}\JL(\brho_1'))$, where the fibered product is taken with respect to \eqref{eq:JL-rho1-V}.
\end{remark}

\subsubsection{Summary}
We summarize the results proved above in the following theorem.

\begin{theorem}\label{thm:JL-nongeneric}
(i) $\cS^0(\ide_G)=\NEW\otimes\ide_{D^{\times}}$, $\cS^1(\ide_G)=0$, $\cS^2(\ide_G)=\omega^{-1}\otimes\ide_{D^{\times}}$.

(ii) $\cS^0(\Sp)=\cS^2(\Sp)=0$, and there exists a short exact sequence
\[0\ra (\NEW\otimes\ide_{D^{\times}})^{\oplus 2}\ra \cS^1(\Sp)
\ra (\NEW\otimes V)\oplus (\omega^{-1}\otimes\ide_{D^{\times}})\ra0.\]

(iii) $\cS^0(\pi_{\alpha})=\cS^2(\pi_{\alpha})=0$ and there exists a short exact sequence 
\[0\ra \omega^{-1}\otimes V\ra \cS^1(\pi_{\alpha})\ra \omega^{-1}\otimes \ide_{D^{\times}}\ra0.\]

(iv) There exist exact sequences
\[0\ra \ide_{D^{\times}}\ra \JL(\brho_1)\ra V\ra0,\]
and
\[0\ra V\ra \JL(\brho_2)\ra (\ide_{D^{\times}})^{\oplus 2}\ra0.\]
Moreover, $\JL(\brho_2)$ is isomorphic to the universal extension of $(\ide_{D^{\times}})^{\oplus 2}$ by $V$.
\end{theorem}

\subsection{The non-minimal case}\label{ss:S-non-minimal}

We briefly explain  how to modify the arguments in \S\ref{ss:S-generic} and \S\ref{ss:S-nongeneric}  to handle the non-minimal case (i.e. $d\neq 1$ in Theorem \ref{thm:lg}).

Let $\brho=\overline{r}_v(1)$ and assume  $\End_{G_{\Q_p}}(\brho)=\F$; in particular, $\brho$ is allowed to be irreducible. We put
\begin{equation}
\label{eq:def-JL-nonmininal}
\JL(\brho):=\left\{\begin{array}{lll} \Hom_{G_{\Q_p}}\big(\chi\omega^{-1},\cS^1(\pi(\brho))\big) &\mathrm{if}\ \brho\sim \smatr{\chi}*0{\chi\omega} \medskip\\
\Hom_{G_{\Q_p}}\big(\brho(-1),\cS^1(\pi(\brho))\big) &\mathrm{otherwise}.
\end{array}\right.
\end{equation}

\begin{proposition}\label{prop:non-minimal}
(i) If $\brho$ is not of the form $\smatr{\chi}*0{\chi\omega}$, then  
\[\wt{H}^1 (U^{v},\F)[\fm_{\overline{r}}]\cong \left(\brho(-1)\otimes\JL(\brho) \right)^{\oplus d}\]
and $\cS^1(\pi(\brho))\cong \brho(-1)\otimes \JL(\brho)$.

(ii) If $\brho\sim \smatr{1}*0{\omega}$,  then
\[\wt{H}^1 (U^{v},\F)[\fm_{\overline{r}}]\cong \left(\brho(-1)\otimes\JL(\brho) \right)^{\oplus d},\]
 and there exists a short exact sequence
\[0\ra \cS^1(\pi(\brho))\To{f} \brho(-1)\otimes\JL(\brho)\ra (\NEW\otimes\ide_{D^{\times}})^{\oplus 2}\ra0.\]
\end{proposition}
\begin{proof}
(i)  We claim that $\cS^1\big(\wt{S}(U^{v}, \F)[\fm_{\overline{r}}]\big)=\wt{H}^1 (U^{v},\F)[\fm_{\overline{r}}]$. If $\brho^{\rm ss}\nsim \chi\oplus \chi\omega$, it is proved in \cite[Lem.~6.1]{Paskunas-JL}. If $\brho^{\rm ss}\sim \chi\oplus\chi\omega$, then the assumption on $\brho$  implies that $R_{\brho}^{\psi\varepsilon^{-1}}$ is formally smooth, so we may apply Proposition \ref{prop:GN-flat} (using Corollary \ref{cor:GKdim-comp}).

The claim implies that $\cS^1(\wt{S}(U^{v} , \F)[\fm_{\overline{r}}])$ is $\brho(-1)$-typic. Since $\wt{S} (U^{v} , \F)[\fm_{\overline{r}}]\cong\pi(\brho)^{\oplus d}$ by Theorem \ref{thm:lg}, $\cS^1(\pi(\brho))$ is also $\brho(-1)$-typic  by Lemma \ref{lem:typic}(ii). The result easily follows.

(ii) First, the proof of Proposition \ref{prop:S1-rho2} shows that
\[0\ra \cS^1(\wt{S}(U^{v} , \F)[\fm_{\overline{r}}])\To{\phi} \wt{H}^1  (U^{v} , \F)[\fm_{\overline{r}}]\ra (\NEW\otimes\ide_{D^{\times}})^{\oplus 2d}\ra0\]
which implies
\[\begin{array}{rll}\Hom_{G_{\Q_p}}\big(\omega^{-1},\cS^1(\wt{S} (U^{v} , \F)[\fm_{\overline{r}}])\big)&\overset{\phi^*}{\simto}&\Hom_{G_{\Q_p}}\big(\omega^{-1}, \wt{H}^1 (U^{v} , \F)[\fm_{\overline{r}}]\big)\\
&\simeq&\Hom_{G_{\Q_p}}\big(\brho(-1),\wt{H}^1 (U^{v} , \F)[\fm_{\overline{r}}]\big),\end{array}\]
where the second isomorphism holds because $\wt{H}^1(U^{v},\F)[\fm_{\overline{r}}]$ is $\brho(-1)$-typic.
Choose an isomorphism $\iota: \pi(\brho)^{\oplus d}\simto \wt{S}(U^{v} , \F)[\fm_{\overline{r}}]$; it induces an isomorphism
\[\JL(\brho)^{\oplus d}\overset{\iota^*}{\simto} \Hom_{G_{\Q_p}}\big(\omega^{-1},\cS^1(\wt{S} (U^{v} , \F)[\fm_{\overline{r}}])\big).\]
Thus we get an isomorphism
\begin{equation}\label{eq:non-minimal}\wt{H}^1 (U^{v} , \F)[\fm_{\overline{r}}]\cong \brho(-1)\otimes\Hom_{G_{\Q_p}}(\brho(-1),\wt{H}^1(U^{v},\F)[\fm_{\overline{r}}])\overset{(\phi^*\circ\iota^*)^{-1}}{\simto} \brho(-1)\otimes \JL(\brho)^{\oplus d}\end{equation}
as desired. Let $f'$ be the composite map
\[\cS^1(\pi(\brho))^{\oplus d}\overset{\iota}{\simto} \cS^1(\wt{S} (U^{v} , \F)[\fm_{\overline{r}}])\To{\phi} \wt{H}^1(U^{v} , \F)[\fm_{\overline{r}}]\overset{\eqref{eq:non-minimal}}{\simto} \brho(-1)\otimes \JL(\brho)^{\oplus d}.\]

 Since $\cS^1(\pi(\brho))$ is contained in $\wt{H}^1(U^{v} , \F)[\fm_{\overline{r}}]$ which is $\brho(-1)$-typic, we may apply Lemma \ref{lem:non-minimal} below to obtain an embedding
 \[0\ra \cS^1(\pi(\brho))\To{f} \brho(-1)\otimes \JL(\brho),\]
 extending the natural embedding $\omega^{-1}\otimes \JL(\brho)\hookrightarrow \brho(-1)\otimes\JL(\brho)$.
Moreover,  $f$ is $G_{\Q_{p}}\times D^{\times}$-equivariant by construction. We are left to show $\mathrm{Coker}(f)\cong (\NEW\otimes\ide_{D^{\times}})^{\oplus 2}$. It is clear that $f^{\oplus d}$ and $f'$ coincide when restricted to $\omega^{-1}\otimes \JL(\brho)^{\oplus d}$, so $f'=f^{\oplus d}$ by the uniqueness part of Lemma \ref{lem:non-minimal}. Since $\Coker(f')\cong (\NEW\otimes\ide_{D^{\times}})^{\oplus 2d}$, we obtain $\mathrm{Coker}(f)\cong(\NEW\otimes\ide_{D^{\times}})^{\oplus 2}$ as required.
\end{proof}

\begin{lemma}\label{lem:non-minimal}
Let $\brho\sim \smatr{\chi_1}{*}0{\chi_2}$ with $\End_{G_{\Q_p}}(\brho)\cong\F$. 
If $M$ is a $\brho$-typic $\F[G_{\Q_p}]$-module, then for any submodule    $M'\subset M$ there exists a unique embedding
\[0\ra M'\ra \brho\otimes \Hom_{G_{\Q_p}}(\chi_1,M')\]
extending the embedding $\chi_1\otimes  \Hom_{G_{\Q_p}}(\chi_1,M')\hookrightarrow \brho\otimes \Hom_{G_{\Q_p}}(\chi_1,M')$ induced from $\chi_1\hookrightarrow \brho$.
\end{lemma}
\begin{proof}
Since $M$ is $\brho$-typic, it is naturally isomorphic to $\brho\otimes M_0$ where $M_0:=\Hom_{G_{\Q_p}}(\brho,M)$. Actually, the assumption on $\brho$ implies that $M_0= \Hom_{G_{\Q_p}}(\chi_1,M)$.
Writing   $M'_0:=\Hom_{G_{\Q_p}}(\chi_1,M')$, we claim that $M'$ is contained in $\brho\otimes M_0'$, both regarded  as subspaces of $ \brho\otimes M_0$. Indeed, letting $\wt{M}':=M'+\brho\otimes M_0'$, we need to prove $\wt{M}'=\brho\otimes M_0'$. It is clear that $\chi_1\otimes M_0'$ is identified with $M'\cap (\chi_1\otimes M_0)$, thus $M'/(\chi_1\otimes M_0')$ embeds in $\chi_2\otimes M_0$ and is $\chi_2$-typic. Using the natural isomorphism 
$\wt{M}'/(\brho\otimes M_0')\cong M'/(M'\cap (\brho\otimes M_0')),$ 
we see that $\wt{M}'/(\brho\otimes M_0')$ is a quotient of $M'/(\chi_1\otimes M_0')$, thus \[\Hom_{G_{\Q_p}}(\chi_1,\wt{M}'/(\brho\otimes M_0'))=0.\] 
On the other hand, if $\wt{M}'/(\brho\otimes M_0')$ is nonzero, then it embeds  in $\brho\otimes(M_0/M_0')$ and we must have $\Hom_{G_{\Q_p}}(\chi_1,\wt{M}'/(\brho\otimes M_0'))\neq 0$,  contradiction.  

The claim implies that the given inclusion $M'\subset M$ provides an embedding required in the lemma,  so we are left to prove the uniqueness.

Consider the exact sequence
\[0\ra \chi_1\otimes M_0'\ra M'\ra Q\ra0\]
with $Q$ being the quotient. As seen above, $Q$ is $\chi_2$-typic. Applying $\Hom_{G_{\Q_p}}(-,\brho\otimes M_0')$ to it, we obtain an exact sequence
\begin{equation}\label{eq:nmin-gamma}
0\ra \Hom_{G_{\Q_p}}(Q,\brho\otimes M_0')\ra \Hom_{G_{\Q_p}}(M',\brho\otimes M_0')\To{\gamma} \Hom_{G_{\Q_p}}(\chi_1\otimes M_0',\brho\otimes M_0').\end{equation}
The result follows because $\Hom_{G_{\Q_p}}(Q,\brho\otimes M_0')=0$ (as $Q$ is $\chi_2$-typic). 
\end{proof}

Using Proposition \ref{prop:non-minimal}, the  arguments in \S\ref{ss:S-generic} and \S\ref{ss:S-nongeneric} when $\brho$ is reducible with $\End_{G_{\Q_p}}(\brho)=\F$, taking into account multiplicities everywhere, go through and  give similar results in the non-minimal case as in Theorems \ref{thm:JL-generic},   \ref{thm:nongeneric-V1=V2} and  \ref{thm:JL-nongeneric}.

\begin{remark}
If $\brho_0=\chi_1\oplus \chi_2$ with $\chi_1\chi_2^{-1}\neq \ide,\omega^{\pm1}$, we put
\[\JL(\brho_0):=\Hom_{G_{\Q_p}}\big(\chi_1\omega^{-1},\cS^1(\pi(\brho))\big).\] 
Combining with Proposition \ref{prop:non-minimal}, the proof of Theorem \ref{thm:JL-generic} shows   $\cS^1(\pi(\brho_0))=\brho_0(-1)\otimes\JL(\brho_0)$. 
\end{remark}

\newcommand{\etalchar}[1]{$^{#1}$}
\providecommand{\bysame}{\leavevmode\hbox to3em{\hrulefill}\thinspace}
\providecommand{\MR}{\relax\ifhmode\unskip\space\fi MR }
\providecommand{\MRhref}[2]{%
  \href{http://www.ams.org/mathscinet-getitem?mr=#1}{#2}
}
\providecommand{\href}[2]{#2}

\bigskip

\bigskip

\noindent Morningside Center of Mathematics, Academy of Mathematics and Systems Science,
 Chinese Academy of Sciences, University of the Chinese Academy of Sciences
Beijing, 100190,
China.\\
{\it E-mail:} {\ttfamily yhu@amss.ac.cn}\\

\noindent  Academy for Multidisciplinary Studies, Capital Normal University, Beijing, 100048\\
{\it E-mail:} {\ttfamily haoran@cnu.edu.cn}\\

 \end{document}